\documentclass[oneside, 16pt]{amsart}
\usepackage[utf8]{inputenc}

\usepackage{wrapfig}
\usepackage{PersCMDS}

\usepackage{amssymb}
\usepackage{mathtools}
\usepackage{ mathrsfs }
\usepackage{tikz-cd}

\usepackage[hidelinks]{hyperref}
\usepackage{enumitem}

\usepackage[style=alphabetic]{biblatex}

\usepackage{subfiles}
\usepackage[]{rotating}
\usepackage{xcolor}

\newtheorem{theorem}{Theorem}[section]
\newtheorem*{theorem*}{Theorem}
\newtheorem{corollary}[theorem]{Corollary}
\newtheorem{proposition}[theorem]{Proposition}
\newtheorem{lemma}[theorem]{Lemma}
\theoremstyle{definition}
\newtheorem{definition}[theorem]{Definition}
\newtheorem{remark}[theorem]{Remark}

\newcommand{\Perv}[2]{\text{Perv}(#1)}
\renewcommand{\derCat}[3]{\text{D}^b_{#1}(#2)}
\renewcommand{\perv}[1]{\text{Perv}(#1)}

\newcommand{\consder}[1]{\text{D}^b_c(#1)}

\title{Spreading out perverse sheaves}
\author[]{Beat Zurbuchen}
\newif\draft

\begin{document}
	\begin{abstract} We spread out geometrically irreducible perverse sheaves. Over an arithmetic base scheme, such sheaves extend from the generic fiber to a relatively perverse, universally locally acyclic sheaf over a dense open of the base scheme. As an application, we prove the arithmetic Kashiwara conjecture by Esnault and Kerz for geometric traits in equicharacteristic and generalize the decomposition theorem for arithmetic complexes to fields finitely generated over a separably closed field. We also recover the Hard Lefschetz theorem for arithmetic perverse sheaves. 
	\end{abstract}
	\maketitle
	
\section*{Introduction} Specialization arguments involving perverse sheaves are often based on a certain equivalence of categories introduced in \cite[{}{6.1}]{BBD}. This equivalence of categories extends sheaves from the generic fiber of a scheme over a trait to the scheme itself. 

In general, one does not expect these extensions to globalize to more general bases because there are sheaves which are not generically unramified. The main observation of this paper is that, if the perverse sheaf on the generic fiber is geometrically irreducible, then it becomes generically unramified after passing to an etale neighborhood of the base  and twisting by a rank one sheaf (coming from the etale neighborhood). An arithmetic character is always generically unramified by the quasi-unipotent monodromy theorem. This allows us to extend geometrically irreducible perverse sheaves from the generic fiber to an integral arithmetic base scheme directly:
\begin{theorem*}Let $\ell$ be a prime. Let $S_0$ be an integral scheme of finite type over $\BZ[1/\ell]$ with function field $k$. Let $X_0 \rightarrow S_0$ be a finite type morphism with generic fiber $X_k$. Let $M \in \Perv{X_k}{\Qbarl}$ be a geometrically irreducible perverse sheaf with coefficients in $\Qbarl$. There exists a dense open subscheme $U_0 \subseteq S_0$ and $M_0 \in \text{Perv}^{\text{ULA}}(X_0/U_0)$ such that the restriction of $M_0$ to the generic fiber is isomorphic to $M$. If $U_0$ is geometrically unibranch, the extension to $U_0$ is unique up to unique isomorphism. 
	\end{theorem*}
	When the base scheme is regular but not arithmetic, we can obtain an extension of a twist over a dominant etale neighborhood, see Theorem \ref{extension_THM}. We deduce the theorem from a variant for geometrically irreducible lisse sheaves, see Theorem \ref{extensionthmWeak_THM} for the arithmetic case, and Theorem \ref{simpson_THM} for the general case. We remark that similar extension have been considered in the literature before. In \cite[Prop.~6.1]{petrov23}, the authors construct extensions when $S_0 = \BZ[1/\ell]$ and the sheaf on the generic fiber is lisse. In \cite{zockmasterthesis}, the author considers such extensions for lisse sheaves on scheme which are smooth over a curve.
	
	The goal of this article is to prove the above extension theorem and to provide the theory required to apply this theorem. The resulting theory is softer in some respects than \cite{BBD}: after spreading out a perverse sheaf over a base scheme, the specialization of the perverse sheaf is provided by taking a fiber over a point of the base scheme. Loc. cit. provides a specialization via an equivalence of categories. This equivalence first spreads out the complex over a base trait and then takes the fiber at the closed point.
	
	Nevertheless, the extension theorem yields essentially the same conclusions as the specialization techniques of loc. cit. In some cases, it even allows one to reach conclusions that do not seem to be accessible by the methods of loc. cit. For example, the extension theorem implies the decomposition theorem over fields finitely generated over a separably closed field. This theorem is stated later in the introduction.
	
	We describe the main argument relevant for our applications. Consider a geometrically irreducible perverse sheaf on a scheme of finite type over an arithmetic field $k$ (i.e. a field finitely generated over its prime field). The above theorem allows us to spread out the perverse sheaf to an arithmetic scheme base scheme, i.e. a regular connected scheme of finite type over $\BZ[1/\ell]$ with function field $k$. After shrinking the base scheme, the fiber of the spreading of the perverse sheaf is geometrically irreducible (see Theorem \ref{filtration_THM}) over all closed points. Lafforgue's theorem (see \cite[Cor.~VII.8]{LafforguePurity}) then tells us that the fiber of the perverse sheaf over a closed point is pure (up to a twist). The applications of the extension theorem provided in this paper all follow this argument.
	
	Kashiwara's conjectures \cite[(C1),~(C2),~(C3)]{kashiwara98conjectures} for constructible mixed complexes on a scheme over a finite field are consequences of the theory of weights. Thanks to the purity afforded by Lafforgue's theorem, these statements can be generalized to all geometrically semisimple complexes over schemes over finite fields. With the specialization method of \cite{BBD}, Kashiwara's conjectures over finite fields can be amplified (see \cite[Rem.~3.3.4,~Prop.~3.3.10]{zhensun16paritysymmetry}) to prove the Hard Lefschetz theorem and the decomposition theorem for arithmetic complexes (see Definition \ref{arithmeticcomplex_DEF}) over arbitrary fields.
	
	The secondary goal of this paper is to revisit Kashiwara's conjectures for arithmetic complexes with the extension theorem in hand. This is then the main topic of our applications. 
	
	Before we describe the applications in more detail, we fix the following conventions: the letter $\ell$ denotes a prime. If $X$ is a scheme, then we require that $\ell$ is invertible on $X$. When $X$ is qcqs, the category $\consder{X}$ refers to the derived category of constructible complexes with coefficients in $\Qbarl$ (in the proetale topology). The notations and conventions are fixed more precisely in a separate section at the end of the introduction.  
\subsection*{Kashiwara's Conjectures} We can apply the extensions to nearby cycles of geometrically semisimple arithmetic perverse sheaves. When the perverse sheaf descends to an arithmetic field, i.e. a field finitely generated over a prime field, the above theorem allows us to spread out nearby cycles over an arithmetic base scheme. Once we have spread out the nearby cycles to an arithmetic base scheme, we can specialize to a finite field. There, we are able to apply the theory of weights to make statements about the nearby cycles. 
	
Over a finite field, Lafforgue's theorem together with a theorem of Gabber show that the nearby cycles are monodromy pure. This idea, together with certain descent results provided in Section \ref{descent_SEC}, allows us to prove the arithmetic Kashiwara conjecture (see \cite[Conj.~9.7]{esnault2025semistablelefschetzpencils}) in equicharacteristic (see Theorem \ref{kashiwara_THM}) for traits of ``geometric origin":
	\begin{theorem*}
		Let $S$ be the spectrum of the strict henselization of a discrete valuation ring essentially of finite type over a field. Let $X$ be a scheme of finite type over $S$ and $M \in \perv{X_{{\eta}}}$ a geometrically semisimple arithmetic perverse sheaf on the generic fiber.
		\begin{enumerate}
			\item The graded pieces of the monodromy filtration $\text{Gr}^M_n(\Psi^{\text{uni}}(M))$ are geometrically semisimple for all $n \in \BZ$.
			\item Let $f\colon X \rightarrow Y$ be a proper morphism of finite type $S$-schemes. The monodromy spectral sequence
			\[
			E_1^{pq} := \pervCoh{p + q}{f_*(\text{Gr}_{-q}^M(\Psi^{\text{uni}}(M)))} \Rightarrow \pervCoh{p + q}{f_*(\Psi^{\text{uni}}(M))}
			\]
			degenerates on the second page. Moreover, the filtration induced by the monodromy spectral sequence agrees with the monodromy filtration. 
		\end{enumerate}  
	\end{theorem*}	
	Another classical application of such specialization arguments is the decomposition theorem for arithmetic complexes (see \cite{DrinfeldKashiwara}). Using the spreading out technique, we also recover the decomposition theorem for arithmetic complexes by specializing to the known case over finite fields via the above models (see Theorem \ref{decomposition1_THM}). In fact, the technique of spreading out perverse sheaves allows us to obtain the following generalization of the decomposition theorem to fields finitely generated over a separably closed field (see Theorem \ref{decomposition2_THM}): 
	\begin{theorem*}\label{decomposition_THM}
		Let $k$ be a field finitely generated over a separably closed field. We consider a proper morphism $f\colon X \rightarrow Y$ of finite type $k$-schemes $X$ and $Y$. Let $K \in \derCat{c}{X}{\Qbarl}$ be an arithmetic semisimple complex on $X$.
		The pushforward $f_*(K)$ is semisimple.
	\end{theorem*}
	The above theorem is only stated for arithmetic sheaves. We expect that the statement could be extended by our method to complexes over a scheme in characteristic zero without having to assume $K$ arithmetic by appealing to the decomposition theorem for semisimple complexes over an algebraically closed field of characteristic zero. We can also utilize a specialization argument to recover the Hard Lefschetz theorem for geometrically semisimple arithmetic perverse sheaves (see Theorem \ref{hardlefschetz_THM}):
	\begin{theorem*}
		Let $k$ be a field and $f\colon X \rightarrow Y$ a projective morphism of finite type $k$-schemes $X$ and $Y$. Let $\eta \in H^2(X, \Qbarl(1))$ be the Chern class of a $f$-ample line bundle on $X$ and $M \in \perv{X}$ a geometrically semisimple arithmetic perverse sheaf. The morphism
		\[
		\eta^i\colon \pervCoh{-i}{f_*(M)} \rightarrow \pervCoh{i}{f_*(M)}(i)
		\]
		is an isomorphism for all $i \geq 0$. 
	\end{theorem*}
\subsection*{Generic base change} Before we give an overview of the various sections of this paper, we would like to point out the generic base change theorems proven in the appendix. This theorem is essential to our specialization arguments and is of independent interest. Theorem \ref{genbc_THM}, Theorem \ref{genula_THM}, and Theorem \ref{genbchom_THM} may be summarized by: 
\begin{theorem*}
	Let $f\colon X \rightarrow Y$ be a finitely presented morphism of qcqs schemes. We consider one of the following two coefficient rings $R$:
	\begin{enumerate}
		\item The ring $R$ is a complete Noetherian local ring. We assume the characteristic of $R/\idm$ is invertible on $Y$.
		\item The ring $R$ is $R = \Qbarl$ for a prime $\ell$ invertible on $Y$.
	\end{enumerate}
	Let $K, L \in \text{D}_c^b(X, R)$ be constructible complexes of $R$-modules. There exists a finite constructible partition
	\[
	Y = \bigsqcup_{i = 1}^n Y_i
	\]
	such that all $1 \leq i \leq n$ satisfy: 
	\begin{enumerate}
		\item The pushforward $f_*(K|_{Y_i})$ is constructible over $Y_i$.
		\item The formation of the pushforward $f_*(K|_{Y_i})$ commutes with base change $Y' \rightarrow Y_i$.
		\item The complex $K|_{Y_i}$ is universally locally acyclic over $Y_{i}$.
		\item There exists an internal mapping object $\text{Hom}(K|_{Y_i}, L|_{Y_i}) \in \text{D}^b_c(X_{Y_i}, R)$.
		\item The formation of the internal mapping object $\text{Hom}(K|_{Y_i}, L|_{Y_i})$ commutes with base change $Y' \rightarrow Y_i$. 
	\end{enumerate}
\end{theorem*}
For items (1) - (3), we remark that Theorem \ref{genbc_THM} and Theorem \ref{genula_THM} apply to complete Noetherian $I$-adic coefficient rings $R$ when the characteristic of $R/I$ is invertible on $Y$. However, for items (4) and (5), we can not prove the above theorem in this generality. The reason is that there are certain technical points occuring during the proof we were unable to resolve. We refer the reader to the corresponding theorems in the appendix for the precise definitions.
	
\subsection*{Overview} We summarize the various sections here for overview. We also point out results which may be of independent interest.

In section 1, we treat arithmetic sheaves and decomposition series of perverse sheaves. We begin by proving a descent lemma for lisse sheaves in Lemma \ref{descentlocsys_LEM}. The descent lemma allows one to descend lisse sheaves so that the geometric and the arithmetic monodromy group agree. While this is not used in this paper, we hope that it may be useful in other contexts. 

With this descent lemma, we can descend semisimple arithmetic complexes on schemes over separably closed fields to arithmetic fields in Corollary \ref{models_COR}. In Theorem \ref{geomsimpleperverse_THM}, we decompose an arithmetic geometrically semisimple perverse sheaf on a scheme over a field $k$ into arithmetic geometrically irreducible perverse sheaves.

In Theorem \ref{filtration_THM}, we construct a certain relative decomposition series of a relatively perverse sheaf over a geometrically unibranch topologically Noetherian base scheme. This relative decomposition series allows one to compare the decomposition series across fibers in a family. 

In Section 2, we construct the models of perverse sheaves alluded to in the introduction. We begin by proving that it is sufficient to construct an extension ``on a section" in Theorem \ref{extensionpre_THM}. In Theorem \ref{extensionthmWeak_THM}, we apply Theorem \ref{extensionpre_THM} to obtain the desired extensions for lisse sheaves. The proof of each of these statements is sketched just after the statement of the theorem. In Theorem \ref{extension_THM}, we construct the extensions for perverse sheaves over a regular scheme after a twist and in Theorem \ref{extensionari_THM} over an arithmetic scheme without a twist.

In Section 3, we construct certain schemes required for our arguments. The material of this section is well-known and we recall it for clarity of exposition. In Section 4, we prove the decomposition theorem and the Hard Lefschetz theorem for arithmetic complexes. In Section 5, we treat the arithmetic Kashiwara conjecture in positive equicharacteristic.

In the appendix, we prove theorems which do not fit the text otherwise. In Theorem \ref{fibfuncperv_THM}, we prove that the category of perverse sheaves admits faithful exact functors to the category of finite dimensional $\Qbarl$-vector spaces for all schemes of finite type over a field. At the end of the appendix, we prove generic base change, generic universal local acyclicity, and the generic existence of a constructible internal mapping object for finitely presented morphisms of qcqs schemes. 
\newline

\textit{Acknowledgments: } I would like to thank Joseph Ayoub for encouraging me to write this article and for helpful comments on a draft. I would also like to thank Moritz Kerz and H\'el\`ene Esnault for their interest and comments on a draft of this article. 
\vspace*{3.3mm}
\begin{center}
	\textbf{Notations and conventions}
\end{center}
\vspace*{2.5mm}
	\begin{itemize}
		\item We fix a prime $\ell$ and assume $\ell$ is invertible on all schemes. 
		\item If $S$ is a scheme, we say $X$ is an $S$-scheme if $X$ is a scheme with a fixed finitely presented structure morphism $X \rightarrow S$.
		\item The letter $k$ always denotes a field. 
		\item For a scheme $X$, we denote by $\derCat{c}{X}{\Qbarl}$ the category of construcible $\Qbarl$-complexes constructed in \cite[Def.~6.8.8]{BhattSPE}. 
		\item We sometimes consider constructible sheaves with coefficients in $\CO$, where $\CO$ is a ring of integers in a finite extension $K/\BQ_\ell$. We denote by $\lambda$ a uniformizer in $\CO$. 
		\item Let $S_0$ be a topologically Noetherian scheme and $X_0$ an $S_0$-scheme. We denote the category of relative perverse sheaves by $\perv{X_0/S_0}$. We also consider the category of universally locally acyclic relatively perverse sheaves $\text{Perv}^{\text{ULA}}(X_0/S_0)$. We define the category of perverse sheaves $\perv{X}$ on a scheme $X$ of finite type over an field $k$ to be $\perv{X} := \perv{X/k}$. The category $\perv{X_0/S_0}$ is abelian by   \cite[Thm.~1.1]{hansen2023relativeperversity} and Proposition \ref{topnot_PROP}.
		\item For a field $k$, we denote by $\overline{k}$ an algebraic closure of $k$.
		\item A field $k$ is called \textit{arithmetic} if it is finitely generated over its prime field.
		\item A scheme $X$ is called \textit{arithmetic} if it is of finite type over $\BZ[1/\ell]$. 

		\item Let $X_0$ be a scheme over a scheme $S_0$ and let $S \rightarrow S_0$ a morphism. We denote by $X_S$ the base change of $X_0$ to $S$. For example, if $\overline{x} \rightarrow S_0$ is a geometric point, then we have $X_{\overline{x}} = X_0\times_{S_0}\overline{x}$. We also use this notation for complexes on $X_0$ to indicate the pullback to $X_S$. If we are given a morphism $\pi\colon X_0 \rightarrow S_0$, then we sometimes also use the letter $\pi$ to denote the base change $\pi\colon X_S \rightarrow S$.
		\item A complex $K \in \derCat{c}{X}{\Qbarl}$ on a scheme $X$ of finite type over a field $k$ is called \textit{semisimple} if $K$ is isomorphic to a direct sum of shifted irreducible perverse sheaves. In other words, there is an isomorphism
		\[
		K \cong \bigoplus_{n \in \BZ} \pervCoh{n}{K}[-n]
		\]
		and $\pervCoh{i}{K}$ is a direct sum of irreducible perverse sheaves. The complex $K$ is called \textit{geometrically semisimple} if $K_{\overline{k}}$ is semisimple.
		\item A complex is called \textit{arithmetic} if it satisfies Definition \ref{arithmeticcomplex_DEF}. All semisimple geometric complexes (see \cite[{}6.2.4]{BBD}) are arithmetic.
	\end{itemize}

\section{Descent and decomposition}\label{descent_SEC} The goal of this section is to provide a variety of theorems necessary for our application of the extension theorem to problems about geometrically semisimple arithmetic complexes. 

Note that the extension theorem only applies to geometrically irreducible perverse sheaves. For our application, we must therefore be able to decompose arithmetic geometrically semisimple complexes into arithmetic geometrically irreducible perverse sheaves. Such a decomposition is provided by Theorem \ref{geomsimpleperverse_THM}.

Section \ref{arithmeticcomplexes_SEC} treats arithmetic complexes in detail. In particular, we show that all semisimple arithmetic complexes on a finite type scheme over a separably closed field admit a model over an arithmetic field in Corollary \ref{models_COR}.

Finally, to prove statements on semisimplicity of complexes by specialization, we want to understand the behaviour of the decomposition series of the fibers of a relatively perverse sheaf. We construct a relative decomposition series in Theorem \ref{filtration_THM} and prove a variety of properties of this decomposition series.
\subsection{Geometrically irreducible lisse sheaves} The extension theorem is first established for geometrically irreducible sheaves. Therefore, to apply the extension theorem, we must be able to provide geometrically irreducible perverse sheaves from geometrically semisimple perverse sheaves in the first place. We achieve this by decomposing geometrically semisimple sheaves into geometrically irreducible sheaves. The goal of the following theorem is to provide such a decomposition for lisse sheaves.
\begin{remark}\label{unibranchClass_RMK}
	Let $X$ be a geometrically unibranch $k$-scheme, and $k'/k$ an arbitrary field extension. The scheme $X_{k'}$ is geometrically unibranch by \cite[Prop.~6.15.7]{EGA4}. In particular, if $X_{k'}$ is connected, $\ell$-adic representations of the etale fundamental group of $\pi_1(X_{k'})$ correspond to lisse $\ell$-adic sheaves by \cite[Lem.~7.4.10]{BhattSPE}. 
	
	Recall that the arithmetic and the geometric monodromy group of a lisse $\Qbarl$-sheaf $F$ on a scheme $X$ as above is the Zariski closure of $\pi_1(X)$ and $\pi_1(X_{\overline{k}})$ in $\mathbf{GL}(F)$ respectively.
\end{remark}
\begin{theorem}\label{geomsimple_THM}
	Let $X$ be a geometrically connected geometrically unibranch  $k$-scheme. Let $F$ be a lisse geometrically semisimple $\Qbarl$-sheaf on $X$. There exists a finite separable extension $k'/k$, lisse geometrically irreducible $\Qbarl$-sheaves $F_1, \ldots, F_n$ on $X_{k'}$, and continuous $\Qbarl$-representations $\rho_1, \ldots, \rho_n$ of the Galois group $G_{k'}$ of $k'$ such that
	\[
	F_{k'} \cong (F_1\otimes\rho_1)\oplus\ldots \oplus (F_n\otimes \rho_n).
	\]
	For all $1 \leq i \leq n$, we can choose $F_i$ so that the geometric and the arithmetic monodromy groups of $F_i$ agree. 
\end{theorem}
We sketch the argument. The first step is to prove that the conjugation action of the Galois group of $k$ on the decomposition factors of $F_{\overline{k}}$ has finite orbits. By a well-known descent argument, we can then prove that the irreducible decomposition factors of $F_{\overline{k}}$ descend to a finite extension of $k$. We can descend the lisse sheaves so that the arithmetic and the geometric monodromy groups agree by Proposition \ref{descentmonodromy_PROP}. 
\begin{lemma}\label{conjaction_LEM}
	Let $\rho\colon G \rightarrow \mathbf{GL}(V)$ be a continuous finite dimensional representation over $\Qbarl$ of a profinite group $G$. Let $H \subseteq G$ be a normal closed subgroup. There exists an open subgroup $U \subseteq G$, such that the conjugation action of $U$  on the isomorphism classes of the irreducible subquotients of $\rho|_H$ is trivial.
\end{lemma}
\begin{proof}
	By Jordan-Hölder, there is a finite number of trace functions $t_1, \ldots, t_n$ associated to the irreducible subquotients of $\rho|_H$. Note that any decomposition series of $\rho|_H$ descends to a finite extension $K/\BQ_\ell$ (see \cite[Lem.~4.3.7]{BhattSPE}). In particular, the trace functions $t_i$ take values in $K$ and are continuous. The group $G$ acts continuously on $H$ by conjugation. The conjugation action of $G$ induces a continuous action on $\text{Hom}(H, K)$ in the compact-open topology. It preserves the finite set $\{t_1, \ldots, t_n\}$. The subspace $\{t_1, \ldots, t_n\}$ is finite and Hausdorff, hence discrete.  Therefore, the stabilizer of $t_i$ in $G$ is open. Since irreducible $\Qbarl$-representations of $H$ are classified by their trace functions, this implies that the isomorphism class of each irreducible subquotient of $\rho|_H$ is stabilized by an open subgroup of $G$. 
\end{proof}

Before we prove Theorem \ref{geomsimple_THM}, we apply Proposition \ref{descentmonodromy_PROP} to local systems. Note that this proposition allows us to descend local systems so that the geometric and the arithmetic monodromy group agree. 
\begin{lemma}\label{descentlocsys_LEM}
	Let $X$ be a geometrically connected geometrically unibranch  $k$-scheme and $F$ an irreducible lisse $\Qbarl$-sheaf on $X_{\overline{k}}$. Suppose 
	\[
	\sigma^*(F) \cong F
	\]
	for all $\sigma\in G_k$. There exists a finite separable extension $k'/k$ such that $F$ descends to a lisse $\Qbarl$-sheaf $F'$ on $X_{k'}$. We can choose $F'$ so that the geometric and the arithmetic monodromy group of $F'$ agree.
\end{lemma}
\begin{proof} We have a short exact sequence
	\[
	0 \rightarrow \pi_1(X_{\overline{k}}) \rightarrow \pi_1(X) \rightarrow G_k \rightarrow 0.
	\]
	This short exact sequence splits after a finite separable extension $k'/k$. Proposition \ref{descentmonodromy_PROP} implies that after a finite separable extension $k''/k'$, the sheaf $F$  descends to a lisse $\Qbarl$-sheaf $F'$ on $X_{k''}$ such that the geometric and the arithmetic monodromy group of $F'$ agree. 
\end{proof}
The following proof is an adaptation of \cite[Prop.~5.3.9]{BBD} to our setting.
\begin{proof}[Proof (Thm.~\ref{geomsimple_THM}).]
	Lemma \ref{conjaction_LEM} implies that after a finite separable extension $k'/k$ the conjugation action of $\pi_1(X_{k'})$ on the isotypical components of $F_{\overline{k}}$ is trivial. In particular, we obtain a direct sum decomposition
	\[
	F_{k'} = \bigoplus_{i = 1}^n G_i
	\] 
	such that $G_i$ is geometrically isotypical for $1 \leq i \leq n$. Lemma \ref{descentlocsys_LEM} implies that there is a finite extension $k''/k'$ and a geometrically irreducible lisse $\Qbarl$-sheaf $F_i$ on $X_{k''}$  whose geometric and the arithmetic monodromy groups agree, and such that $G_i$ is geometrically $F_i$-isotypical. Define the $G_{k''}$-representation
	\[
	\rho_i := \text{Hom}_{\pi_1(X_{\overline{k}})}(F_i, G_i).
	\]
	Then we have an isomorphism given by evaluation
	\[
	\rho_i\otimes F_i = G_i.
	\]
	Thus
	\[
	F_{k''} = (F_1\otimes \rho_1)\oplus \ldots \oplus (F_n\otimes \rho_n).
	\]
\end{proof}
	\subsection{Arithmetic complexes and models}\label{arithmeticcomplexes_SEC} We analyze the condition of arithmeticity in more detail by proving the existence of models for semisimple arithmetic complexes on schemes over separably closed fields. We begin by defining the notion of a model. 
\begin{definition}\label{model_DEF}
	Let $X$ be a  $k$-scheme and $k_0 \subseteq k$ a subfield. A \textit{model of $X$ over $k_0$} is a $k_0$-scheme $X'$ with a fixed isomorphism $\alpha\colon X \rightarrow X'_{k}$. 
	
	Let $K \in\derCat{c}{X}{}$ be a complex. If $X'$ is a model of $X$ over $k_0$, then a \textit{model of $K$ on $X'$} is a complex $K' \in \derCat{c}{X'}{}$ together with a fixed isomorphism $\alpha^*(K'_k) \rightarrow K$. We say $K$ \textit{admits a model over $k_0$} if there exists a model $X'$ of $X$ over $k_0$ such that there exists a model of $K$ on $X'$.
\end{definition}
\begin{remark}\label{uniquenessmodels_REM}
	Models of schemes are unique up to finite field extensions. Let $X$ be a $k$-scheme. Let $X'$ and $X''$ be models of $X$ over fields $k_0 \subseteq k$ and $k_1 \subseteq k$. There exists a finite extension $k_2/k_0k_1$ such that $X'_{k_2} \cong X''_{k_2}$ by a limit argument (see \cite[Prop.~10.1]{milne2024descentalgebraicschemes}). In particular, a model over an algebraically closed field is unique up to isomorphism.
	
	Models of schemes are also unique up to finitely generated field extensions contained in the base-field. We consider schemes $X$, $X'$, and $X''$ as above. Then we have an isomorphism
	\[
	X'\times_{k_1} k \cong X''\times_{k_2}k.
	\]
	Since the schemes $X'$ and $X''$ are of finite type, there exists a finitely generated field extension $k_3/k_0k_1$ of $k_0k_1$ with $k_3 \subseteq k$ such that
	\[
	X'\times_{k_0}k_3 = X''\times_{k_1}k_3
	\]
	Thus we see that models of schemes are unique up to finitely generated field extensions. 
	
	A further limit argument implies that any scheme $X$ of finite type over a field $k$ admits a model over an arithmetic subfield $k_0 \subseteq k$. 
\end{remark}
\begin{definition}\label{arithmeticcomplex_DEF}
	Let $X$ be a $k$-scheme. A complex $K \in \derCat{c}{X}{}$ is called \textit{arithmetic} if there exists a model $X'$ of $X$ over an arithmetic field $k_0$ such that the following condition is satisifed. The model induces an $\text{Aut}(\overline{k}/k_0)$-action on $X_{\overline{k}}$. The orbit $\sigma^*(K_{\overline{k}})$ for  $\sigma \in \text{Aut}(\overline{k}/k_0)$ consists of finitely many isomorphism classes.
\end{definition} 
\begin{remark}\label{independence_RMK}
	Let $X$ be a $k$-scheme. If a complex on $X$ admits a model over an arithmetic field, then it is arithmetic. In particular, a semisimple geometric complex (see \cite[{}6.2.4]{BBD}) is arithmetic.
	
	The condition of being arithmetic is independent of the implicitly chosen model by Remark \ref{uniquenessmodels_REM} in the following sense. Let $K \in \derCat{c}{X}{}$ be an arithmetic complex and $X'$ a model of $X$ over a field $k_0$. There exists a finitely generated extension $k_1/k_0$ with $k_1 \subseteq k$ such that the action of $\text{Aut}(\overline{k}/k_1)$ has a finite orbit when acting on the isomorphism class of $K_{\overline{k}}$.
\end{remark}
We first establish that arithmeticity passes to irreducible subquotients. 
\begin{theorem}\label{subquotientsarithmetic_THM}
	Let $X$ be a $k$-scheme an $K \in \derCat{c}{X}{}$ an arithmetic complex. We have the following:
	\begin{enumerate}[]\item The perverse cohomology sheaves $\pervCoh{n}{K}$ are arithmetic for all $n \in \BZ$.
		\item Any irreducible subquotient of $\pervCoh{n}{K}_{\overline{k}}$ is arithmetic for all $n \in \BZ$. 	
	 \end{enumerate}
\end{theorem}
\begin{proof} Note that the perverse cohomology sheaves of $K$ are arithmetic because an isomorphism $\sigma^*(K) \xrightarrow{\sim} K$ induces isomorphisms
	\[
	\sigma^*(\pervCoh{n}{K}) \xrightarrow{\sim} \pervCoh{n}{K}
	\]
	for all $n \in \BZ$ by $t$-exactness of $\sigma^*$. 
	
	We assume $K$ is perverse. Let $X'$ be a model over an arithmetic field $k_0$ such that the orbit of the isomorphism class $K_{\overline{k}}$ under the $\text{Aut}(\overline{k}/k_0)$-action is finite. Let $\{M_1, \ldots, M_n\}$ be the orbit of the isomorphism class of $K_{\overline{k}}$. Let $\{N_1,\ldots, N_m\}$ be the isomorphism classes of irreducible subquotients of $M_1, \ldots, M_n$. Then $\text{Aut}(\overline{k}/k_0)$ acts on the set $\{N_1, \ldots, N_m\}$ by the Jordan-Hölder theorem because the functors $\sigma^*$ are $t$-exact equivalences of categories for $\sigma \in \text{Aut}(\overline{k}/k_0)$. In particular, the orbit of $N_i$ is finite for $1 \leq i \leq m$.
\end{proof}
We can always descend a complex to the algebraic closure of a finitely generated field. Let us begin by proving that morphisms uniquely descend to models over separably closed fields.
\begin{lemma}\label{smoothbasechange_LEM}
	Let $k$ be a separably closed field and $X$ a $k$-scheme. Let $k'/k$ be a separably closed extension and $K, L \in \derCat{c}{X}{}$ complexes on $X$. The base change morphism
	\[
	\text{Hom}(K, L) \rightarrow \text{Hom}(K_{k'}, L_{k'})
	\]
	is an isomorphism.
\end{lemma}
\begin{proof}
	Let $\alpha\colon X_{k'} \rightarrow X$ be the base change morphism. \cite[Prop.~9.5.5]{LeiFu} implies that the morphism
	\[
	\alpha^*(\text{Hom}(K, L)) \rightarrow \text{Hom}(\alpha^*(K), \alpha^*(L)) 
	\]
	is an isomorphism. The natural morphism
	\[
	H^0(X, \text{Hom}(K, L)) \rightarrow 	H^0(X_{k'}, \alpha^*(\text{Hom}(K, L))) 
	\]
	is an isomorphism by smooth base change. Thus the composite
	\[
	\text{Hom}(K, L) \rightarrow \text{Hom}(\alpha^*(K), \alpha^*(L)) 
	\]
	is an isomorphism. This composite identifies with the morphism above. 
\end{proof}	
\begin{lemma}\label{lefschetzprinciple_LEM}
	Let $X$ be a $k$-scheme and $K \in \derCat{c}{X}{}$. Suppose $X'$ is a model of $X$ over a field $k_0 \subseteq k$. There exists a finitely generated field extension $k_1/k_0$ with $k_1 \subseteq \overline{k}$ such that $K_{\overline{k}}$ admits a model on $X'_{\overline{k}_1}$.
	
	Given an extension $k_1/k_0$ with $k_1 \subseteq \overline{k}$ such that $K_{\overline{k}}$ admits a model on $X'_{\overline{k}_1}$, then the model is unique up to unique isomorphism.
\end{lemma}
\begin{proof} We can assume $k$ to be algebraically closed by base changing $k$ to $\overline{k}$. The uniqueness follows from Lemma \ref{smoothbasechange_LEM}.

	We construct a subfield $k_1 \subseteq k$ during the proof by repeatedly adjoining a finite number of elements of $k$ to $k_1$. We then say that we are enlarging $k_1$. We begin the argument with $k_1 := k_0$. 
	
	We assume the lemma is known when $K$ is a perverse sheaf. We now prove the lemma by induction on the number of non-zero perverse cohomology sheaves of $K$. In the base case we have $K = 0$, i.e. there is nothing to prove. Suppose $K$ has at least one non-zero perverse cohomology sheaf. Let $n \in \BZ$ be the largest integer with $\pervCoh{n}{K} \neq 0$. We have a distinguished triangle
	\[
	\pvTRC{\leq n - 1}(K) \rightarrow K \rightarrow \pervCoh{n}{K}[-n] \xrightarrow[]{+}.
	\]
	We can enlarge $k_1$ such that the complexes $\pvTRC{\leq n - 1}(K)$ and $\pervCoh{n}{K}$ admit models over ${\overline{k}_1}$ by the induction hypothesis and the assumption. By Lemma \ref{smoothbasechange_LEM}, the morphism
	\[
	\pervCoh{n}{K}[-n] \rightarrow \pvTRC{\leq n - 1}(K)[1]
	\]
	descends to a morphism of models on $X'_{\overline{k}_1}$. Hence its cone admits a model on $X'_{\overline{k}_1}$. The cone is $K[1]$ hence $K$ admits a model on $X'_{\overline{k}_1}$.
	
	We assume the lemma is known when $K$ is an irreducible perverse sheaf. We can now prove the lemma for a perverse sheaf by induction on the length, similar to the first step.
	
	We treat the case when $K$ is an irreducible perverse sheaf. There exists a locally closed immersion $j\colon U \rightarrow X$ with $U$ irreducible and smooth such that $j^*(K)$ is lisse, concentrated in degree $-d$ for some integer $d \geq 0$, and $K = j_{!*}j^*(K)$. We can enlarge $k_1$ and assume the locally closed subset $j\colon U \rightarrow X$ admits a model 
	\[
	U' \rightarrow X'_{k_1}
	\]
	over $k_1$. Note that $U'$ is geometrically irreducible and smooth because $U$ is irreducible and smooth and $k$ is algebraically closed.
	
	Assume, after enlarging $k_1$, that $j^*(K)$ admits a model on $U'_{\overline{k}_1}$. The complex $K$ admits a model on $X'_{\overline{k}_1}$, since intermediate extension commutes with base change. Therefore, it is sufficient if we construct a model for $j^*(K)$ on $U_{\overline{k}_1}$ after enlarging $k_1$. We can write $j^*(K) = F[d]$ for a lisse sheaf $F$ on $U$. In particular, it is sufficient to construct a model of $F$ on $U'_{\overline{k}_1}$ after an enlargement of $k_1$.  
	
	There exists a finite Galois cover $q\colon V \rightarrow U$ such that $q^*(F)$ has pro-$\ell$-monodromy (i.e. the image of the fundamental group $\pi_1(V)$ in $\text{GL}(F)$ is a pro-$\ell$ group). We can enlarge $k_1$ and assume $q$ descends to a finite Galois cover $V' \rightarrow U'$ over $k_1$. Assume $q^*(F)$ admits a model on $V'_{\overline{k}_1}$.  The base change morphism
	\[
	\pi_1(U) \rightarrow \pi_1(U'_{{\overline{k}_1}})
	\]
	is surjective because $\overline{k}_1$ is separably closed. Thus any direct factor of $q_*q^*(F)$ admits a model on $U'_{\overline{k}_1}$. In particular, $F$ admits a model over $\overline{k}_1$ because $F$ is a direct factor of $q_*q^*(F)$ by the projection formula. Thus it is sufficient if we construct a model of $q^*(F)$ on $V'_{\overline{k}_1}$.
	
	The base change morphism
	\[
	\pi_1^{\ell}(V_{\overline{k}}) \rightarrow \pi_1^{\ell}(V'_{{\overline{k}_1}})
	\]
	on pro-$\ell$ fundamental groups is an isomorphism by \cite[Cor.~4.5]{orgogozoFundamentalGroup}. Thus $q^*(F)$ admits a model over $\overline{k}_1$. 
\end{proof}
For geometrically irreducible arithmetic perverse sheaves, we prove the existence of models after a twist by a rank one sheaf.
\begin{theorem}\label{modelsIrr_THM}
	Let $X$ be a $k$-scheme. Let $M \in \perv{X}$ be a geometrically irreducible arithmetic perverse sheaf on $X$. For a model $X'$ of $X$ over a field $k_0 \subseteq k$, there exists a finitely generated field extension $k_1/k_0$ with $k_1 \subseteq\overline{k}$ and a one-dimensional character $\chi$ of the Galois group of $k_1k$ such that $M_{k_1k}\otimes \chi$ admits a model on $X'_{k_1}$. 
\end{theorem}
\begin{proof}
	We construct $k_1$ by enlarging the field during the proof as in Lemma \ref{lefschetzprinciple_LEM}. By Remark \ref{independence_RMK}, we can assume the action of $\text{Aut}(\overline{k}/k_1)$ has a finite orbit when acting on the isomorphism class of $M$ by taking a finitely generated extension $k_1/k_0$. By Lemma \ref{lefschetzprinciple_LEM}, we can enlarge $k_1$ and assume $M$ descends to a model $M'_{\overline{k}_1}$ on $X'_{\overline{k}_1}$.
	
	After a finite extension of $k_1$, we can assume there exists an immersion $j\colon U \rightarrow X'_{k_1}$ with $U$ smooth and geometrically connected, $M_{k_1k} = j_{!*}(j^*(M_{k_1k}))$, and $j^*(M_{k_1k})$ is lisse. Since the intermediate extension commutes with base change, we can replace $X$ by $U$ and assume $M$, and hence $M'_{\overline{k}_1}$, are lisse.
	
	The action of $\text{Aut}(\overline{k}/k_1)$ on the isomorphism class of $M$ factors through the Galois group $G_{k_1}$ of $k_1$. Therefore, the action of $G_{k_1}$ has a finite orbit when acting on the isomorphism class of $M'_{\overline{k}_1}$ because the restriction morphism
	\[
	\text{Aut}(\overline{k}/k_1) \rightarrow G_{k_1}
	\] 
	is surjective. An argument with trace functions, similar to the one in Lemma \ref{conjaction_LEM}, implies that the stabilizer in $G_{k_1}$ of $M$ is closed and hence open (see also \cite[Cor.~5.4]{zockmasterthesis}). Thus, after passing to a finite (separable) extension of $k_1$, Lemma \ref{descentlocsys_LEM} implies that $M'_{\overline{k}_1}$ descends to a model $M'$ on $X_{k_1}$. The space 
	\[
	\chi := \text{Hom}(M'_{\overline{k}}, M_{\overline{k}})
	\] 
	is one-dimensional and comes with a continuous action by the group $G_{k_1k}$. We have an isomorphism given by evaluation
	\[
	M'_{k_1k}\otimes\chi = M_{k_1k}.
	\]
	Thus $M'$ is a model of $M_{k_1k}\otimes\chi^\vee$ on $X'_{k_1}$. 
\end{proof}
In particular, we can prove the existence of models under a semisimplicity assumption.
\begin{corollary}\label{models_COR}
	Let $k$ be a separably closed field and $X$ a  $k$-scheme. Let $K \in \derCat{c}{X}{}$ be an arithmetic semisimple complex. The complex $K$ admits a model over an arithmetic field.
\end{corollary}
\begin{proof}
	Because $K$ is semisimple and the irreducible factors are arithmetic by Theorem \ref{subquotientsarithmetic_THM}, we can assume $K$ is an irreducible perverse sheaf. The sheaf $K$ is then geometrically irreducible. This case is treated by Theorem \ref{modelsIrr_THM}.
\end{proof}


\subsection{Geometrically semisimple perverse sheaves}
We generalize Theorem \ref{geomsimple_THM} to perverse sheaves. This provides a variant of \cite[Prop.~5.3.9]{BBD} over arbitrary base-fields. 
\begin{theorem}\label{geomsimpleperverse_THM}
	Let $X$ be a $k$-scheme. Let $M \in \Perv{X}{\Qbarl}$ be a geometrically semisimple perverse sheaf on $X$. There exists a finite separable extension $k'/k$, geometrically irreducible perverse sheaves $M_1, \ldots, M_n \in \Perv{X_{k'}}{\Qbarl}$, and continuous $\Qbarl$-representations $\rho_1, \ldots, \rho_n$ of the Galois group of $k$ such that
	\[
	M_{k'} = (M_1\otimes\rho_1) \oplus \ldots \oplus (M_n\otimes\rho_n).
	\]
	If $M$ is arithmetic, then $M_i$ is arithmetic for $1\leq i \leq n$. 
\end{theorem}
This theorem will be deduced from the variant for lisse sheaves already proven. 
\begin{proof}  Let $Z_1, \ldots, Z_n \subseteq X_{\overline{k}}$ be the supports of the irreducible subquotients of $M_{\overline{k}}$. There exists a finite separable extension $k'/k$ such that the closed subsets come from base change by geometrically irreducible closed subsets $Z_i \subseteq X_{k'}$. For all closed immersions $i\colon Z \rightarrow X_{k'}$ the natural map induces an isomorphism
	\[
	\pervCoh{0}{i^!(M)} \rightarrow \pervCoh{0}{i^*(M)}
	\]
	because it induces an isomorphism after base change to $X_{\overline{k}}$ by geometric semisimplicity. 
	In particular, we obtain a decomposition
	\[
	M_{k'} = \bigoplus_{i = 1}^n N_i
	\]
	with perverse sheaves $N_i \in \Perv{X_{k'}}{\Qbarl}$ whose irreducible subquotients have support $Z_i$. Let $1 \leq i \leq n$. There is a dense open immersion $j\colon U_i \rightarrow Z_i$ such that $M_i$ is lisse over $U_i$, the scheme $U_i$ is geometrically unibranch, and we have 
	\[
	M'_i = j_{!*}j^*(N_i).
	\] 
	For $j^*(N_i)$, the required decomposition was constructed in Theorem \ref{geomsimple_THM}.
	
	Note that $(M_i)_{\overline{k}}$ are irreducible subquotients of $M_{\overline{k}}$, hence they are arithmetic by Theorem \ref{subquotientsarithmetic_THM}.
\end{proof}
\subsection{Decomposition series of a relatively perverse sheaf} During the course of the proof of the decomposition theorem, it is necessary to understand the behaviour of the decomposition series in a family. The following theorem extends the decomposition series from the generic fiber to the entire space by \cite[Thm.~1.1]{hansen2023relativeperversity}. Before stating the theorem, we construct the relative intermediate extension of a perverse sheaf generically. This construction will also be used later when extending perverse sheaves.
	\begin{lemma}\label{intermed_LEM}
		Let $S_0$ be an irreducible qcqs scheme and $X_0$ an $S_0$-scheme. Let $j\colon V_0 \rightarrow X_0$ be an immersion and $M \in \text{Perv}(V_0/S_0)$. There exists a constructible dense open subscheme $U_0 \subseteq S_0$ and $N \in \text{Perv}^{\text{ULA}}(X_{U_0}/U_0)$ such that $N_x = j_{!*}(M_x)$ for all $x \in U_0$. 
	\end{lemma}
	\begin{proof}
		By generic base change (see Theorem \ref{genbc_THM}), there exists a constructible dense open subscheme $U_0 \subseteq S_0$ so that $j_*(M)$ is constructible over $U_0$. We define $N \in \perv{X_{U_0}/U_0}$ to be the relative perverse image of the morphism 
		\[
		{}^{\idp/U_0} \mathscr{H}^0(j_! M) \rightarrow {}^{\idp/U_0} \mathscr{H}^0(j_* M)
		\]
		By Theorem \ref{genbc_THM}, we can shrink $U_0$ so that $j_*(M)$ commutes with base change over $U_0$. For all $x \in U_0$, this implies that the complex $N_x$ coincides with $j_{!*}(M_x)$, since restriction to a fiber induces an exact functor from the relatively perverse category to the absolutely perverse category on the fiber. We can shrink $U_0$ again so that $N$ is ULA over $U_0$ by Theorem \ref{genula_THM}.
	\end{proof}	
	We need a lemma proving that the locus where the reduced geometric fibers are smooth admits a dense open subset.
	\begin{lemma}\label{relreduced_LEM}
		Let $S_0$ be an integral qcqs scheme with function field $k$. Let $X_0$ be an $S_0$-scheme. We assume the reduced generic geometric fiber $(X_{\overline{k}})^{\text{red}}$ is smooth over $\overline{k}$. There exists a universal homeomorphism $F\colon S_0 \rightarrow S_0$ and a dense open constructible subscheme $U_0 \subseteq S_0$ such that 
		\[
		(X_0\times_{S_0, F}U_0)^{\text{red}} \rightarrow U_0
		\]
		is smooth.
	\end{lemma}
	\begin{proof} If the characteristic of $k$ is zero, there is nothing to prove. Suppose the characteristic of $k$ is $p > 0$. There exists $n \geq 1$ such that $(X_{k^{p^{-n}}})^{\text{red}}$ is geometrically reduced over $k^{p^{-n}}$. Therefore, $(X_{k^{p^{-n}}})^{\text{red}}$ is smooth over $k^{p^{-n}}$. The $n$-fold Frobenius is a universal homeomorphism
		\[
		F\colon S_0 \rightarrow S_0
		\]
		which induces the extension $k^{p^{-n}}/k$ at the generic point.  In particular, the generic fiber of the morphism
		\[
		(X_0\times_{F, S_0}S_0)^{\text{red}} \rightarrow S_0
		\]
		is smooth. Since this morphism is finitely presented, it is smooth over a constructible dense open subscheme of $S_0$. 
	\end{proof}
	In the following theorem, we assume the supports of the irreducible subquotients of the relatively perverse sheaf on the generic fiber are geometrically irreducible. This is so we can apply the monodromy specialization theorem, for which we require that any dense open subset of the support is geometrically connected. This can always be achieved after a finite separable extension of the base field by applying Theorem \ref{geomsimple_THM} to the irreducible factors of a decomposition series and by noting that the support of a geometrically irreducible perverse sheaf is geometrically irreducible. 
		\begin{theorem}\label{filtration_THM}
		Let $S_0$ be a topologically Noetherian integral geometrically unibranch scheme with function field $k$. Let $X_0 \rightarrow S_0$ be a morphism of finite type and $M \in \text{Perv}(X_0/S_0)$ a relatively perverse sheaf over $S_0$. We assume the irreducible subquotients of $M_{k}$ have geometrically irreducible support.  There is a dense open subscheme $U_0 \subseteq S_0$ and relatively perverse sheaves $M_i \in \text{Perv}^{\text{ULA}}(X_{U_0}/U_0)$ for $0 \leq i \leq n$ which fit into a filtration
		\[
		0 = M_0 \subset M_1 \subset \ldots M_{n - 1} \subset M_n = M_{U_0}.
		\]
		Moreover, we have a closed subset $Z_i \subseteq X_{U_0}$ and a dense open subscheme $V_i \subseteq Z_i$ for all $1\leq i \leq n$. The following properties are satisfied:
		\begin{enumerate}[]
			\item The generic fiber $(M_i/M_{i - 1})_k$ is irreducible.
			\item The morphism $Z_i \rightarrow U_0$ has geometrically irreducible fibers. 
			\item The morphism $V_i \rightarrow U_0$ has geometrically connected fibers. The reduced geometric fibers of $V_i \rightarrow U_0$ are smooth.
			\item The restriction $(M_i/M_{i - 1})|_{V_i}$ is lisse.
			\item  The image of $\pi_1(V_{i, \overline{x}})$ in $\mathbf{GL}((M_i/M_{i -1})_{\overline{x}})$ does not depend on the geometric point $\overline{x} \rightarrow U_0$ up to conjugacy (see Definition \ref{conjugate_DEF}). 
			\item Denote by $j\colon V_i \rightarrow Z_i$ the inclusion. Let $x \in U_0$ be a point. We have
			\[
			j_{!*}(j^*((M_i/M_{i - 1})_x)) = (M_i/M_{i - 1})_x.
			\]
		\end{enumerate}
	\end{theorem}
	\begin{proof} Over the course  of the argument, we repeatedly shrink $U_0 \subseteq S_0$ without changing notation. By Theorem \ref{genula_THM}, we can shrink $S_0$ to a dense open subset $U_0\subseteq S_0$ so that $M$ is universally locally acyclic over $U_0$ by Theorem \ref{genula_THM}. By \cite[Thm.~1.10]{hansen2023relativeperversity}, we can lift a decomposition series of $M_k$ to a decomposition series 
		\[
		0 = M_0 \subset M_1 \subset \ldots \subset M_{n - 1} \subset M_n = M|_{U_0}
		\]
		of relatively perverse sheaves $M_i \in \text{Perv}^{\text{ULA}}(X_0/U_0, \Qbarl)$. 
		
		This filtration reproduces a decomposition series on the generic fiber by construction, thus (1) is satisfied.
		
		Let $1 \leq i \leq n$. Define $Z_i$ to be the support of the quotient $M_i/M_{i - 1}$. Because the support of the generic fiber of $M_i/M_{i - 1}$ is geometrically irreducible by assumption, the support $Z_i$ is geometrically irreducible over the generic fiber. We can shrink $U_0$, so that $Z_i$ has geometrically irreducible fibers over $U_0$. Thus we obtain property (2). 
		
		There is a dense open subset $j_i\colon V_i \rightarrow Z_i$, such that $((V_i)_{\overline{k}})^{\text{red}}$ is smooth, irreducible, and $M_i/M_{i - 1}$  is lisse over $V_i$. Thus we obtain (3) and (4). In particular, by shrinking $U_0$, the reduced geometric fibers of $V_i$ over $U_0$ are smooth and connected over $U_0$ by Lemma \ref{relreduced_LEM}.	In particular, by shrinking $U_0$, Lemma \ref{relreduced_LEM} allows us to assume there is a universal homeomorphism $U_0' \rightarrow U_0$ such that 
		\[
		(X_{U_0'})^{\text{red}} \rightarrow U_0'
		\]
		is smooth. Therefore, Theorem \ref{specializationmonodromy_THM} applies to the fibers of this morphism. This implies we can shrink $U_0'$ so that the image of the fundamental group does not depend on the geometric fiber over $U_0'$ up to conjugacy. In particular, we can shrink $U_0$ so that the image of the fundamental group does not depend on the chosen geometric fiber over $U_0$. Therefore, property (5) is satisifed.
		
		By shrinking $U_0$, Lemma \ref{intermed_LEM} permits us to assume that there are relatively perverse universally locally acyclic sheaves $N_i \in \text{Perv}^{\text{ULA}}(X_{U_0}/U_0, \Qbarl)$ such that $N_{i, x}$ is isomorphic to the intermediate extension $j_{i!*}((M_i/M_{i - 1})_x)$ for all $1 \leq i \leq n$. Since both $(M_i/M_{i -1})_k$ and $N_i$ are univerally locally acyclic and isomorphic to $j_{i!*}((M_i/M_{i - 1})_k)$ on the generic fiber,  \cite[Thm.~1.10]{hansen2023relativeperversity} shows that there is an isomorphism
		\[
		N_i \cong M_i/M_{i - 1}.
		\]  
		In particular, the fibers $(M_i/M_{i - 1})_x$ are intermediate extensions along the immersion $V_{i, x} \rightarrow Z_{i, x}$.
	\end{proof}
	We also note the following Lemma, describing how semisimplicity behaves under passage to the generic fiber: 
	\begin{lemma}\label{irred_LEM}
		Let $k_0$ be a field and $S_0$ a connected regular $k_0$-scheme with function field $k$ of dimension $d$. Let $X_0$ be an $S_0$-scheme and $M \in \perv{X_0}$.
		\begin{enumerate}[]\item Suppose $M$ is universally locally acyclic over $S_0$. Then $M_k[-d]$ is irreducible if and only if $M$ is irreducible.
		\item If $M$ is geometrically semisimple, then the restriction to the generic fiber $M_k[-d]$ is semisimple.
		\end{enumerate}
	\end{lemma}
	\begin{proof}
		Point (1) can be proven from \cite[Cor.~1.12]{hansen2023relativeperversity} and \cite[Thm.~1.10]{hansen2023relativeperversity}.
		
		Suppose $M$ is geometrically semisimple. We are free to base change to the algebraic closure of $k_0$ by passing to an irreducible component of $(S_0)_{\overline{k}_0}$. We can assume $M$ is irreducible. By Theorem \ref{genula_THM} we can assume $M$ is ULA over $S_0$ by passing to a dense open subset. By (1), the generic fiber $M_k[-d]$ is irreducible. An irreducible perverse sheaf is geometrically semisimple, hence the claim. 
	\end{proof}
	\section{Extension Theorem} The goal of this section is to prove the extension theorem. We begin by proving a preliminary extension theorem which gives a criterion for when a sheaf extends in terms of ramification after pulling back the sheaf along a section. We then apply this criterion to arithmetic sheaves to obtain the extension theorem for lisse sheaves. We then finish the section by constructing extensions for perverse sheaves. 
\subsection{Preliminary extension theorem} We go on to prove the extension theorem. We begin by proving the following extension theorem from the Zariski-Nagata purity theorem and Abhyankar's Lemma. For the extension theorem for lisse sheaves, we have to use integral models repeatedly. We therefore consider lisse $\CO$-sheaves over a ring of integers $\CO$ in a finite extension $K/\BQ_\ell$. We denote by $\lambda$ a uniformizer of $\CO$. The following definition is convenient for what follows.
\begin{definition}
	Let $X_0$ be a connected Noetherian regular scheme with function field $k$ and $D_0 \subset X_0$ an irreducible divisor on $X_0$ with generic point $x \in X_0$. Let $k'$ be the fraction field of the strict henselization of $X_0$ at $x$.  Let $F$ be a representation of the Galois group of $k$ with coefficients in $\CO$. We say $F$ is \textit{unramified along $D_0$}, if the Galois group $G_{k'}$ (the inertia) acts trivially on $F$. We say $F$ is \textit{unramified} if it is unramified along all irreducible divisors on $X_0$. We say it is \textit{generically unramified} if it is unramified at almost all divisors in $X_0$, i.e. if it ramifies only at finitely many irreducible divisors. 
\end{definition}
We can now state the preliminary extension theorem. 
\begin{theorem}\label{extensionpre_THM}
	Let $S_0$ be an connected regular Noetherian scheme. Let $X_0 \rightarrow S_0$ be a smooth finite type morphism with connected fibers. Let $S_0'\rightarrow S_0$ be a dominant etale neighborhood with a section $s\colon S_0' \rightarrow X_0$. Let $F$ be a lisse $\CO$-sheaf on the generic fiber such that $s^*(F)$ is generically unramified over $S_0'$.  There exists a dense open subset $U_0 \subseteq S_0$ and a unique lisse $\CO$-sheaf $F_0$ on $X_{U_0}$ such that the generic fiber of $F_0$ over $S_0$ is isomorphic to $F$.  
\end{theorem}
This amounts to proving that the sheaf $F$ is generically unramified on $X_0$. By the Zariski-Nagata purity theorem, the ramification of $F$ is along (irreducible) divisors. Because $\pi$ is flat, each irreducible divisor on $X_0$ is either horizontal, i.e. a divisor which intersects the generic fiber, or vertical, i.e. the pullback of a divisor from $S_0$ to $X_0$. We already know that $F$ does not ramify at the horizontal divisors because $F$ is lisse on the generic fiber. Hence $F$ can only ramify along vertical divisors. The plan is to detect the ramification along a vertical divisor after pulling $F$ back to $S_0'$ along the section $s$. 

We begin by stating the Zariski-Nagata theorem, which states that a sheaf extends from the generic point to the entire space if it is unramified. The next lemma shows that the ramification along a vertical divisor can be detected after pullback along an etale section if the ramification is trivial mod $\lambda$. We construct a generic extension of $F\otimes\CO/\lambda\CO$ by a limit argument, which shows that $F\otimes\CO/\lambda\CO$ is generically unramified. This is then sufficient to prove the preliminary extension theorem. 

The Zariski-Nagata theorem can be stated as follows.
\begin{lemma}\label{unramifiedext_LEM}
	Let $X_0$ be a Noetherian regular irreducible scheme with function field $k$. Suppose $F$ is an unramified $\CO$-representation of $G_k$. There exists a unique lisse sheaf $F_0$ with $\CO$-coefficients such that the Galois representation induced by $F_0$ at the generic point is canonically isomorphic to $F$.
\end{lemma}
\begin{proof} The morphism
	\[
	G_k \rightarrow \pi_1(X)
	\]
	is surjective because $X$ is normal. Hence the extension $F_0$ (if it exists) is unique up to unique isomorphism. By a limit argument, we can thus assume the representation $F$ has coefficients in $\CO/\lambda^n\CO$. In this case, it follows from the Zariski-Nagata purity theorem. 	
\end{proof}
We begin by proving that the ramification can be picked up by a section from Abhyankar's Lemma.
\begin{lemma}\label{pickup_LEM}
	Let $S_0$ be the spectrum of a discrete valuation ring and $X_0 \rightarrow S_0$ a smooth finite type morphism with connected fibers. Suppose we are given a section $s\colon S_0' \rightarrow X_0$ on a surjective etale neighborhood $S_0'\rightarrow S_0$. Let $F$ be lisse $\CO$-sheaf on the generic fiber such that $s^*(F)$ is unramified and $F\otimes \CO/\lambda\CO$ is unramified along the special fiber. Then $F$ is unramified along the special fiber of $\pi$. 
\end{lemma}
\begin{proof} 
	Let $T_0$ be the strict henselization of $X_0$ at the image of the closed point of $S_0$ along the section $s$. The special fiber is a connected regular Cartier divisor on $X_0$. In particular, the preimage $D_0 \subset T_0$ is a regular Cartier divisor on $T_0$. Thus $D_0$ is connected and regular. We denote by $\pi_1^t(T_0)$ the fundamental group of finite etale covers of $T_0 - D_0$ tamely ramified along $D_0$.
	
	Because $F\otimes\CO/\lambda\CO$ is unramified along the special fiber, the local system $F\otimes \CO/\lambda\CO$ extends to a local system on $X_0$ by Lemma \ref{unramifiedext_LEM}. Any local system on $T_0$ is trivial because it is a strictly henselian local ring.  Thus $F\otimes\CO/\lambda\CO$ restricted $T_0 - D_0$ is trivial because it extends to a local system on $T_0$. This implies that the image of the morphism
	\[
	\pi_1(T_0 - D_0) \rightarrow \text{GL}(F_{\overline{k}})
	\]
	is a pro-$\ell$-group because it is trivial mod $\lambda$ (where $k$ is the function field of $S_0$). Note that the kernel of the morphism
	\[
	\pi_1(T_0 - D_0) \rightarrow \pi_1^t(T_0)
	\]
	admits no finite quotients of order co-prime to residual characteristic of the generic point of $D_0$. We have assumed $\ell$ to be invertible, which implies that the residual characteristic of the generic point of $D_0$ is not $\ell$. Now we obtain that the Galois group of the function field of $T_0$ acts through $\pi_1^t(T_0)$ on $F$.
	
	Let $T$ be the strict henselization of the generic point of $D_0$ in $X_0$ and $\Tilde{S}_0$ the strict henselization of $S_0$. The section lifts to a section $\tilde{s}\colon \Tilde{S}_0 \rightarrow T_0$, hence we have a diagram	
	\[
	\begin{tikzcd}
		\pi_1^t(T) \arrow[r] & {\pi_1^t(T_0)} \arrow[r] \arrow[d] & \pi_1^t(\Tilde{S}_0) \arrow[l, "\Tilde{s}"', bend right=30] \\
		& \text{GL}(F_{\overline{k}})                                 &                                            
	\end{tikzcd}
	\]
	By Abhyankar's Lemma, see \cite[Exp.~XII,~Cor.~5.3]{SGA1}, all arrows in the first row of this diagram are isomorphisms. In particular, to prove the action of the Galois group of the generic point of $T$ on $F$ is trivial, it is sufficient to prove the action of the generic point of $\Tilde{S}_0$ on $s^*(F)$ is trivial. This is satisfied by assumption.
\end{proof}

We construct models for our sheaves, which will be useful also when proving the full extension theorem. The following lemma follows from a limit argument.
\begin{lemma}\label{spreadsheaves_LEM}
	Let $S_0$ be an integral Noetherian scheme with function field $k$ and $X_0$ and $S_0$-scheme. Let $F$ be a constructible $\CO/\lambda^n\CO$-local system on $X_k$. There exists a dense open subscheme $U_0 \subseteq S_0$ and a constructible $\CO/\lambda^n\CO$-local system $F_0$ on $X_{U_0}$ such that $F_0$ restricted to $X_k$ is isomorphic to $F$. 
\end{lemma}
\begin{proof}[Proof~(Thm.~\ref{extensionpre_THM})]
	By Lemma \ref{spreadsheaves_LEM}, there exists a dense open subset $U_0\subseteq S_0$ over which $F\otimes\CO/\lambda\CO$ extends. The local system $F\otimes\CO/\lambda\CO$ is unramified over $X_{U_0}$. The pullback $s^*(F)$ is generically unramified by assumption. By Lemma \ref{pickup_LEM}, the Galois representation $F$ is unramified at almost all vertical divisors. Therefore, the Zariski-Nagata purity theorem, Lemma \ref{unramifiedext_LEM}, allows us to extend $F$ as desired. 
\end{proof}
\subsection{Lisse sheaves}\label{lissesheavesExtension_SUBSEC} In the previous section, we proved a criterion for extending lisse sheaves in terms of ramification after pullback along a section on an etale neighborhood of the base. In the following theorem, we apply this to sheaves with generically unramified determinant.
\begin{theorem}\label{simpson_THM}
Let $S_0$ be a connected regular Noetherian scheme and $X_0 \rightarrow S_0$ a smooth finite type morphism with geometrically connected fibers. Let $F$ be a geometrically irreducible lisse $\Qbarl$-sheaf on the generic fiber of this morphism. Suppose there is a dominant etale neighborhood $S_0' \rightarrow S_0$ with a section $s\colon S_0' \rightarrow X_0$ such that $\text{det}(s^*(F))$ is generically unramified on at least one connected component of $S_0'$.  There exists a dense open subscheme $U_0 \subseteq S_0$ such that $F$ extends to a lisse $\Qbarl$-sheaf $F_0$ on $X_{U_0}.$ The extension is unique up to unique isomorphism.
\end{theorem}
For arithmetic schemes, the condition on the determinant is always satisfied by the quasi-unipotent monodromy theorem.
\begin{theorem}\label{extensionthmWeak_THM}
	Let $S_0$ be an irreducible arithmetic scheme and $X_0 \rightarrow S_0$ a smooth finite type morphism with geometrically connnected fibers. Let $F$ be a geometrically irreducible lisse $\Qbarl$-sheaf on the generic fiber of this morphism. There exists a dense open subset $U_0 \subseteq S_0$ such that $F$ extends to a lisse $\Qbarl$-sheaf on $X_{U_0}$. 
\end{theorem}
We give a sketch of the argument. We first remark that we can transport many properties of $F$ to the nearby cycles at almost all codimension one points $x \in S_0$ (i.e. generic points of irreducible divisors) by reducing the sheaf $F$ mod $\lambda^N$ for sufficiently large $N \geq 1$. We can prove this way that for almost all codimension one points $x \in S_0$, the nearby cycles $\Psi_x(F)$ are lisse and geometrically irreducible. By Schur's Lemma, this implies that the action of inertia on $\Psi_x(F)$ is given by scalars at almost all codimension one points $x\in S_0$. 

Moreover, for almost all codimension one points $x \in S_0$, smooth base change implies the crucial inertia-equivariant base change isomorphism
\[
s^*(\Psi_x(F)) \xrightarrow{\sim} \Psi_x(s^*(F)).
\]
Note that the right hand side of this isomorphism essentially evaluates the action of inertia on $s^*(F)$. We know that this action is generically given by scalars by looking at the left-hand side. Thus we only have to control these scalars. This can be achieved by considering the determinant $\text{det}(s^*(F))$.

We begin by proving that each of the above named properties can be satisfied generically. This allows us to prove Theorem \ref{simpson_THM}. We then treat the arithmetic case. We show that a one-dimensional Galois representation of the generic point is always generically unramified from the quasi-unipotent monodromy theorem. Afterwards, we prove Theorem \ref{extensionthmWeak_THM}. 

We fix some further notation and preliminary reductions. We can pass to an irreducible component of $S_0'$ to and assume $S_0'$ is irreducible (and regular). For ease of notation, we base change to $S_0'$. We are allows to take this base change because we can verify the assumptions in Theorem \ref{extensionpre_THM} after the base change. Thus we have an honest section $s\colon S_0 \rightarrow X_0$. Moreover, let $F'$ be a lisse torsion free $\CO$-sheaf with $F'\otimes_\CO \Qbarl \cong F$. Such a sheaf exists by \cite[Lem.~6.8.13]{BhattSPE}. We define $F_n := F'\otimes \CO/\lambda^n\CO$ for all $n \geq 1$. We denote by $k$ the function field of $S_0$.

\begin{lemma}\label{bcnearbycycles_LEM}
	For almost all points of codimension one $x \in S_0$, the nearby cycles $\Psi_x(F)$ are lisse, irreducible, and the natural base change morphism
	\[
	s^*(\Psi_x(F)) \xrightarrow{\sim} \Psi_x(s^*(F))
	\]
	is an isomorphism.
\end{lemma}
\begin{proof}
	Let $N \geq 1$. By Lemma \ref{spreadsheaves_LEM}, there exists $U_0 \subseteq S_0$ such that $F_n$ extends to a local system on $X_{U_0}$ for all $1 \leq n \leq N$. We denote the extended local system by $F_n$ as well.
	
	The smooth base change theorem implies that the specialization map
	\[
	i_x^*(F_1) \rightarrow \Psi_x(F_1)
	\]
	and
	\[
	i_x^*(s^*(F_1)) \rightarrow \Psi_x^*(s^*(F_1))
	\]
	are isomorphisms for all codimension one points $x \in U_0$. Note that $i_x^*(F_1)$ is a local system, hence $\Psi_x(F_1)$ is a local system. This implies $\Psi_x(F')$ is a lisse sheaf, hence $\Psi_x(F)$ is a lisse sheaf (see \cite[Lem.~6.6.8]{BhattSPE})
	
	We have a commutative diagram
	\[
	\begin{tikzcd}
		s^*(i_x^*(F_1)) \arrow[r] \arrow[d, equal] & s^*(\Psi_x(F_1)) \arrow[d] \\
		i_x^*(s^*(F_1)) \arrow[r]         & \Psi_x(s^*(F_1))          
	\end{tikzcd}
	\]
	In particular, the base change $s^*(\Psi_x(F_1)) \rightarrow \Psi_x(s^*(F_1))$ is an isomorphism for all codimension one points $x \in U_0$. Nakayama's Lemma shows that
	\[
	s^*(\Psi_x(F)) \rightarrow \Psi_x(s^*(F))
	\]
	is an isomorphism for all codimension one points $x \in U_0$ (because the morphism is an isomorphism after taking the derived tensor product with $\CO/\lambda$).
	
	To prove that the nearby cycles are irreducible we prove that the geometric monodromy of $F$ is conjugate to the monodromy of $\Psi_x(F)$ (see Definition \ref{conjugate_DEF}) at almost all codimension one points $x \in U_0$. Here, by (geometric) monodromy we mean the image of the (geometric) fundamental group in the general linear group of a fiber of $F$. Let $x \in U_0$ be a codimension one point. There exists a morphism $S \rightarrow S_0$ from the spectrum of an aic valuation ring of rank one $S$ to $S_0$ mapping the generic point to the generic point of $S_0$ and the special point to $y$. Let $j\colon X_{\eta} \rightarrow X_S$ be the inclusion of the generic fiber. We have
	\[
	j_*j^*(F_1) = F_1,
	\]
	by smooth base change.  Hence $j_*(F)$ is a lisse sheaf. We have a diagram on fundamental groups 
	\[
	\begin{tikzcd}
		& \pi_1^{}(X_{\overline{x}}) \arrow[d] &              \\
		\pi_1^{}(X_\eta) \arrow[r, two heads] & \pi_1(X_S) \arrow[r]             & \text{GL}(j_*(F))
	\end{tikzcd}.
	\]
	by Lemma \ref{fundgroupsurjects_LEM} and Psoposition \ref{topnot_PROP}. Therefore, the monodromy of $\Psi_x(F)$ is conjugate to a subgroup of the geometric monodromy of $F$. 
	
	We can now argue as in the proof of Theorem \ref{specializationmonodromy_THM}, i.e. as in \cite[Thm.~8.18.2]{KatzESDE}, to show that for almost all codimension one points $x \in U_0$ the monodromy of $\Psi_x(F)$ is conjugate to the geometric monodromy of $F$. In particular, $\Psi_x(F)$ is irreducible for almost all codimension one points $x \in U_0$. 
\end{proof}
\begin{proof}[Proof~(Thm.~\ref{simpson_THM}).]
	By Lemma \ref{bcnearbycycles_LEM}, the inertia acts by scalars on $\Psi_x(F)$ at almost all $x \in S_0$ of codimension one. The base change in Lemma \ref{bcnearbycycles_LEM} is equivariant for the inertia action. Thus the inertia acts by scalars on $s^*(F)$ at almost all points $x \in S_0$ of codimension 1. The character
	\[
	\text{det}(s^*(F))
	\] 
	is generically unramified by assumption. Hence the inertia groups at almost all points $x \in S_0$ of codimension 1 act by multiplication by roots of unity on $s^*(F)$ of order bounded by the rank of $F$. There exists $N \geq 1$ such that the roots of unity of order bounded by the rank of $F$ inject into $(\CO/\lambda^N\CO)^*$. By Lemma \ref{spreadsheaves_LEM}, we can find a dense open subset $U_0 \subseteq S_0$ such that $F_N$ extends to a local system over $X_{U_0}$. Thus the Galois representation $s^*(F)$ is generically unramified on $S_0$. We obtain the extension by Theorem \ref{extensionpre_THM}.
\end{proof}
For the arithmetic case, we require the following lemma. This lemma is presumably well-known but we provide it because we were unable to locate a reference. The special case $S_0 = \BZ[1/\ell]$ is treated in \cite[III,~Cor.2.2]{serreabelianladic}.
\begin{lemma}\label{charactersunr_LEM}Suppose $S_0$ is arithmetic. Let $\chi \colon G_k \rightarrow \Qbarl^*$ be a continuous character. The character $\chi$ is generically unramified on $S_0$.
\end{lemma}	
\begin{proof}
	We can assume, by taking the ring of integers $\CO$ in a sufficiently large extension of $\BQ_\ell$, that the character factors through a character $\chi'\colon G_k \rightarrow \CO^*$. The group $\CO^*$ is a finitely generated pro-finite abelian group. In particular, it has a finite torsion part $F$ and there is a splitting
	\[
	\CO^* = T\times F
	\]
	with a torsion-free profinite abelian group $F$ (a free $\BZ_\ell$-module of finite rank) and a finite group $T$. We denote by $\chi'_{\text{fin}}\colon  G_k \rightarrow T$ the finite part of $\chi'$. Let $x \in S_0$ be a codimension $1$ point. By the quasi-unipotent monodromy theorem (see Proposition \ref{qunimon_PROP}), the action on $\chi$ by the inertia group at $x$ is finite. In particular, the character $\chi$ is unramified at $x$ if and only if $\chi'_{\text{fin}}$ is unramified at $x$.
	
	Since the character $\chi_{\text{fin}}$ has finite monodromy, it is generically unramified. Thus $\chi$ is generically unramified.
\end{proof}
\begin{proof}[Proof~(Thm.~\ref{extensionthmWeak_THM})]
	We can pass to $S_0^{\text{red}}$ and assume $S_0$ is integral by the topological invariance of the proetale site (see (see \cite[Lem.~5.4.2]{BhattSPE})) Because $S_0$ is arithmetic and integral, it is generically regular. Hence we can assume $S_0$ is regular by passing to a dense open subscheme. The required section over an etale neighborhood of $S_0$ exists by \cite[Cor.~17.16.3~(ii)]{EGA4}. Therefore, Lemma \ref{charactersunr_LEM} allows us to apply Theorem \ref{simpson_THM} to obtain an extension of $F$ to a dense open subscheme of $S_0$. 
\end{proof}
\subsection{Perverse sheaves} We obtain the extension theorem for perverse sheaves by reducing to the case of lisse sheaves treated above. We treat the general case and then the arithmetic case. In the arithmetic case, one can discard the etale neighborhood and the twist by the character.
\begin{theorem}\label{extension_THM} Let $S_0$ be a connected Noetherian regular scheme with function field $k$. Let $X_0 \rightarrow S_0$ be a finite type morphism and $M \in \perv{X_k}$ a geometrically irreducible perverse sheaf on the generic fiber of this morphism. There exists a connected dominant etale neighborhood $S_0' \rightarrow S_0$ with function field $k'$, a character $\chi$ of the Galois group of $k'$, and a relatively perverse universally locally acylic sheaf $M_0 \in \text{Perv}^{\text{ULA}}(X_0'/S_0')$ such that the generic fiber of $M_0$ is isomorphic to $M_{k'}\otimes\chi$.

\end{theorem}
\begin{proof} There is an immersion $j\colon V_k \rightarrow X_k$ by a reduced scheme $V_{k}$, such that $(V_{\overline{k}})^{\text{red}}$ is connected and smooth of dimension $d$ and a lisse, geometrically $\Qbarl$-sheaf $F$ on $V_k$ such that $M = j_{!*}(F[d])$.  The locally closed subscheme $V_k$ is geometrically connected because $M$ is geometrically irreducible. The subscheme $V_{k}$ spreads out to a locally closed subscheme $V_0 \subseteq X_0$ by a limit argument. By Lemma \ref{relreduced_LEM}, there exists a universal homeomorphism $F\colon S_0 \rightarrow S_0$ such that $(V_0\times_{F, S_0}S_0)^{\text{red}}$ is generically smooth over $S_0$. The topological invariance of the proetale site (see \cite[Lem.~5.4.2]{BhattSPE}) allows us to base change along $F$ and pass to the reduced closed subscheme. Thus we may assume $V_0 \rightarrow S_0$ is smooth. By shrinking $S_0$ further, we can asssume $V_0 \rightarrow S_0$ has geometrically connected fibers because the generic fiber is geometrically connected.
	
	Because $V_0 \rightarrow S_0$ is smooth, \cite[Cor.~17.16.3~(ii)]{EGA4} shows that there exists an irreducible dominant etale neighborhood $S_0' \rightarrow S_0$ with a section $s\colon S_0' \rightarrow V_{S_0'}$. Let $k'$ be the function field of $S_0'$. Let $n$ be the rank of $F$. There exists a finite separable extension $k''/k'$ and a character $\chi$ of the Galois group of $k''$ such that 
	\[
	\chi^{-n} = s^*\text{det}(F).
	\]
	This implies
	\[
	s^*\text{det}(F\otimes\chi) = 1.
	\]
	Since the extension $k''/k'$ is separable, a limit argument implies that there exists a dominant etale neighborhood $S_0'' \rightarrow S_0'$ inducing the extension $k''/k'$ on function fields. Since $S_0'' \rightarrow S_0'$ is etale and $S_0'$ is regular, the scheme $S_0''$ is regular. Therefore, Theorem \ref{simpson_THM} shows, after shrinking $S_0''$, that $F\otimes\chi$ extends to a lisse $\Qbarl$-sheaf on $X_0''$. By Lemma \ref{intermed_LEM}, we can shrink $S_0''$ further and obtain $M_0 \in \text{Perv}^{\text{ULA}}(X_0''/S_0'')$ whose generic fiber is isomorphic to $M_{k''}\otimes\chi$. 
	
%
\end{proof}
\begin{theorem}\label{extensionari_THM}
	Let $S_0$ be an integral arithmetic scheme with function field $k$ and $X_0$ a $S_0$-scheme. Let $M \in \Perv{X_k}{\Qbarl}$ be a geometrically irreducible perverse sheaf.
	\begin{enumerate}
		\item  There exists a dense open subscheme $U_0 \subseteq S_0$ and $M_0 \in \text{Perv}^{\text{ULA}}(X_0/U_0)$ such that the restriction of $M_0$ to the generic fiber is isomorphic to $M$. 
		\item Suppose $U_0$ is geometrically unibranch. The extension $M_0$ is unique up to unique isomorphism. 
	\end{enumerate} 
\end{theorem}
\begin{proof}
	We can repeat the argument in Theorem \ref{extension_THM}, the only difference being that we apply Theorem \ref{extensionthmWeak_THM} instead of Theorem \ref{simpson_THM} to obtain the extension of the lisse sheaf. The unicity follows from \cite[Thm.~1.10]{hansen2023relativeperversity}.
\end{proof}
\newpage
	\section{Models} We prove existence theorems for schemes which are the bases we will apply the extension theorem to. Let us begin by stating a basic existence theorem.
	\begin{lemma}\label{modelbasic_LEM} Let $k$ be an arithmetic field. There exists a connected regular arithmetic scheme $S_0 \subseteq k$ with function field isomorphic to $k$. 
\end{lemma}
\begin{proof}
	Any finitely generated $\BZ[1/\ell]$-subalgebra of $k$ which generates $k$ is generically regular because $\BZ[1/\ell]$ is a Dedekind domain. 
\end{proof}
The following lemma is similar: 
\begin{lemma}\label{modelfields_LEM}
	Let $k$ be a field with prime field $k_0 \subseteq k$. We can write
	\[
	k = \varinjlim R_i
	\]
	as a filtered colimit of smooth $k_0$-algebras $R_i \subseteq k$. The transition morphisms are given by the inclusions.
\end{lemma}
\begin{proof}
	Note that, since $k_0$ is a prime field, it is perfect. In particular, if a finite type algebra over $k_0$ is regular, then it is smooth. Now we may apply the same argument as in Lemma \ref{modelbasic_LEM}.
\end{proof}
A ``geometric" discrete valuation ring is always a local ring of a closed point on a curve: 
\begin{lemma}\label{geomdvr_LEM}
	Let $S$ be a discrete valuation ring essentially of finite type over a field $k_0$. There exists a finitely generated extension $k_1/k_0$, a regular geometrically connected curve $C$ over $k_1$, and a closed point $0 \in C$ such that there exists an isomorphism
	\[
	S \cong \CO_{C,0}. 
	\]
\end{lemma}
\begin{proof} By definition, we can find a normal domain $R$ of finite type over $k_0$ and a prime ideal $\idp \subset R$ of height one with $R_\idp = S$. After localizing $R$, we may assume there is $\pi \in R$ such that $\idp = (\pi)$. Let $\idm \subset R$ be a maximal ideal with $\idp \subset \idm$, then we may find a system of parameters $x_1, \ldots, x_n \in \idm$ with $x_1 = \pi$. We then consider the morphism
	\[
	k_0[x_1, \ldots,x_n] \rightarrow R.
	\]
We consider the prime ideal $\idq := (x_1) \subset k_0[x_1, \ldots, x_n]$. Then we obtain a quasi-finite morphism
\[
k_0[x_1, \ldots, x_n]_\idq \rightarrow R_\idp
\]
because $\idp$ is the only prime ideal lying over $\idq$ by assumption. We have
\[
k_0[x_1, \ldots, x_n]_{\idq} = k_0(x_2, \ldots,x_n)[x_1]_{(x_1)}.
\]
Since $R_\idp$ is quasi-finite over this local ring, we see that it is also the localization of a normal curve by Zariski's main theorem. We can then normalize the field $k_0(x_2, \ldots,x_n)$ in $R_\idp$ to obtain a geometrically connected regular curve over a field finitely generated over $k_0$.  
	
%
%
\end{proof}
\begin{corollary}\label{strhenscurves_COR}
	Let $S$ be the spectrum of the strict henselization of a discrete valuation ring essentially of finite type over field and $\pi \in \Gamma(S, \CO_S)$ a uniformizer. We can write
	\[
	S = \varprojlim C_i
	\]
	as a cofiltered limit of geometrically connected regular curves $C_i$ over an arithmetic field with an element $\pi \in \Gamma(C_i, \CO_{C_i})$.
\end{corollary}
\begin{proof}
	Lemma \ref{geomdvr_LEM} implies $S$ is the strict henselization of the local ring of an affine regular geometrically connected curve $C$ over a field $k$ at a point $0 \in C$. We may write $k = \varinjlim k_i$ as the filtered union over all arithmetic subfields $k_i \subseteq k$. By passing to a subsystem, this permits us to write
	\[
	C = \varprojlim C_i
	\]
	as the cofiltered limit of geometrically connected regular affine curves $C_i$ over $k_i$ with distinguished points $0 \in C$. For each $i \in I$, we may write
	\[
	\text{Spec}(\CO_{C_i, 0}^{\text{sh}}) = \varprojlim S_{i, j} 
	\]
	as the cofiltered limit over all connected affine etale neighborhoods $S_{i, j} \rightarrow C_i$ of $0$. Since this limit is functorial and the strict henselization commutes with filtered colimits (see \cite[Prop.~2.8.19,~(ii)]{LeiFu}), we obtain
	\[
	S = \varprojlim_{i, j} S_{i, j}.
	\] 
	Since we have $\pi \in \Gamma(S, \CO_S)$, we may assume $\pi \in \Gamma(S_{i, j}, \CO_{S_{i, j}})$ by passing to a subsystem. Thus we obtain the claim.

\end{proof}
\newpage
	\section{Decomposition and hard Lefschetz} The goal of this section is to prove the decomposition theorem and the Hard Lefschetz theorem for arithmetic complexes. We prepare the proofs by giving a criterion for a complex to be semisimple on the geometric generic fiber in terms of its specializations at the geometric points supported at a dense set of points.
	\subsection{Specializing semisimplicity} We provide a criterion for the geometric semisimplicity of a complex in terms of its specializations. We begin by describing a criterion for nullity of a morphism of complexes on the geometric generic fiber in terms of its geometric specializations at closed points. We then apply the filtrations from Theorem \ref{filtration_THM} to obtain a criterion for semisimplicity in terms of its specializations.
	\begin{lemma}\label{morphismzero_LEM}
		Let $S_0$ be an integral qcqs scheme with function field $k$. Let $X_0$ be an $S_0$-scheme and $K, L \in \derCat{c}{X_0}{\Qbarl}$ complexes admitting a morphism $K \rightarrow L$. The morphism $K_{\overline{k}} \rightarrow L_{\overline{k}}$ is zero if and only if there exists a dense subset $U_0 \subseteq S_0$ such that $K_{\overline{x}} \rightarrow L_{\overline{x}}$ is zero for all points $x \in U_0$. Moreover, if $K_{\overline{k}} \rightarrow L_{\overline{k}}$ vanishes, there exists a dense open subset $V_0 \subseteq X_0$ such that $K_{\overline{x}} \rightarrow L_{\overline{x}}$ vanishes for all $x \in V_0$. 
	\end{lemma}
	\begin{proof}
		By Theorem \ref{genbc_THM} and Theorem \ref{genbchom_THM}, there exists a dense open subscheme $V_0 \subseteq Y_0$ such that
		\begin{enumerate}
			\item The complex $\text{Hom}(K, L)$ is constructible over $V_0$. 
			\item The formation of $\text{Hom}(K, L)$ commutes with base change $Y \rightarrow V_0$.
			\item The formation of the pushforward $\pi_{*}(\text{Hom}(K, L))$ commutes with base change $Y \rightarrow V_0$.
			\item The complex $\pi_*(\text{Hom}(K, L))$ is lisse over $V_0$. 
		\end{enumerate}
		For all points $x \in V_0$, the specialization morphism of $\pi_{*}(\text{Hom}(K, L))$ induces an isomorphism
		\[
		\text{Hom}_{\derCat{c}{X_{\overline{x}}}{\Qbarl}}(K_{\overline{x}}, L_{\overline{x}})  \rightarrow \text{Hom}_{\derCat{c}{X_{\overline{k}}}{\Qbarl}}(K_{\overline{k}}, L_{\overline{k}})  
		\]
		This isomorphism maps the morphism $K_{\overline{x}} \rightarrow L_{\overline{x}}$ to the morphism $L_{\overline{k}} \rightarrow K_{\overline{k}}$. 
		
		In particular, if $K_{\overline{k}} \rightarrow L_{\overline{k}}$ is zero, there exists the dense open subset $U_0 := V_0$ such that the fiber of the  morphism is zero over all closed points in $V_0$. 
		
		On the other hand, suppose we are given a dense subset $U_0 \subseteq S_0$ such that $K_{\overline{x}} \rightarrow L_{\overline{x}}$ is zero over all closed points $x \in U_0$. The subset $U_0\cap V_0 \subseteq S_0$ is non-empty, so we may choose $y \in U_0\cap V_0$. The morphism $K_{\overline{y}} \rightarrow L_{\overline{y}}$ is zero by assumption, so we can evaluate the specialization map to show  $K_{\overline{k}} \rightarrow L_{\overline{k}}$ is zero. 
	\end{proof}	
	For the proof of the Hard Lefschetz theorem, we note the following corollary of the above lemma. 
	\begin{corollary}\label{isocrit_COR}
		Let $S_0$ be an integral qcqs scheme with function field $k$ and $X_0$ an $S_0$-scheme. Let $K, L \in \derCat{c}{X_0}{}$ and suppose there is a morphism $K \rightarrow L$. The morphism $K_{{k}} \rightarrow L_{{k}}$ is an isomorphism if and only if there exists a dense subset $U_0 \subseteq S_0$ such that $K_{{x}} \rightarrow L_{x}$ is an isomorphism. If $K_k \rightarrow L_k$ is an isomorphism, there exists a dense open subset $U_0 \subseteq S_0$ such that $K_x \rightarrow L_x$ is an isomorphism for all $x \in U_0$. 
	\end{corollary}
	\begin{proof}
		This statement can be obtained by applying Lemma \ref{morphismzero_LEM} to the identity morphism of the cone of $K \rightarrow L$. 
	\end{proof}
	\begin{lemma}\label{critsemisimpleperv_LEM}
		Let $S_0$ be an integral geometrically unibranch topologically Noetherian scheme with function field $k$. Let $X_0 \rightarrow S_0$ be a morphism of finite type. Let $M \in \Perv{X_0/S_0}{\Qbarl}$. Then $M_{\overline{k}}$ is semisimple if and only if there is a dense subset $U_0 \subseteq S_0$ such that for all $x \in U_0$ the restriction to the geometric fiber $M_{\overline{x}}$ is semisimple. Moreover, if the the generic fiber $M_{\overline{k}}$ is semisimple, there exists a dense open subset $U_0 \subseteq X_0$ such that $M_{\overline{x}}$ is semisimple for all $x \in U_0$.
	\end{lemma}
	\begin{proof} There exists a finite separable extension $k'/k$ such that the irreducible components of $M_{k'}$ have geometrically irreducible support. By a limit argument, we can find a dominant connected etale neighborhood $S_0' \rightarrow S_0$ inducing the extension $k'/k$ on function fields.  We can base change to $S_0'$ without changing the statement. Thus we may assume the decomposition factors of $M_k$ have geometrically irreducible support.
		
		We apply Lemma \ref{filtration_THM}. Thus we obtain a dense open subset $V_0 \subseteq S_0$ and a filtration
		\[
		0 = M_0 \subset M_1 \subset \ldots \subset M_{n - 1} \subset M_n = M.
		\] 
		by relatively perverse sheaves $M_i \in \text{Perv}(X_{V_0}/V_0)$. The generic geometric fiber of a graded piece $M_i/M_{i - 1}$ for $1 \leq i \leq n$ is geometrically semisimple if and only if there exists a dense subset $U_0 \subseteq S_0'$ such that the fiber over a geometric point supported in $U_0$ is geometrically simple by properties (5) and (6) of the filtration. Moreover, the connecting morphisms
		\[
		M_{i}/M_{i - 1} \rightarrow M_{i - 1}[1]
		\]
		are zero on the geometric generic fiber if and only if there exists a dense subset such that the fiber is zero over all geometric points supported at the closed points by Lemma \ref{morphismzero_LEM}. The argument for the existence of a dense open subset such that semisimplicity holds at all geometric fibers is the same.
	\end{proof}
	\begin{proposition}\label{critsemisimple_PROP}
	Let $S_0$ be an integral geometrically unibranch topologically Noetherian scheme with function field $k$. Let $X_0 \rightarrow S_0$ be a finite type morphism. Let $K \in \derCat{c}{X_0}{}$. Then $K_{\overline{k}}$ is semisimple if and only if there is a dense subset $U_0 \subseteq S_0$ such that for all points $x \in U_0$ the restriction to the geometric fiber $K_{\overline{x}}$ is semisimple. If the generic fiber $K_{\overline{k}}$ is semisimple, then there exists a dense open subset $V_0 \subseteq S_0$ such that $K_{\overline{x}}$ is semisimple for all $x \in V_0$.
\end{proposition}
\begin{proof}
	We can apply Lemma \ref{morphismzero_LEM} to the connecting morphisms
	\[
	{}^{\idp/S_0}\mathscr{H}^{i + 1}(K)[i + 1] \rightarrow {}^{\idp/S_0} \tau^{\leq i}(K)[1]
	\]
	for $i \in \BZ$ to prove that we have a decomposition
	\[
	K_{\overline{k}} = \bigoplus_{i \in \BZ} \pervCoh{i}{K_{\overline{k}}}[-i]
	\]
	over the generic point if and only if there exists a dense subset $U_0\subseteq S_0$ such that 
	\[
	K_{\overline{x}} = \bigoplus_{i \in \BZ} \pervCoh{i}{K_{\overline{x}}}[-i]
	\]
	for all points $x \in U_0$. We can then apply Lemma \ref{critsemisimpleperv_LEM} to the relative perverse cohomology sheaves to obtain the statement. The argument for the existence of a dense open subset such that semisimplicity holds at all points in the dense open subset is the same. 
\end{proof}
	\subsection{Decomposition theorem} We prove the decomposition theorem for arithmetic complexes. 
	\begin{theorem}\label{decomposition1_THM}
		Let $k$ be a separably closed field. Let $X$ and $Y$ be $k$-schemes and $f\colon X \rightarrow Y$ a proper morphism of $k$-schemes. Let $K \in \derCat{c}{X}{\Qbarl}$ be a semisimple arithmetic complex. The pushforward $f_*(K)$ is semisimple.
	\end{theorem}
	\begin{proof} Note that $K$ splits into a direct sum of shifted irreducible arithmetic perverse sheaves by Theorem \ref{subquotientsarithmetic_THM}. Thus we can assume $K$ is an irreducible perverse sheaf. The schemes, the morphism, and the perverse sheaf descend to an arithmetic field by a limit argument and Corollary \ref{models_COR}. Hence we can assume the schemes $X$ and $Y$, the morphism $X \rightarrow Y$, and the complex $K$ are defined over an arithmetic field $k_0$.
	
	We first consider the case when $k$ is the algebraic closure of a finite field. In this case, $k_0$ is finite and $K$ is an irreducible perverse sheaf on $X$. By Lafforgue's theorem \cite[Cor.~VII.8]{LafforguePurity}, we can twist $K$ by a character of $k$ and assume the perverse sheaf $K$ is pure of weight zero. By \cite[Stab.~5.1.14]{BBD}, the pushforward $f_*(K)$ is pure of weight zero. Thus \cite[Cor.~5.4.6]{BBD} implies that $f_*(K)$ is geometrically semisimple.  
			
	We turn to the general case. There is an irreducible regular arithmetic scheme $S_0$ with function field isomorphic to $k_0$ by Lemma \ref{modelbasic_LEM}. After possibly shrinking $S_0$, we can find a diagram with Cartesian squares
		\[
		\begin{tikzcd}
			X \arrow[d] \arrow[r, "f"] & Y \arrow[r] \arrow[d] & k_0 \arrow[d] \\
			X_0 \arrow[r, "f"]         & Y_0 \arrow[r]         & S_0         
		\end{tikzcd}
		\] 
		such that $X_0 \rightarrow Y_0$ is proper, $k_0 \rightarrow S_0$ is the inclusion of the function field of $S_0$, and $Y_0$ is an $S_0$-scheme. After shrinking $S_0$,  Theorem \ref{extensionari_THM} shows there exists a complex $K_0 \in \derCat{c}{X_0}{\Qbarl}$  such that $K_{{k}}$ is isomorphic to $K$. After shrinking $S_0$, Proposition \ref{critsemisimple_PROP} shows we may assume $K_x$ is geometrically semisimple for all points $x \in S_0$.
		
		The residue field at a closed point in $S_0$ is finite.  Note that $K_x$ is defined over a finite field, hence $K_{\overline{x}}$ is arithmetic. Therefore, the pushforward $f_*(K_x)$ is geometrically semisimple for all closed points $x \in S_0$ by the case of finite fields already treated. By proper base change, Proposition \ref{critsemisimple_PROP} implies $f_*(K)$ is geometrically semisimple because the scheme $S_0$ is Jacobson.
	\end{proof}
	We note that the decomposition theorem can be generalized as follows using the extension theorem.
	\begin{theorem}\label{decomposition2_THM}
		Let $k$ be a field finitely generated over a separably closed field. Let $K \in \derCat{c}{X}{}$ be an arithmetic semisimple complex and $f\colon X \rightarrow Y$ a proper morphism. The pusforward $f_*(K)$ is semisimple.
	\end{theorem}
	\begin{proof} 
		We may assume $K$ is an irreducible perverse sheaf since $K$ is semisimple. This implies $K_{\overline{k}}$ is semisimple. After passing to a finite extension of $k$, Theorem \ref{geomsimpleperverse_THM} implies that there exist geometrically irreducible arithmetic perverse sheaves $M_1, \ldots, M_n \in \perv{X_{k'}}$, and continuous $\Qbarl$-representations $\rho_1, \ldots,\rho_n$ of the Galois group of $k$ such that
		\[
		K \cong (M_1\otimes\rho_1) \oplus \ldots \oplus (M_n\otimes \rho_n).
		\]
		Note that $M_i\otimes \rho_i$ is a direct factor of a semisimple perverse sheaf, hence it is semisimple. Moreover, any perverse subsheaf $M \subseteq M_i\otimes\rho_i$ is given by $M = M_i\otimes \rho'$ for 
		\[
		\rho' = \text{Hom}(M_i, M).
		\]
		Thus we obtain that $\rho_i$ is semisimple. This also shows that the tensor product of a semisimple perverse sheaf with a semisimple representation of the Galois group of $k$ is semisimple. By the projection formula, we may pass to $M_i$ for some $1 \leq i \leq n$ and assume $K$ is a geometrically irreducible perverse sheaf. 
		
		Let $k_0 \subset k$ be a separably closed subfield, such that $k$ is finitely generated over $k_0$ and $k_0$ admits no finite algebraic extension in $k$ (i.e. any polynomial irreducible in $k_0$ remains irreducible in $k$). This implies $\overline{k}_0\otimes_{k_0}k$ is a field. Since $\overline{k}_0/k_0$ is purely inseparable, we may base change to $\overline{k}_0k$ and assume $k_0$ is algebraically closed. Our goal now is to write $k$ as the function field of a smooth $k_0$-scheme and then spread out $K$ over this smooth $k_0$-scheme. However, note that we want the spreading of $K$ to be arithmetic, so that we may apply the decomposition theorem for arithmetic complexes to the pushforward. This is why we will first descend $K$ to a scheme over an arithmetic field, spread it out there, and then pull the spread out complex back to $k_0$. This way we ensure that the spread out complex is arithmetic.		
				
		We may write
		\[
		k_0 = \varinjlim k_i
		\]
		as a filtered colimit over all arithmetic subfields $k_i \subset k_0$. By passing to a subsystem, we may assume that there are finitely generated extensions $k_i'/k_i$ with
		\[
		k_i'\otimes_{k_i}k_0 = k
		\] 
		and $k_i' \subset k$. After twisting $K$ by a character of the Galois group of $k$, Theorem \ref{modelsIrr_THM} allows us to assume there are schemes $X_i$ and $Y_i$ over $k'_i$ such that $X_i\times_{k'_i}k = X$, $Y_i\times_{k'_i} k = Y$, a perverse sheaf $M \in \perv{X_i}$ with $M_k = K$, and a proper morphism
		\[
		X_i \rightarrow Y_i
		\]
		reproducing the morphism $X \rightarrow Y$ after base change. Since $K$ is a geometrically irreducible perverse sheaf, we see that $M$ is a geometrically irreducible perverse sheaf. 
		
		Since $k_0$ is algebraically closed, the extension $k/k_0$ is separable. Thus $k'_i/k_i$ is separable. In particular, there is a smooth $k_i$-scheme $S_i$ with function field $k'_i$. After shrinking $S_i$, we may find a proper morphism of $S_i$-schemes $X_{ii} \rightarrow Y_{ii}$ with generic fiber isomorphic to $X_i \rightarrow Y_i$. After shrinking $S_i$, Theorem \ref{extensionari_THM} shows there exists a relatively perverse universally locally acyclic sheaf $M_i \in \text{Perv}^{\text{ULA}}(X_{ii}/S_{ii})$ with generic fiber $M_{k'_i} = M$. 
		
		We put $X_0 := X_{ii}\times_{k_i} k_0$, $Y_0 := Y_{ii}\times_{k_i} k_0$, $S_0 := S_{ii}\times_{k_i} k_0$, and $M_0 := (M_i)_{k_0}$. These schemes fit into a diagram
		\[
		\begin{tikzcd}
			X \arrow[d] \arrow[r, "f"] & Y \arrow[r] \arrow[d] & k \arrow[d] \\
			X_0 \arrow[r, "f"]         & Y_0 \arrow[r]         & S_0         
		\end{tikzcd}
		\]
		with Cartesian squares. Let $d$ be the dimension of $S_0$. Since $S_0$ is smooth over $k_0$, Lemma \ref{irred_LEM} implies $M_0[d]$ is an irreducible perverse sheaf on $X_0$. Thus $M_0$ is a semisimple complex. Note that $M_0$ descends to the arithmetic field $k_i$, so $M_0$ is arithmetic. Therefore, Theorem \ref{decomposition1_THM} implies $f_*(M_0)$ is a semisimple complex. Lemma \ref{irred_LEM} and proper base change imply that $f_*(K)$, the restriction of $f_*(M_0)$ to the generic fiber over $S_0$, is semisimple. 

	\end{proof}
	\subsection{Hard Lefschetz theorem} We prove the Hard Lefschetz theorem for geometrically semisimple arithmetic perverse sheaves.
	\begin{theorem}\label{hardlefschetz_THM}
		Let $f\colon X \rightarrow Y$ be a projective morphism of $k$-schemes. Let $\eta \in H^2(X, \Qbarl(1))$ be the Chern class of a $f$-ample line bundle on $X$ and $M \in \perv{X}$ a geometrically semisimple arithmetic perverse sheaf. The morphism
		\[
		\eta^i\colon \pervCoh{-i}{f_*(M)} \rightarrow \pervCoh{i}{f_*(M)}(i)
		\]
		is an isomorphism.
	\end{theorem}
	\begin{proof}
		We can pass to the algebraic closure of $k$ without changing the statement, thus we can assume $k$ to be algebraically closed. Because $M$ is semisimple, we can assume $M$ is irreducible by Theorem \ref{subquotientsarithmetic_THM}. 
		
		As in the proof of the decomposition theorem, we can arrive at a diagram
		\[
		\begin{tikzcd}
			X \arrow[d] \arrow[r, "f"] & Y \arrow[r] \arrow[d] & k_0 \arrow[d] \\
			X_0 \arrow[r, "f"]         & Y_0 \arrow[r]         & S_0         
		\end{tikzcd}
		\] 
		such that there exists $M_0 \in \text{Perv}^{\text{ULA}}(X_0/S_0)$ over $S_0$ with restriction to the generic fiber isomorphic to $M$, the $f$-ample line bundle on $X$ spreads out to a $f$-ample line bundle on $X_0$. The $f$-ample line bundle gives rise to a morphism $\eta\colon \Qbarl \rightarrow \Qbarl(1)[2]$ in the derived category of proetale sheaves on $X_0$ such that the fiber of this morphism over each point of $S_0$ equals the Chern class of the $f$-ample line bundle. We consider the induced morphism on relative perverse cohomology sheaves
		\[
		\eta^i\colon ^{\idp/S_0}\mathscr{H}^{-i}(f_*(M_0)) \rightarrow{} ^{\idp/S_0}\mathscr{H}^i(f_*(M_0))(i).
		\]
		By Corollary \ref{isocrit_COR}, this morphism is an isomorphism on the generic fiber if and only if it is an isomorphism at a dense set of closed points. Thus we can reduce the Hard Lefschetz theorem to the case of finite fields (see \cite[Thm.~5.4.10]{BBD}) as in the proof of the decomposition theorem.
	\end{proof}
\newpage
	\section{The Arithmetic Kashiwara Conjecture} We prove the arithmetic Kashiwara conjecture for ``geometric"  traits. The proof consists of three steps.
	
	We first treat the arithmetic Kashiwara conjecture for perverse sheaves with models over an arithmetic field. In this case, we can spread out the perverse sheaf to a family over an arithmetic base scheme using the extension theorem. We then construct nearby cycles in families. The generic base change and constructibility theorem for generalized coefficient rings plays a crucial role because it allows us to show that nearby cycles vary ``constructibly" in families (see Theorem \ref{prolfamily_THM}). Similarly to the previous proofs, we can then reduce to a finite base field by passing to a fiber over a closed point. 

	\subsection{Nearby cycles in families} 
	
	The goal of this section is to study pro-$\ell$ nearby cycles in depth. This construction allows us to consider nearby cycles in families in Theorem \ref{prolfamily_THM}.
	\begin{definition}
		We define the affine pro-algebraic group scheme
		\[
		I_\ell := \varprojlim \mu_{\ell^n}
		\]
		over $\BZ[1/\ell]$. We denote by $\underline{I_\ell}$ the proetale sheaf representing $I_\ell$. 
		
		Let $X_0$ be a qcqs scheme and $R$ a sheaf of rings in the proetale topology on $X_0$. Let $K \in \text{D}(X_0, R)$. We say $I_\ell$ \textit{acts} on $K$ if there is a group morphism
		\[
		\underline{I_\ell} \rightarrow \text{Aut}_R(K)
		\]
		of proetale sheaves on $X_0$. We say the action by $I_\ell$ on $K$ is \textit{quasi-unipotent} if for each weakly etale neighborhood $X_0' \rightarrow X_0$, and each $t \in I_\ell(X_0')$, the action by $t$ on $K_{X_0'}$ is quasi-unipotent.
	\end{definition}
	\begin{remark}
		Although the sheaf $\underline{I_\ell}$ is naturally isomorphic to $\BZ_\ell(1)$, we prefer to keep the notation above to indicate an inertia group.
		
		Furthermore, we remark that an action by $I_\ell$ is automatically continuous. Roughly speaking, any proetale sheaf is also a sheaf of condensed sets. The action by $I_\ell$ is always an action by a condensed sheaf of group on a condensed sheaf. Thus it is continuous. 
	\end{remark}
	We prove the following variant of the quasi-unipotent monodromy theorem: 
	\begin{proposition}\label{qunimon_PROP}
		Let $X_0$ be a topologically Noetherian scheme. We assume that no finite extension of any residue field of any point of $X_0$ contains all $\ell$'th roots of unity. Let $K \in \consder{X_0}$ be a complex with an $I_\ell$-action.
		\begin{enumerate}
			\item The action by $I_\ell$ is quasi-unipotent. 
			\item There exists a functorial nilpotent monodromy operator
			\[
			N\colon K(1) \rightarrow K.
			\]
		\end{enumerate}
	\end{proposition}
	\begin{proof}
		
		By devissage, we may assume $K$ is a locally constant sheaf $F$ on $X_0$. By a further devissage, we may assume $X_0$ is integral (see \cite[Cor.~2.3.4]{hansen2025arithmeticpropertieselladicetale} for more details on the devissage). By \cite[Rem.~7.4.8]{BhattSPE}, lisse proetale sheaves on $X_0$ are classified by the proetale fundamental group. We can find a connected weakly etale torsor
		\[
		X_0' \rightarrow X_0
		\] 
		which trivializes the lisse sheaf $\underline{I_\ell}$ by adding all the $\ell$'th roots of unity to $X_0$ and passing to a connected component.\footnote{One could also show that $X_0'$ is topologically Noetherian. If $X_0$ is Noetherian, then $X_0'$ is Noetherian as well (see \cite[\href{https://stacks.math.columbia.edu/tag/05KG}{Tag 05KG}]{stacks-project}).} Since $X_0'$ is a torsor over a topologically Noetherian scheme, \cite[Rem.~7.4.8]{BhattSPE} implies that lisse $\Qbarl$-sheaves on $X_0'$ are classified by the proetale fundamental group. In particular, a morphism of lisse sheaves on $X_0'$ is zero if and only if it is zero at a point. Thus we may pass to a fiber of $F$ over a point in $X_0$ and assume $X_0$ is a field. 
				
		In this case, $F$ is given by a continuous finite-dimensional Galois representation $\rho$ of $k$ such that no finite extension of $k$ admits all $\ell$'th roots of unity and we are given an morphism of proetale sheaves
		\[
		\underline{I_\ell} \rightarrow \text{Aut}(\rho). 
		\]
		We add all the $\ell$'th roots of unity to $k$ and obtain an extension $k'/k$. We have an injection
		\[
		\text{Gal}(k'/k) \rightarrow (\BZ_\ell)^*
		\]
		The image is open: the group $(\BZ_\ell)^*$ is a $\BZ_\ell$-module of rank one and finite torsion part. The image of the Galois group is infinite and closed, by compactness of the Galois group. Thus it must be a $\BZ_\ell$-submodule of $(\BZ_\ell)^*$. Because the image is infinite, it must have rank one as a $\BZ_\ell$-module. This implies that the image is open. In particular, we note that there is $n \geq 1$ such that $t := 1 + \ell^n$ lies in the Galois group of $k'/k$. 
		
		We choose a trivialization $I_\ell|_{k'} \cong \BZ_\ell$. This allows us to construct a morphism of groups
		\[
		\alpha\colon \BZ_\ell \rightarrow \text{Aut}(\rho|_{k'}).
		\]
		We have a commutative diagram
		\[
		\begin{tikzcd}
			\BZ_\ell \arrow[d, "t"] \arrow[r] & \text{Aut}(\rho|_{k'}) \arrow[d, "t"] \\
			\BZ_\ell \arrow[r]                & \text{Aut}(\rho|_{k'}).               
		\end{tikzcd}
		\]
		The morphism on the right is given by conjugation by $t$ and the morphism on the left is multiplication by $t$. Thus we obtain the relation
		\[
		t\alpha(1) = \alpha(t)t. 
		\]
		Let $\lambda \in \Qbarl^*$ be an eigenvalue of $\alpha(1)$ and $x \in \rho$ an eigenvector. Then we have
		\[
		\lambda (tx) = \alpha(t)(tx).
		\] 
		We see that $\lambda$ is an eigenvalue of $\alpha(t) = \alpha(1)^{1 + \ell^n}$. Thus $\lambda^{1/1 + \ell^n}$ is an eigenvalue of $\alpha(1)$. Since $\rho$ is finite-dimensional, there are $k, k'> 0$ with
		\[
		\lambda^{1/(1 + \ell^n)^k} = \lambda^{1/(1 + \ell^n)^{k + k'}}
		\] 
		Thus $\lambda$ is a root of unity. Since the action of $I_\ell$ is continuous, we see that all elements of $\BZ_\ell$ act by roots of unity by considering eigenspaces in an isotypical decomposition of $\rho$ under $I_\ell$. 
		
		We define a morphism of proetale sheaves
		\[
		\Tilde{N}\colon \underline{I_\ell} \rightarrow \text{End}_R(K, K)
		\]
		as follows: Let $X_0' \rightarrow X_0$ be a weakly etale neighborhood of $X_0$. Let $t \in \underline{I_\ell}(X_0')$. The action by $\rho(t)$ is quasi-unipotent on $K|_{X_0'}$ by (1), so we can put
		\[
		\Tilde{N}(t) := \frac{\text{log}(\rho(nt))}{n}
		\]
		for a sufficiently large natural number $n \geq 1$ such that $\rho(nt)$ is unipotent. This defines a group morphism.
		
		We can now tensor the above morphism
		\[
		\underline{I_\ell}\otimes_{\BZ_\ell} \text{Hom}_{\BZ_\ell}(\underline{I_\ell}, \BZ_\ell) \rightarrow \text{End}_{\Qbarl}(K, K)\otimes_{\BZ_\ell} \text{Hom}_{\BZ_\ell}(\underline{I_\ell}, \BZ_\ell). 
		\]
		Tensor product functoriality induces an isomorphism
		\[
		\text{End}_{\Qbarl}(K, K)\otimes_{\BZ_\ell} \text{End}_{\BZ_\ell}(\underline{I_\ell}, \BZ_\ell) \rightarrow \text{End}_{\Qbarl}(K\otimes_{\BZ_\ell} \underline{I_\ell}, K)
		\]
		Note that we have an isomorphism
		\[
		\underline{I_\ell} = \text{End}_{\BZ_\ell}(\BZ_\ell, \underline{I_\ell}).
		\]
		Thus we obtain
		\begin{align*}
		\underline{I_\ell}\otimes_{\BZ_\ell} \text{Hom}_{\BZ_\ell}(\underline{I_\ell}, \BZ_\ell) &= 	\text{End}_{\BZ_\ell}(\BZ_\ell, \underline{I_\ell})\otimes_{\BZ_\ell} \text{Hom}_{\BZ_\ell}(\underline{I_\ell}, \BZ_\ell) \\
		&= \text{Hom}_{\BZ_\ell}(\underline{I_\ell}, \underline{I_\ell})
		\end{align*}
		We may compose and obtain the morphism
		\[
		\Tilde{\Tilde{N}}\colon \text{Hom}_{\BZ_\ell}(\underline{I_\ell}, \underline{I_\ell})\rightarrow \text{Hom}_{\Qbarl}(K\otimes_{\BZ_\ell}\underline{I_\ell}, K).
		\]
		We define
		\[
		N := \Tilde{\Tilde{N}}(X_0)(\text{id}).
		\]
		This defines the nilpotent monodromy operator
		\[
		N\colon K(1) \rightarrow K
		\]
		because $K\otimes_{\BZ_\ell}\underline{I_\ell} = K(1)$ by definition of the Tate twist. 
	\end{proof}
	We can find objects with $I_\ell$-actions by forming pro-$\ell$ nearby cycles: 
	\begin{definition}		
		Let $S_0$ be the spectrum of a discrete valuation ring and $\pi \in S_0$ a uniformizer. We define
		\[
		\Tilde{S}_0 := S_0\big[\pi^{1/\ell^n}~|~ n \geq 1\big].
		\]
		Let $\Tilde{i}\colon s \rightarrow \Tilde{S}_0$ be the inclusion of the special fiber and $\Tilde{j}\colon \Tilde{\eta} \rightarrow \Tilde{S}_0$ the inclusion of the complement. 
		
		Let $R$ be either a complete Noetherian $I$-adic ring or $R = \Qbarl$. Let $X_0$ be a $S_0$-scheme and $K \in \text{D}(X_\eta, R)$ an $I$-derived complete complex of $R$-modules on the generic fiber. We define the \textit{pro-$\ell$ nearby cycles}
		\[
		\Psi^\ell_{\pi}(K) := \tilde{i}^*\tilde{j}_*\tilde{j}^*(K_{\Tilde{V}_0}) \in \text{D}(X_s, R).
		\]
		The pro-$\ell$ nearby cycles $\Psi^\ell_{\pi}(K)$ come with an action by $I_\ell$. Note that they are automatically $I$-derived complete because we take the completed pullback $\tilde{i}^*$ (see Remark \ref{genbcprep_REM}).
	\end{definition}
	\begin{remark}\label{prolnc_REM}
		When $S_0$ is a strict trait, the above definition computes the part of the nearby cycles on which the inertia acts through the maximal pro-$\ell$ quotient. The Kummer map provides an isomorphism $I_{S_0}^\ell = I_\ell(s)$ since $s$ is separably closed. The $I_\ell$-action on the pro-$\ell$ nearby cycles identifies with the action of the maximal pro-$\ell$ quotient of inertia on the nearby cycles along this isomorphism.
				
		In general, the formation of the pro-$\ell$ nearby cycles commutes with base change along weakly etale morphisms by weakly etale base change. In particular, the pro-$\ell$ nearby cycles compare to the pro-$\ell$ nearby cycles on the strict henselization. This relates the pro-$\ell$ nearby cycles to the usual nearby cycles. For example, this implies that the pro-$\ell$ nearby cycles are constructible when the nearby cycles are constructible over the strict henselization. 
	\end{remark}
	We have the following base-change:
	\begin{lemma}\label{prolncbc_LEM}Let $S$ be the spectrum of a discrete valuation ring. We assume one of the following two conditions on the coefficient ring $R$:
		\begin{enumerate}
			\item The ring $R$ is Noetherian and complete with respect to an ideal $I \subset R$. The characteristic of $R/I$ is invertible on $S$.
			\item The ring $R$ is $R = \Qbarl$ for a prime $\ell$ invertible on $S$.
		\end{enumerate}
		Let $S' \rightarrow S$ be a morphism of spectra of discrete valuation rings. We assume there is a uniformizer $\pi \in \Gamma(S, \CO_S)$ mapping to a uniformizer $\pi' \in \Gamma(S', \CO_{S'})$. Let $X$ be an $S$-scheme and $K \in \text{D}^b_c(X_\eta, R)$ a constructible complex on the generic fiber of $X$ over $S$. We denote by $\pi\colon X_{S'} \rightarrow X$ the projection induced by base change. We have an $I_\ell$-equivariant natural isomorphism
		\[
		\pi^*(\Psi^\ell_{\pi}(K)) = \Psi^\ell_{\pi'}(\pi^*(K)). 
		\]
	\end{lemma}
	\begin{proof} Note that we can always construct a natural morphism 
		\[
		\pi^*(\Psi^\ell_{\pi}(K)) \rightarrow \Psi^\ell_{\pi'}(\pi^*(K)).
		\]
		To prove this morphism is an isomorphism, we may pass to the strict henselization of $S$ and $S'$ by smooth base change. Thus we may assume $S$ and $S'$ are strict traits.
	
	Let $I^{(\ell)}_{S}$ and $I^{(\ell)}_{S'}$ be the kernel of the morphism $I \rightarrow I^\ell$ from the inertia group to the maximal pro-$\ell$ quotient for the inertia of $S$ and $S'$ respectively. We can now write
	\begin{align*}
	\Psi^\ell_{\pi}(K) &= \Psi(K)^{I^{(\ell)}_{S}} \\
	\Psi^\ell_{\pi'}(\pi^*(K)) &= \Psi(\pi^*(K))^{I^{(\ell)}_{S'}}.
	\end{align*}
	It follows from the change of trait isomorphism that we have an isomorphism
	\[
	\pi^*(\Psi(K))^{I^{(\ell)}_{S}} = \Psi(\pi^*(K))^{I^{(\ell)}_{S'}}
	\]
	This isomorphism identifies with the morphism constructed above. Thus we obtain the desired isomorphism.
	\end{proof}
	\begin{theorem}\label{prolfamily_THM}
		Let $S_0$ be an irreducible qcqs scheme. We assume one of the following two conditions on the coefficient ring $R$:
		\begin{enumerate}
			\item The ring $R$ is a Noetherian ring complete with respect to an ideal $I \subset R$. We assume the characteristic of $R/I$ is invertible on $S_0$.
			\item The ring $R$ is $R = \Qbarl$ for a prime $\ell$ invertible on $S_0$. 
		\end{enumerate}
		Let $C_0$ be a finitely presented smooth $S_0$-scheme with geometrically connected fibers of dimension one and $\pi \in \Gamma(C_0, \CO_{C_0})$ a global section of the structure sheaf of $C_0$. We define $Y_0 :=V(\pi)$ to be the zero-locus of $\pi$. We assume $Y_0$ is an relative normal crossing divisor with connected fibers over $S_0$. Let
		\[
		f\colon X_0 \rightarrow C_0
		\]
		be a morphism of $S_0$-schemes. We consider a constructible complex $K_0 \in \text{D}^b_c(X_0, R)$. There exists a constructible dense open subscheme $U_0 \subseteq S_0$ and a complex $$\Psi^{\ell}_\pi(K) \in \text{D}^b_c(X_0, R)$$ with an action by $I_\ell$ such that such that we have a natural $I_\ell$-equivariant isomorphism
		\[
		\Psi^{\ell}_{\pi}(K)_x = \Psi^{\ell}_{\pi_x}(K_{x})
		\]
		for all $x \in U_0$. 
	\end{theorem}
	\begin{proof} By Remark \ref{genbcqcqs_REM}, we may assume $R$ is a Noetherian ring complete with respect to an ideal $I \subset R$. We denote by $U_0 \subseteq C_0$ the complement of $Y_0$. We define the scheme
		\[
		\Tilde{C}_0 := C_0[\pi^{1/\ell^n}~|~ n \geq 1].
		\]
	Note that $\Tilde{C}_0$ admits an action by $I_\ell$ over $C_0$. Let $\tilde{i}\colon Y_0 \rightarrow \Tilde{C}_0$ be the closed inclusion of the reduced preimage of $Y_0$ in $\Tilde{C}_0$ and $\tilde{j}\colon \Tilde{U}_0 \rightarrow \Tilde{C}_0$ the inclusion of the complement. We put
	\[
	\Psi_\pi^\ell(K) := \Tilde{i}^*\Tilde{j}_*\Tilde{j}^*(K_{\Tilde{C}_0}).
	\]
	This object naturally comes with an action by $I_\ell$ because $\Tilde{C}_0$ admits an action by $I_\ell$ over $C_0$. Our goal is to show that this complex is constructible and that the formation of the above functor commutes with base-change generically. Once this is achieved, the definition of the pro-$\ell$ nearby cycles implies that this functor is naturally isomorphic to the pro-$\ell$ nearby cycles formed over each fiber. 
	
	To prove those assertions, we may assume $S_0$ contains all $\ell$'th roots of unity by adjoining all $\ell$'th roots of unity to $S_0$ and passing to an irreducible component. The resulting integral morphism to $S_0$ is weakly etale because $\ell$ is invertible on $S_0$. By weakly etale base change, it is sufficient if we prove the claim assuming $S_0$ admits all $\ell$'th roots of unity.
	
	This allows us to choose a trivialization $I_\ell \cong \BZ_\ell$. Then $\Tilde{C}_0$ is a $\BZ_\ell$-torsor over $U_0$. Let us denote by $q\colon \Tilde{C}_0 \rightarrow C_0$ the natural projection. Then we have
	\[
	q^{-1}(Y_0)^{\text{red}} = Y_0.
	\]
	We see that $$\SL := q_*(R)$$ is a constructible $R[[\BZ_\ell]]$-module over the Iwasawa algebra $R[[\BZ_\ell]]$.
	
	We identify the Iwasawa algebra $R[[\BZ_\ell]] \cong R[[t]]$ with a power series ring over $R$. The ring $R[[t]]$ is a Noetherian ring, complete with respect to the ideal $(I, t) \subset R[[t]]$. Therefore, Theorem \ref{genbc_THM} and the projection formula for integral morphisms implies we may shrink $S_0$ and assume the formation of $\Psi_\pi^\ell(K)$ commutes with base change. Loc. cit. also permits us to assume $\Psi_\pi^\ell(K)$ is constructible over $R[[t]]$. Since we know the formation of $\Psi^\ell_\pi(K)$ commutes with base-change, we have
	\[
	\Psi^\ell_\pi(K\otimes_R R/I) = \Psi^\ell_\pi(K)\otimes_R R/I
	\]
	by the projection formula for pro-$\ell$ nearby cycles over a field and constructibility of the pro-$\ell$ nearby cycles over the Iwasawa algebra. This permits us to reduce to the discrete Noetherian coefficient ring $\Lambda := R/I$. 
	
	There exists a finite constructible partition of $X_0$ such that the cohomology sheaves of $\Psi_\pi^\ell(K)\otimes_{\Lambda[[t]]} \Lambda[[t]]/t$ are locally constant with perfect stalks over each stratum. 
	 We want to prove there is $n \geq 1$ such that the natural projection
	\[
	\Psi_\pi^\ell(K)\otimes_{\Lambda[[t]]} \Lambda[[t]]/t^{n + k}\rightarrow \Psi_\pi^\ell(K)\otimes_{\Lambda[[t]]} \Lambda[[t]]/t^n
	\]
	is an isomorphism for all $k \geq 0$. To prove this isomorphism, it is sufficient if we prove it for one point on each stratum. Since the formation of the above complexes commutes with base-change, we may prove this isomorphism assuming $S_0$ is a field. In this case, it follows from the constructibility of the pro-$\ell$ nearby cycles. 
	
	Since $\Psi_\pi^\ell(K)$ is $t$-derived complete, we obtain an isomorphism
	\[
	\Psi_\pi^\ell(K) = \Psi^\ell_\pi(K)\otimes_{\Lambda[[t]]} \Lambda[[t]]/t^n.
	\]
	The complex on the right is a constructible complex with coefficients in $\Lambda[[t]]/t^n$. Since this ring is a finite free module over $\Lambda$, we see that the complex $\Psi^\ell_f(K)$ has locally constant cohomology sheaves with perfect stalks over $\Lambda$ over the same partition as above. Note that we may find a finite filtration of $\Psi_\pi^\ell(K)$ with subquotients isomorphic to $\Psi_\pi^\ell(K)\otimes_{\Lambda[[t]]}\Lambda[[t]]/t$. This complex is classical, so the complex $\Psi^\ell_f(K)$ is classical by \cite[Cor.~5.2.9]{BhattSPE} and \cite[Prop.~5.2.6]{BhattSPE}.

	\end{proof}
	\subsection{Geometrically irreducible perverse sheaves} The goal of this section is to consider the special case when the arithmetic perverse sheaf is geometrically irreducible and descends to an arithmetic field. The following definitions also appear in	 \cite[Def.~7.2,~Def~7.4]{esnault2025semistablelefschetzpencils}.
	\begin{definition}
		Let $X$ be a $k$-scheme and $M \in \perv{X}$ a perverse sheaf with a nilpotent operator $N\colon M(1) \rightarrow M$. Let $f \colon X \rightarrow Y$ be a morphism of $k$-schemes. We say \textit{$f$ satisfies the monodromy property with respect to $M$} if the filtration induced by the spectral sequence
		\[
		E_1^{pq} := \pervCoh{p + q}{f_*(\text{Gr}_{-q}^M(M))} \Rightarrow \pervCoh{p + q}{f_*(M)}
		\]
		equals the monodromy filtration on $\pervCoh{i}{f_*(M)}$ for all $i \in \BZ$. 
		
		We assume $k$ is finite. We say $M$ is \textit{monodromy pure of weight $w \in \BZ$} if the graded pieces of the monodromy filtration $\text{Gr}_i^M(M)$ are pure of weight $w + i$ for all $i \in \BZ$. 
	\end{definition}
We prove the arithmetic Kashiwara conjecture for geometrically irreducible perverse sheaves admitting a model over an arithmetic field tensored by an indecomposable unipotent representation of the tame inertia. The reason we allow for a perverse sheaf tensored by a unipotent representation of inertia is the monodromy property: without the theory of weights, we do not know how to show that the monodromy property passes from a perverse sheaf to a perverse sheaf tensored with a representation of the inertia. 
	\begin{theorem}\label{weightmon_THM}
		Let $S$ be the strict henselization of a discrete valuation ring essentially of finite type over a field with uniformizer $\pi \in S$. Let $X$ be an $S$-scheme and $M \in \Perv{X_{{\eta}}}{}$ a geometrically irreducible perverse sheaf which admits a model over an arithmetic field (see Definition \ref{model_DEF}). Let $\rho\colon I^\ell_S \rightarrow \mathbf{GL}(\rho)$ be an indecomposable unipotent representation of the pro-$\ell$ inertia of $S$. We have:
		\begin{enumerate}
			\item The action of $I_\ell$ on $\Psi^\ell_{\pi}(M\otimes \rho)$ is quasi-unipotent.
			\item The graded pieces of the monodromy filtration on $\text{Gr}^M_i(\Psi^\ell_{\pi}(M\otimes \rho))$ are geometrically semisimple for all $i \in \BZ$. 
			\item Let $f\colon X \rightarrow Y$ be a proper morphism of $S$-schemes. The monodromy spectral sequence
			\[
			E_1^{pq} := \pervCoh{p + q}{f_*(\text{Gr}_{-q}^M(\Psi^\ell_{\pi}(M\otimes\rho)))} \Rightarrow \pervCoh{p + q}{f_*(\Psi^\ell_{\pi}(M\otimes\rho))}
			\]
			degenerates on the second page. Moreover, the morphism $f$ satisfies the monodromy property with respect to $\Psi^\ell_{\pi}(M\otimes\rho)$. 
		\end{enumerate}
	\end{theorem}
	We recall two lemmas from the theory of weights. They will play a role in the second step of the proof, where we consider nearby cycles over a curve over a finite field. We recall Gabber's monodromy purity theorem: 
	\begin{theorem}\label{weightmonodromyfinite_THM}
		Let $k$ be a finite field. Let $C$ be a smooth curve over $k$ and $0 \in C$ a point. Let $X$ be a $C$-scheme. Let $M \in \perv{X}$ be pure of weight $w \in \BR$. The pro-$\ell$ nearby cycles $\Psi^\ell_{0}(M)$ are monodromy pure of weight $w$.
	\end{theorem} 
	\begin{proof}
		The nearby cycles $\Psi(M)$ are monodromy pure of weight $w$ by \cite[Sec.~5]{beilinsonbernsteinjantzen}. Let $I^{(\ell)}$ be the kernel of the natural projection $I \rightarrow I^\ell$ from the inertia of $S$ to the maximal pro-$\ell$ quotient of the inertia of $S$. We have an $I^\ell$-equivariant isomorphism
		\[
		\Psi^\ell_0(M) = \Psi(M)^{I^{(\ell)}}.
		\]
		by Remark \ref{prolnc_REM}. All finite quotients of $I^{(\ell)}$  are of order co-prime to $\ell$. Thus $I^{(\ell)}$ acts on $\Psi(M)$ through a finite quotient. In particular, the pro-$\ell$ nearby cycles are a direct factor of the nearby cycles. This implies that the graded pieces of the monodromy filtration of $\Psi^{\ell}_0(M)$ are a direct factor of the graded pieces of the monodromy filtration of $\Psi(M)$. Thus the monodromy filtration on $\Psi^{\ell}_0(M)$ is monodromy pure of weight $w$. 
	\end{proof}
	The following lemma is a consequence of a formula for the monodromy filtration of a tensor product. We prove this formula in the appendix. The proof is based on constructing a "fiber functor'' from the category of perverse sheaves to the category of finite dimensional $\Qbarl$ vector spaces. 
	\begin{lemma}\label{puretensor_LEM}
		Let $k$ be a finite field and $X$ a $k$-scheme. Let $M \in \perv{X}$ and $N\colon M(1) \rightarrow M$ a nilpotent operator. Let $\rho$ be a representation of the Galois group of $k$, equipped with a nilpotent operator $N\colon \rho(1) \rightarrow \rho$. We assume the monodromy filtrations on $M$ and $\rho$ are monodromy pure of weight $w, w' \in \BZ$ respectively. The tensor product $M\otimes \rho$ is monodromy pure of weight $w + w'$. 
	\end{lemma}
	\begin{proof}
		Let $k \in \BZ$. Proposition \ref{tensorformula_PROP} shows
		\[
		\text{Gr}^M_k(M\otimes V) = \bigoplus_{i + j = k} \text{Gr}^M_i(M)\otimes\text{Gr}^M_j(V)
		\]
		Tensoring a pure sheaf of weight $w + i$ with a pure Galois representation of weight $w' + j$ gives a pure perverse sheaf of weight $$w + w' + i + j = w + w' + k,$$ hence the claim.
	\end{proof}

	\begin{proof}[Proof~(Thm.~\ref{weightmon_THM}).]
		
	We begin by constructing an auxiliary sheaf on a geometrically connected regular curve $C$ over a field $k$ with a morphism $S \rightarrow C$ such that the closed point maps to a rational point $0 \in C(k)$ and the induced map on the pro-$\ell$ inertia
	\[
	I^t_S \rightarrow I_0^t
	\] 
	is an isomorphism. We choose a morphism
	\[
	C \rightarrow \BA^1_{k}
	\]
	which is unramified at $0$. Let $k_0 \subset k$ be the prime field, then we compose this morphism with the natural projection
	\[
	\BA^1_k \rightarrow \BA^1_{k_0}.
	\]
	Let $\eta$ be the generic point of $S$, then we have the composition
	\[
	\eta \rightarrow \BG_{m, \overline{k}_0}.
	\]
	This map induces an isomorphism
	\[
	I^\ell_{S} \xrightarrow{\sim} \pi_1^{\ell}(\BG_{m, \overline{k}_0}).
	\]
	By composing the inverse of this isomorphism with $\rho$, we obtain a lisse $\Qbarl$-sheaf $\Tilde{\SL}_0$ on $\BG_{m, \overline{k}_0}$. Since the representation $\rho$ is unipotent and indecomposed, it is isomorphic to a Jordan block of some dimension. One can check from this that the conjugation action by the Galois group of the base field $k_0$ does not change the isomorphism type of $\Tilde{\SL}_0$. Thus we find that the sheaf $\Tilde{\SL}_0$ descends to a sheaf over a finite extension of the prime field $k_0$. We replace $k_0$ by a finite extension so that $\Tilde{\SL}_0$ descends to the new base field. Let $j\colon \BG_{m, k_0} \rightarrow \BA^1_{k_0}$ be the inclusion, then we denote by $\SL_0$ the pullback of $j_!(\Tilde{\SL}_{0})$ to $C$. This gives a constructible etale sheaf on $C$, lisse outside the fiber of $C \rightarrow \BA^1_{k}$ over 0, such that the inertia representation at $0$ is isomorphic to $\rho$. Note that the sheaf $\SL_0$ admits a model over an arithmetic field. Moreover, when the base field $k$ is finite, we see that the representation of the decomposition group at 0 is monodromy pure, at least after a twist, because it only has only one one-dimensional primitive part.
	
%
%
	We can write
	\[
	S = \varprojlim C_i
	\]
	as a cofiltered limit over regular geometrically connected curves over arithmetic fields (which increase with $i$). Since $X$ is finitely presented over $S$, we may write
	\[
	X = \varprojlim X_i
	\]
	as a cofiltered limit of $C_i$-schemes. The uniqueness of models noted in Remark \ref{uniquenessmodels_REM} allows us to assume $M$ descends to the generic fiber of $X_i$. The morphisms
	\[
	S \rightarrow \text{Spec}(\CO_{C_i, 0})
	\]
	induce an isomorphism on pro-$\ell$ inertia. Since the formation of pro-$\ell$ nearby cycles commutes with base-change along such morphisms by Lemma \ref{prolncbc_LEM}, we may assume $S$ is the strict henselization of a local ring of a regular geometrically connected curve $C$ over an arithmetic field $k_0$. Moreover, we may assume the perverse sheaf $M$ exists on the generic fiber of the morphism $X \rightarrow C$. After a finite extension of $k_0$, we may assume the curve $C$ is smooth and the point $0 \in C$ is rational over $k_0$. Since $k_0$ is arithmetic, we see that the pro-$\ell$ nearby cycles $\Psi^\ell_0(M)$ satisfy the quasi-unipotent monodromy theorem by Proposition \ref{qunimon_PROP}. Thus we obtain (1).
	
	By Lemma \ref{modelfields_LEM}, there exists an arithmetic regular connected base-scheme $S_0$ with function field $k_0$. After shrinking $S_0$, we may assume there exists a smooth $S_0$-scheme $C_0$ with geometrically connected fibers of dimension 1, a section $S_0 \rightarrow C_0$, $S_0$-schemes $X_0$ and $Y_0$, and a proper morphism $X_0 \rightarrow Y_0$, such that these objects reproduce the corresponding previous objects over the generic fiber of $S_0$. By Theorem \ref{extensionari_THM}, we may assume the perverse sheaf $M$ extends to a relatively perverse universally locally acycic sheaf $M_0 \in \text{Perv}^{\text{ULA}}(X_0/S_0)$, after shrinking $S_0$. By shrinking $S_0$ further, Proposition \ref{critsemisimple_PROP} permits us to assume the fiber of $M_0$ over each point in $S_0$ is geometrically irreducible.
	
	By the generic constructibility theorem proven above, we may shrink $S_0$ and assume there exists a constructible complex
	\[
	\Psi^{\ell}_{\pi}(M_0\otimes \SL_0) \in \consder{X_{S_0}}
	\] 
	with an action by $I_\ell$ such that $\Psi^{\ell}_{\pi}(M_0)_{k_0}$ is $I_\ell$-equivariantly isomorphic to $\Psi^{\ell}_{\pi}(M)$. By generic universal local acyclicity, we may shrink $S_0$ and assume $\Psi^{\ell}_{\pi}(M_0)$ is ULA over $S_0$.
	
	Since the scheme $S_0$ is arithmetic, the quasi-unipotent monodromy theorem from Proposition \ref{qunimon_PROP} implies that there is a nilpotent monodromy operator
	\[
	N\colon \Psi^{\ell}_{\pi}(M_0\otimes\SL_0)(1) \rightarrow \Psi^{\ell}_{\pi}(M_0\otimes\SL_0).
	\]
	Since this is a morphism in the abelian category of relatively perverse universally locally acyclic sheaves over $S_0$, we can form the monodromy filtration over $S_0$. Note that this filtration reproduces the monodromy filtration over each point in $S_0$ because the passage to a fiber is $t$-exact by definition of the relatively perverse $t$-structure.
	
	Proposition \ref{critsemisimple_PROP} shows that the monodromy filtration has geometrically semisimple graded pieces over the generic fiber if and only if the monodromy filtration has geometrically semisimple graded pieces over a dense subset of closed points of $S_0$. Since the fiber of $M_0$ over each closed point of $S_0$ is geometrically irreducible, we may assume $S_0$ is a finite field to prove point (2).
	
	We consider the spectral sequence
	\[
	E_1^{pq} := {}^{\idp/S_0}\mathscr{H}^{p + q}(f_*(\text{Gr}_{-q}^M(\Psi^\ell_\pi(M_0)))) \Rightarrow {}^{\idp/S_0}\mathscr{H}^{p + q}(f_*(\Psi^\ell_\pi(M_0)))
	\]
	The monodromy filtration and the pushforward commute with passage to a fiber. Therefore, the formation of this spectral sequence commutes with base-change.
	
	In particular, Lemma \ref{morphismzero_LEM} shows that the spectral sequence over the generic fiber degenerates on the second page if and only if the spectral sequence degenerates on the second page over a dense subset of closed points. To prove point (3), we may therefore assume $S_0$ is a finite field.
	
	Moreover, we can consider the filtration induced by the above spectral sequence on the relative perverse cohomology sheaves of $f_*(\Psi^\ell_{\pi}(M_0))$. The formation of these images commutes with passage to fibers by $t$-exactness of passage to fibers. Moreover, we may utilize the existence of the nilpotent operator on $\Psi^\ell_{\pi}(M_0)$ to construct a monodromy operator on the relative perverse cohomology sheaves of the pushforward. We can then form the monodromy filtration in the relative perverse category over $S_0$. By $t$-exactness of passage to fibers, the formation of the relative monodromy filtration commutes with passage to fibers. 
	
	Thus we have two filtrations on each relatively perverse cohomology sheaf of the pushforward of the pro-$\ell$ nearby cycles. By taking intersections, we can appeal to Corollary \ref{isocrit_COR} to show that the two filtrations agree on the generic fiber if and only if they agree at a dense subset of closed points. To prove point (4), we may therefore assume $S_0$ is a finite field. Thus we have reduced the proof of points (2), (3), and (4) to the case when $S_0$ is a finite field. 
	
	In fact, the reduction shows that we are in the following situation: we have a finite field $k_0$, two $k_0$-schemes $X_0$ and $Y_0$, a smooth geometrically connected curve $C_0$, a point $0 \in C_0$, and morphisms
	\[
	X_0 \xrightarrow{f} Y_0 \rightarrow C_0.
	\]
	Furthermore, we have a geometrically irreducible perverse sheaf $M_0 \in \perv{X_0}$. After twisting $M_0$, we may assume $M_0$ is pure of weight zero by \cite[Cor.~VII.8]{LafforguePurity}. Morever, we also have the sheaf $\SL_0$ on $C_0$. Since the monodromy filtration of $\SL_0$ at $0$ only has one primitive part, we may twist $\SL_0$ and assume it is monodromy pure of weight zero.
	
	Then, the projection formula shows that we have an inertia equivariant isomorphism
	\[
	\Psi^{\ell}_0(M_0\otimes\SL_0) \cong \Psi^{\ell}_0(M_0)\otimes \Psi^{\ell}_0(\SL_0)
	\] 
	Note that $\Psi^{\ell}_0(\SL_0)$ is monodromy pure of weight zero, because the representation of the decomposition group of $C_0$ at $0$ induced by $\SL_0$ is monodromy pure of weight zero. Therefore, Theorem \ref{weightmonodromyfinite_THM} and Lemma \ref{puretensor_LEM} show that $$\text{Gr}^M_n(\Psi^{\ell}_0(M_0\otimes \SL_0))$$ is pure of weight $n \in \BZ$ for all $n \in \BZ$. Thus we see that the graded pieces are geometrically semisimple, i.e. we obtain (2).
	
	We consider the monodromy spectral sequence
	\[
	E_1^{pq} := \pervCoh{p + q}{f_*(\text{Gr}_{-q}^M(\Psi^\ell_0(M\otimes\rho)))} \Rightarrow \pervCoh{p + q}{f_*(\Psi^\ell_0(M\otimes\rho))}.
	\]
	Note that the monodromy filtration on the pushforward is monodromy pure of weight zero by \cite[Stab.~5.1.14]{BBD}, Lemma \ref{puretensor_LEM}, and the projection formula. Therefore, the maps of this spectral sequence must vanish on the second page, i.e. the spectral sequence degenerates on the second page. To prove the monodromy property, we can now apply \cite[Prop.~7.7]{esnault2025semistablelefschetzpencils}.

	\end{proof}
\subsection{Arithmetic Kashiwara conjecture} We can now give the main application of this paper: 
	\begin{theorem}\label{kashiwara_THM}
		Let $S$ be the strict henselization of a discrete valuation ring essentially of finite type over a field. Let $X$ be an $S$-scheme and $M \in \perv{X_{{\eta}}}$ a geometrically semisimple arithmetic perverse sheaf on the generic fiber.
		\begin{enumerate}
			\item The graded piece of the monodromy filtration $\text{Gr}^M_n(\Psi^{\text{uni}}(M))$ is geometrically semisimple for all $n \in \BZ$. 
			\item Let $f\colon X \rightarrow Y$ be a proper morphism of $S$-schemes. The monodromy spectral sequence
			\[
			E_1^{pq} := \pervCoh{p + q}{f_*(\text{Gr}_{-q}^M(\Psi^{\text{uni}}(M)))} \Rightarrow \pervCoh{p + q}{f_*(\Psi^{\text{uni}}(M))}
			\]
			degenerates on the second page. Moreover, the morphism $f$ satisfies the monodromy property with respect to $\Psi^{\text{uni}}(M)$. 
		\end{enumerate} 
	\end{theorem}
	\begin{proof} We first note that it is sufficient to prove the statements for the quasi-unipotent nearby cycles $\Psi^{\text{quni}}(M)$, by which we mean the direct factor of $\Psi^t(M)$ on which inertia acts quasi-unipotently. Indeed, the perverse sheaf $\Psi^{\text{uni}}(M)$ is a direct factor of $\Psi^{\text{quni}}(M)$.
		
	We begin by proving that it is sufficient to provide the statement after passing to the normalization in a separable Galois extension $S' \rightarrow S$. Let $G$ be the Galois group of the extension. Let $P$ be the wild inertia subgroup of $G$. Note that $M$ pulled back to $S'$ has an action by $G$. This induces a (strict) $P$-action on the tame nearby cycles $\Psi^t_{S'}(M)$. We have
	\[
	\Psi^t_{S}(M) = \Psi^t_{S'}(M)^P.
	\]
	The action of $P$ commutes with the action of inertia on $\Psi^t_{S'}(M)$ by functoriality of the nearby cycles. Thus we obtain
	\[
	\Psi^{\text{quni}}_S(M) =  \Psi^{\text{quni}}_{S'}(M)^P.
	\]
	Thus it is sufficient to prove the statement after restricting along $S' \rightarrow S$ because $\Psi^{\text{quni}}_{S'}(M)^P$ is a direct factor of $\Psi^{\text{quni}}_{S'}(M)$.
			
	After passing to a finite separable Galois extension of the function field of $S$, Theorem \ref{geomsimpleperverse_THM} and Theorem \ref{modelsIrr_THM} show that we may assume the perverse sheaf $M$ to be of the form $$M = M'\otimes\rho\otimes \chi$$ for a geometrically irreducible arithmetic perverse sheaf $M'$ admitting a model over an arithmetic field, an indecomposable unipotent representation $\rho\colon I^\ell \rightarrow \mathbf{GL}(\rho)$ of the pro-$\ell$ inertia of $S$, and a continuous character $\chi\colon I^\ell \rightarrow \Qbarl^*$. Moreover, we can assume the action of inertia on $\Psi(M)$ is pro-$\ell$ because the sheaves have coefficients in $\Qbarl$. In particular, it is sufficient if we prove the statement for the pro-$\ell$ nearby cycles for some uniformizer $\pi \in \Gamma(S, \CO_S)$. 
	
	We remark that action of inertia on the nearby cycles $\Psi^{\ell}_{\pi}(M')$ is quasi-unipotent by Theorem \ref{weightmon_THM}. Since we have the isomorphism
	\[
	\Psi^{\ell}_{\pi}(M) = \Psi^{\ell}_{\pi}(M')\otimes \rho \otimes\chi
	\]
	by the projection formula, we see that $\Psi^{\ell}_{\pi}(M)$ has quasi-unipotent inertia if and only if $\chi$ has finite order. 
	
	If $\chi$ does not have finite order, we obtain
	\[
	\Psi^{\text{quni}}(M) = 0.
	\]
	Thus we may assume $\chi$ has finite order. In this case, we obtain
	\[
	\Psi^{\text{quni}}(M) = \Psi^{\ell}_{\pi}(M).
	\]
	Since $\chi$ has finite order, $\chi$ admits a model over an arithmetic field. Thus the perverse sheaf $M'\otimes\chi$ admits a model ver an arithmetic field. Note that we can write
	\[
	\Psi^{\ell}_{\pi}(M) = \Psi^{\ell}_{\pi}((M'\otimes\chi)\otimes\rho). 
	\]
	Thus the statement follows from Theorem \ref{weightmon_THM}.
	\end{proof}
	\newpage
	\section{Appendix} The goal of the appendix is to prove theorems which otherwise do not fit the text. 
	\subsection{Specialization of monodromy}	To prove the specialization theorem on monodromy groups, we require a certain morphism on fundamental groups to be surjective. We prepare the comparison of the proetale fundamental group with the etale fundamental group. We have the following variant of the Hilbert basis theorem for topologically Noetherian schemes:
\begin{proposition}\label{topnot_PROP}
	Let $S_0$ be a topologically Noetherian scheme. Let $X_0 \rightarrow S_0$ be a morphism of finite type. Then $X_0$ is topologically Noetherian. 
\end{proposition}
\begin{proof}
	
	We may assume $X_0$ and $S_0$ are affine, because the problem is Zariski local on $X_0$ (since $X_0$ is quasi-compact). By embedding $X_0$ in affine space, we may assume $X_0 = \BA^n_{S_0}$. By induction on $n$, we may assume $X_0 = \BA^1_{S_0}$. By Noetherian induction on $S_0$, we may assume $\BA^{1}_{Z_0}$ is topologically Noetherian for any closed subscheme supported on a proper closed subset $Z_0 \subset S_0$. Thus we may assume $S_0$ irreducible. By passing to a dense open subscheme of $S_0$, we may assume $S_0$ is affine. By passing to the reduction of $S_0$, we may assume $S_0 = \text{Spec}(R)$ is the spectrum of a topologically Noetherian domain $R$.
	
	Let 
	\[
	Z_0 \supseteq \ldots \supseteq Z_n \supseteq \ldots
	\]
	be a descending chain of closed subsets $Z_n \subseteq \BA^1_{S_0}$ indexed by $n \in \BN$. We want to show the chain stabilizes. Thus we may assume $Z_0 \neq \BA^1_{S_0}$. Then there is a non-zero polynomial $f \in R[x]$ with $Z_0 \subseteq V(f)$. We write the polynomial as
	\[
	f  = a_nx^n + \ldots + a_0
	\]
	with coefficients $a_i \in R$ and $a_n \neq 0$ for $n \geq 0$. Noetherian induction implies $\BA^{1}_{V(a_n)}$ is topologically Noetherian, thus the above chain stabilizes when intersected with $\BA^1_{V(a_n)}$. Thus we may pass to the complement $\BA^1_{D(a_n)}$ and assume $f$ is monic. 
	
	Suppose $n = 0$. In this case, we have $f = 1$ and hence $V(f) = \emptyset$. Thus $Z_n = \emptyset$ for all $n \geq 0$ which shows that the chain has stabilized. 
	
	Suppose $n > 0$. We consider the finite finitely presented flat morphism
	\[
	V(f) \rightarrow S_0.
	\]
	Either the generic point of $S_0$ has characteristic zero (and we set $F := \text{id}_{S_0}$) or we can find an iterate of Frobenius $F\colon S_0 \rightarrow S_0$ such that
	\[
	(V(f)\times_{S_0, F} S_0)^{\text{red}} \rightarrow S_0
	\]
	is generically etale. The natural projection
	\[
	(V(f)\times_{S_0, F} S_0)^{\text{red}} \rightarrow V(f)
	\]
	is a universal homeomorphism. Thus it is sufficient if we show that the preimages of the closed subsets $Z_i$ in $(V(f)\times_{S_0, F} S_0)^{\text{red}}$ stabilize. By the Noetherian induction hypothesis, we may restrict $S_0$ to a dense open subscheme. Thus we may assume the above morphism is finite etale. The reference \cite[Lem.~6.6.10~(3)]{BhattSPE} shows that $(V(f)\times_{S_0, F} S_0)^{\text{red}}$ is topologically Noetherian. This implies that the chain of closed subsets $Z_n$ stabilizes, thus $\BA^{1}_{S_0}$ is topologically Noetherian.
\end{proof}
\begin{corollary}\label{loctopnot_COR}
	Let $S_0$ be a locally topologically Noetherian scheme and $X_0 \rightarrow S_0$ a morphism locally of finite type. Then $X_0$ is locally topologically Noetherian.
\end{corollary}
\begin{proof}
	This follows immediately from Proposition \ref{topnot_PROP}.
\end{proof}
\begin{remark}\label{etalegroupclassifies_REM}
	Let $X_0 \rightarrow S_0$ be a smooth morphism with connected fibers over a connected geometrically unibranch locally topologically Noetherian scheme $S_0$. Then $X_0$ is geometrically unibranch and, hence, it is connected. In particular, Corollary \ref{loctopnot_COR} and \cite[Lem.~7.4.10]{BhattSPE} show that lisse proetale sheaves on $X_0$ are classified by the etale fundamental group.
\end{remark}
	The following property is crucial for the specialization theory.
	\begin{lemma}\label{fundgroupsurjects_LEM}
	Let $S_0$ be an irreducible geometrically unibranch scheme with function field $k$ and $X_0$ a smooth $S_0$-scheme. If $X_0$ is connected, the generic fiber $X_k$ is connected. In particular, if $X_0$ is connected, then the morphism
		\[
		\pi_1(X_k) \rightarrow \pi_1(X_0)
		\]
	on etale fundamental groups is surjective.
	\end{lemma}
	\begin{proof} 
		We may assume $S_0$ is normal by passing to the normalization (because the normalization of $S_0$ induces a universal homoeomorphism). Then the scheme $X_0$ is normal, hence it is irreducible. Thus the generic fiber $X_k$ is irreducible. In particular, the generic fiber is connected. This applies to all connected finite etale neighborhoods of $X_0$, thus the morphism on etale fundamental groups is surjective.
	\end{proof}

	To state the following theorem, it is convenient to make the following definition.
	\begin{definition}\label{conjugate_DEF}
		Let $M_1$ and $M_2$ be modules over a commutative ring $R$. Let $H_1 \subseteq \text{GL}(M_1)$ and $H_2 \subseteq \text{GL}(M_2)$ be subgroups. We say $H_2$ is \textit{conjugate to a subgroup of $H_1$} if there is an isomorphism $M_1 \xrightarrow{\sim} M_2$ of $R$-modules such that the image of $H_2$ along the induced map $\text{GL}(M_2) \rightarrow\text{GL}(M_1)$ is contained in $H_1$. We say $H_1$ is \textit{conjugate to $H_2$} if there exists an isomorphism as above such that the image of $H_2$ along the induced isomorphism is equal to $H_1$. 
	\end{definition}
	The following theorem generalizes \cite[Thm.~8.18.2]{KatzESDE} to more general base schemes.
	\begin{theorem}\label{specializationmonodromy_THM}
		Let $S_0$ be a topologically Noetherian scheme. Let $X_0 \rightarrow S_0$ be a smooth morphism with geometrically connected fibers. Let $F$ be a lisse $\CO$-sheaf on $X_0$. Let $x \in S_0$ be a point. We define the group (well-defined up to conjugacy)
		\[
		\Gamma(x) := \text{im}(\pi_1(X_{\overline{x}}) \rightarrow \text{GL}(F)),
		\]	 
		where $\overline{x} \rightarrow S_0$ is any geometric point supported at $x$.
		\begin{enumerate}[]\item Let $x_0, x_1 \in S_0$ be points such that $x_0 \rightsquigarrow x_1$. Then $\Gamma(x_1)$ is conjugate to a subgroup of $\Gamma(x_0)$.
		\item There exists a finite constructible partition
		\[
		S_0 = \bigsqcup_{i = 1}^n S_i
		\]
		such that for all $1 \leq i \leq n$ and for all $x, y \in S_i$, the group $\Gamma(x)$ is conjugate to $\Gamma(y)$.
		\end{enumerate}
	\end{theorem}
	\begin{proof}
		Because $S_0$ is topologically Noetherian, we can find an aic valuation ring of finite rank $V_0$ and a morphism $V_0 \rightarrow S_0$ mapping the generic point of $V_0$ to $x_0$ and the special point to $x_1$. The etale fundamental group of the scheme $X_{V_0}$ classifies lisse etale sheaves by Remark \ref{etalegroupclassifies_REM}. Thus we obtain the diagram
		\[
		\begin{tikzcd}
			& \pi_1^{\text{}}(X_{\overline{x}_1}) \arrow[d] &              \\
			\pi_1^{\text{}}(X_{\overline{x}_0}) \arrow[r, two heads] & \pi_1^{\text{}}(X_{V_0}) \arrow[r]             & \text{GL}(F)
		\end{tikzcd}.
		\] 
		The arrow $\pi_1(X_{\eta}) \rightarrow \pi_1^{}(X_{V_0})$ is surjective by Lemma \ref{fundgroupsurjects_LEM}. Thus $\Gamma(x_0)$ is conjugate to $\Gamma(x_1)$. To prove (2), we can argue as in \cite[Thm.~8.18.2]{KatzESDE}.
	\end{proof}
	\subsection{Descent lemma} The goal of this section is to prove the descent lemma. The argument is based on the following lifting lemma. We first make a general remark:
	\begin{remark}\label{continuousgroupmorphism_REM}
		The ring $\Qbarl$ admits a unique valuation extending the $\ell$-adic valuation of $\BQ_\ell$. We equip $\Qbarl$ with the valuation topology. For any scheme $X$ locally of finite type over $\Qbarl$, this allows us to functorially topologize $X(\Qbarl)$ by locally embedding it into affine space. 
		
		Let $G$ be a profinite group and $\mathbf{G}$ a locally algebraic group over $\Qbarl$. A \textit{continuous group homomorphism $G \rightarrow \mathbf{G}$} means a continuous group homomorphism
		\[
		G \rightarrow \mathbf{G}(\Qbarl),
		\]
		where $\mathbf{G}(\Qbarl)$ is equipped with the topology described above.
		
	\end{remark}
	\begin{lemma}\label{liftinglemma_LEM}
		 Let $G$ be a profinite group and $\mathbf{G}$ an algebraic group over $\Qbarl$. Suppose $\mathbf{H} \subseteq \mathbf{G}$ is a normal algebraic subgroup of $\mathbf{G}$ and we are given a continuous map $G \rightarrow \mathbf{G}/\mathbf{H}$. We assume $\mathbf{H}^0$ is a torus. There exists an open subgroup $U \subseteq G$ and a continuous  map $U \rightarrow \mathbf{G}$ fitting into the diagram 
				\[
				\begin{tikzcd}
						U \arrow[r] \arrow[d, dashed] & G \arrow[d] \\
						\mathbf{G} \arrow[r]                   & \mathbf{G}/\mathbf{H}.         
					\end{tikzcd}
				\]
		If $G \rightarrow \mathbf{G}/\mathbf{H}$ is a group morphism, then we can choose $U \rightarrow \mathbf{G}$ to be a group morphism. 
	\end{lemma}  
	\begin{proof}
		We first treat the case where $\mathbf{H}$ is finite. There exists an open subgroup $U \subseteq \mathbf{G}$ such that $U\cap\mathbf{H} = \{1\}$. If we restrict the morphism $G \rightarrow \mathbf{G}/\mathbf{H}$ to the open subgroup $U$, we obtain the desired lifting because the morphism $U \rightarrow \mathbf{G}/\mathbf{H}$ is an open injection. This lifting is automatically a group morphism if $G \rightarrow \mathbf{G}/\mathbf{H}$ is a group morphism 
		
		We consider the general case. By the first step, we can lift the morphism to $\mathbf{G}/\mathbf{H}^0$. To obtain the lifting to $\mathbf{G}$, note that \cite[Prop.~5.3,~Step~1]{conradgrpcohomology} says that there is an algebraic subgroup $\mathbf{G'}  \subseteq\mathbf{G}$ such that $\mathbf{G'}\rightarrow \mathbf{G}/\mathbf{H}^0$ is an isogeny. We can lift the morphism to $\mathbf{G'}$ by the first step. By composition, we obtain a lifting of the morphism to $\mathbf{G}$. 
	\end{proof}
	We can now deduce a preliminary descent lemma.
	\begin{lemma}\label{prellift_LEM}
		Let $\rho\colon N \rightarrow \mathbf{GL}_n$ be a continuous irreducible $\Qbarl$-representation of a profinite group $N$. Let $G = N\ltimes H$ be the continuous product with a profinite group $H$. Suppose we have $\rho^g\cong \rho$ for all $g \in N$. There exists an open subgroup $U \subseteq H$ and a continuous lifting $\Tilde{\rho} \colon N\ltimes U \rightarrow \mathbf{GL}_n$ of $\rho$. 
	\end{lemma}
	\begin{proof}
		The isomorphisms 
		\[
		\rho^g \cong \rho
		\]
	are unique up to a scalar by Schur's Lemma. Thus we obtain a continuous morphism (see \cite[Prop.~B.2]{esnault2026nonabelianversiondelignesfixed})
		\[
		H \rightarrow \mathbf{PGL}_n.
		\]
		By Lemma \ref{liftinglemma_LEM}, we can find an open subgroup $U \subseteq H$ such that there is a lifting
		\[
		\phi\colon U \rightarrow\mathbf{SL}_n.
		\]
		We can take $\Tilde{\rho}(g, x) := \rho(g)\sigma(x)$ to obtain the lifting. 
	\end{proof}
	\begin{lemma}\label{conjactioninner_LEM}
		Let $\mathbf{G}$ be an algebraic group over $\Qbarl$. We assume there is a normal algebraic subgroup $\mathbf{N} \subseteq \mathbf{G}$ such that $\mathbf{N}^0$ is reductive. We consider a profinite group $G$ with a continuous morphism $\varphi\colon G \rightarrow \mathbf{G}$. There exists an open subgroup $U \subseteq G$ such that the conjugation action of $U$ on $\mathbf{N}$ is inner and induced by a continuous group homomorphism $U \rightarrow \mathbf{N}$.
	\end{lemma}
	\begin{proof}
%
		
		The automorphism group of $\mathbf{N}^0$ is representable by a locally algebraic group $\underline{\text{Aut}}(\mathbf{N}^0)$ over $\Qbarl$ (see SGA) because $\mathbf{N}^0$ is reductive. The conjugation action of $\mathbf{G}$ on $\mathbf{N}^0$ induces a morphism of group schemes
		\[
		\mathbf{G}\rightarrow \underline{\text{Aut}}(\mathbf{N}^0).
		\]
		By passing to an open subgroup of $G$, we may assume the composite from $G$ to the automorphism group $\underline{\text{Aut}}(\mathbf{N}^0)$ factors through the connected component of the automorphism group. The connected component is the image of the morphism induced by the conjugation action (SGA)
		\[
		\mathbf{N}^0 \rightarrow \underline{\text{Aut}}(\mathbf{N}^0).
		\]
		The connected component of the kernel of this morphism is a torus, because $\mathbf{N}^0$ is reductive by assumption. By Lemma \ref{liftinglemma_LEM}, we may pass to an open subgroup of $G$ and assume the conjugation action of $G$ on $\mathbf{N}^0$ is induced by the conjugation action along a continuous group homomorphism
		\[
		\sigma\colon G \rightarrow \mathbf{N}^0.
		\]
		
		The conjugation action of $G$ on $\mathbf{N}$ factors to give an action on the finite quotient $\mathbf{N}/\mathbf{N}^0$. Since this quotient is finite, we may pass to an open subgroup of $G$ and assume the conjugation action of $G$ on $\mathbf{N}/\mathbf{N}^0$ is trivial. 
%


		We compare the conjugation action induced by $\varphi$ and $\sigma$ on $\mathbf{N}$. For each $g \in G$ and $x \in \mathbf{N}$, we define
		\[
		\beta_g(x) := x\sigma(g^{-1})\varphi(g)x^{-1}\varphi(g^{-1})\sigma(g).
		\]
		Let $\mathbf{Z}$ be the center of $\mathbf{N}^0$. One can verify that
		\[
		\beta_g(x) \in \mathbf{Z}
		\]
		for all $g \in G$ and $x \in \mathbf{N}$. Moreover, we have
		\[
		\beta_g(xy) = \beta_g(x)
		\] 
		for all $y \in \mathbf{N}^0$. The group $\mathbf{N}/\mathbf{N}^0$ acts on $\mathbf{Z}$ by conjugation. With this action, the above morphism defines a crossed homomorphism. Thus we obtain a class in group cohomology 
		\[
		[\beta_g] \in H^1(\mathbf{N}/\mathbf{N}^0, \mathbf{Z}).
		\]
		The description of $H^1(\mathbf{N}/\mathbf{N}^0, \mathbf{Z})$ as crossed homomorphisms modulo principal crossed homomorphisms allows us to equip the group cohomology with the structure of an algebraic group. This implies that the cohomology group $H^1(\mathbf{N}/\mathbf{N}^0, \mathbf{Z})$ is finite by Shapiro's Lemma. Morever, the morphism 
		\[
		G \rightarrow H^1(\mathbf{N}/\mathbf{N}^0, \mathbf{Z})
		\]
		induced by $\beta$ is continuous. Thus we may pass to an open subgroup of $G$ and assume $\beta_g$ is a principal crossed homomorphism for all $g \in G$. After passing to an open subgroup of $G$, Lemma  \ref{liftinglemma_LEM} shows that the function $g \mapsto \beta_g$ can be lifted to a continuous function
		\[
		G \rightarrow \mathbf{Z},~ g \mapsto z_g
		\]
		such that
		\[
		\beta_g(x) = xz_gx^{-1}z_g^{-1}
		\]
		for all $x \in \mathbf{N}$ and $g \in G$. This relation may also be written as
		\[
		\varphi(g)^{-1}x\varphi(g) = \sigma(g)^{-1}z_gxz_g^{-1}\sigma(g)
		\]
		for all $x \in \mathbf{N}$ and $g \in G$. 
		
		Thus we obtain a continuous function
		\[
		\phi\colon G \rightarrow \mathbf{N},~g \mapsto z_g^{-1}\sigma(g). 
		\]
		such that the conjugation on $\mathbf{N}$ induced by $\phi(g)$ agrees with the conjugation induced by $\varphi(g)$ for all $g \in G$. Let $\mathbf{Z}' \subseteq \mathbf{N}$ be the center of $\mathbf{N}$. The composite
		\[
		G \xrightarrow{\phi} \mathbf{N} \rightarrow \mathbf{N}/\mathbf{Z}'
		\]
		is a continuous group morphism. By applying Lemma \ref{liftinglemma_LEM} again, we conclude that there exists a continuous group homomorphism $G \rightarrow \mathbf{N}$ inducing the conjugation action induced by $\varphi$.
		
	\end{proof}
	\begin{proposition}\label{descentmonodromy_PROP}
		Let $N$ be a profinite group and $\rho\colon N \rightarrow \mathbf{GL}_n$ a continuous, irreducible representation of $N$ over $\Qbarl$. We consider a continuous semidirect product $G = H\ltimes N$ with a profinite group $H$. We denote the Zariski closure of $\rho(N)$ in $\mathbf{GL}_n$ by $\mathbf{N}$. Suppose $\rho^g\cong\rho$ for all $g \in H$. There exists an open subgroup $U \subseteq H$ and a continuous lifting $\Tilde{\rho} \colon N\ltimes U \rightarrow\mathbf{N}$ of $\rho$. 
	\end{proposition}
	\begin{proof}
		By Lemma \ref{prellift_LEM}, we can assume $\rho$ lifts to a continuous group morphism $\rho'\colon N\ltimes U \rightarrow \mathbf{GL}_n$  for an open subgroup $U \subseteq H$. Let $\mathbf{G}$ be the Zariski closure of $\rho'(N\ltimes U)$ in $\mathbf{GL}_n$. By Lemma \ref{conjactioninner_LEM}, we can find a morphism $\phi\colon U \rightarrow \mathbf{N}$ such that the conjugation action of $U$ on $\mathbf{N}$ is given by $\phi$. Thus we can obtain the desired extension of the above group morphism by putting $\Tilde{\rho}(g, x) := \rho(g)\phi(x)$. 
	\end{proof}

\subsection{Perverse sheaves} We would like to prove the tensor product formula for perverse sheaves with a nilpotent operator tensored with a vector space with a nilpotent operator. This formula is well-known for finite dimensional vector spaces in characteristic zero. However, we were unable to find the formula in the exact generality required for this article.

We must therefore prove the tensor product formula. Our solution to this problem is to provide a sufficiently good functor from the category of perverse sheaves to the category of finite-dimensional vector spaces so that we may reduce the tensor product formula to the category of finite-dimensional vector spaces. The reduction is based on the following theorem:
\begin{theorem}\label{fibfuncperv_THM}
	Let $X$ be a $k$-scheme. There exists a faithful, exact, and $\Qbarl$-linear functor to the category of finite dimensional $\Qbarl$-vector spaces
	\[
	\perv{X} \rightarrow \Qbarl-\text{fdVec}
	\]
\end{theorem}
We construct an explicit functor when $X = \BA^n$ from Beilinson's basic lemma. We then reduce the general case to affine space. 
\begin{definition}
	Let $n \geq 0$ and $k$ a field. An \textit{affine hyperplane} is a closed subscheme $H \subseteq \BA^n_k$ defined by an equation of the form
	\[
	a_1x_1 + \ldots + a_nx_n = b
	\]
	for $a_i, b \in k$. For all $m \geq 0$, we define an \textit{affine $m$-hyperplane} $(H_1, \ldots, H_m)$ to be an ordered tuple of $m$ affine hyperplanes. We often identify an affine $m$-hyperplane $H$ with its union $H_1\cup \ldots\cup H_m$.
\end{definition}
\begin{remark}\label{hypplane_REM}
	Let $n \geq 0$ and $m \geq 1$. We have a parameter space of affine $m$-hyperplanes given by $X_m := \BA^{nm + m}$. Over the parameter space, we have a universal affine $m$-hyperplane $H_m\subseteq \BA^{n}\times X_m$ given by the union of the $m$ affine hyperplanes over each point in $X_m$. The scheme $H_m$ admits the irreducible decomposition
	\[
	H_m = H_{m, 1}\cup \ldots\cup H_{m, m}
	\] 
	because affine $m$-hyperplanes are ordered. 	We write 
	\begin{align*}
	H_1 &:= H_{m, m} \\H_2 &:= H_{m, 1}\cup \ldots \cup H_{m, m - 1}.
	\end{align*}
	
	We define the inclusions
	\begin{align*}
	i_{1}&\colon H_1 \rightarrow \BA^n\times X_m \\
	i_2&\colon H_2 \rightarrow \BA^n\times X_m \\
	i&\colon H_m \rightarrow \BA^n\times X_m.
	\end{align*}
	with corresponding complements $j_i\colon U_i \rightarrow \BA^n\times X_m$ and $j\colon U \rightarrow \BA^n\times X$ for $i = 1,2$. 
	
	We will also denote base changes of the above morphisms using the same letters. For example, we denote the inclusion
	\[
	j_2\colon U \rightarrow U_1
	\]
	by $j_2$ because it is the base-change of $j_2$ along $j_1$. 
	
	Towards the end, we would like to pass to the generic hyperplane. We write $k_m$ for the function field of $X_m$. We construct the restriction to the generic hyperplane: We have morphisms
	\[
	H_{k_m} \rightarrow H_m \rightarrow \BA^n.
	\]
	We denote pullback along this morphism by $K \mapsto K_{k_{m}}$. 
%
%
\end{remark}
We deduce Beilinson's basic lemma from the following lemma: 
\begin{lemma}\label{hypplanebasiclemma_LEM}
	Let $k$ be an algebraically closed field and $n \geq 0$. Let $K \in \derCat{c}{\BA^n_k}{}$. Let $i\colon H \rightarrow \BA^n_k$ be a generic affine hyperplane. The natural map
	\[
	i^*(K)(-1)[-2] \rightarrow i^!(K)
	\]
	is an isomorphism.
\end{lemma}
\begin{proof}
	The morphism
	\[
	\pi\colon H_1 \rightarrow \BA^n
	\]
	is smooth. Let $i'\colon H \rightarrow H_1$ be the inclusion of a hyperplane and $i\colon H \rightarrow \BA^n$ the inclusion into $\BA^n$. We are considering the canonical map
	\[
	i^*(K)\otimes i^!(\Qbarl) \rightarrow i^!(K).
	\]
	We can write $i = i'\pi$ which means that the terms are given by
	\[
	(i'\pi)^*(K)\otimes (i'\pi)^!(\Qbarl) \rightarrow (i'\pi)^!(K)
	\]
	Now note that we can also take the composite
	\begin{align*}
	(i'\pi)^*(K)\otimes (i'\pi)^!(\Qbarl) &= i'^*(\pi^*(K))\otimes (i')^!(\pi^!(\Qbarl)) \\&\rightarrow (i')^!(\pi^*(K)\otimes\pi^!(\Qbarl)) \\&\xrightarrow{\sim} (i')^!(\pi^!(K))
	\\& = (i'\pi)^!(K)
	\end{align*}
	Note that the second arrow is an isomorphism because $\pi$ is smooth. Moreover, this composite is equal to the composite above. Thus it is sufficient if we show that
	\[
	i'^*(\pi^*(K))\otimes (i')^!(\Qbarl) \rightarrow (i')^!(\pi^*(K))
	\]
	is an isomorphism.
	
	Note that we can write $H$ as a fiber of the morphism
	\[
	H_1 \rightarrow X_1.
	\]
	By \cite[Prop.~2.3]{Lu2022}, the canonical map for $i'$ is an isomorphism for $H \subset \BA^n$ if $\pi^*(K)$ is universally locally acyclic over $X_1$ at each point in the fiber of $H_1 \rightarrow X_1$ over $H$. By Theorem \ref{genbc_THM}, we see that a generic hyperplane satisfies the isomorphism from the lemma. 
\end{proof}
We begin by proving that a certain map is surjective:
\begin{lemma}\label{hypplanesurj_LEM}
	Let $k$ be an algebraically closed field, $n \geq 0$, and $m \geq 0$. Let $M \in \perv{\BA^n}$. Let $j\colon \BA^n - H \rightarrow \BA^n$ be the complement of a generic affine $m$-hyperplane. The morphism
	\[
	j_!(M) \rightarrow M
	\]
	is a surjection of perverse sheaves.
\end{lemma}
\begin{proof} If $m = 0$, then $j$ is the identity. Thus we may assume $m > 0$.
	
	Suppose $m = 1$. The morphism is a surjection whenever $\pervCoh{0}{i^*(M)} = 0$. Lemma \ref{hypplanebasiclemma_LEM} shows we may assume
	\[
	i^*(M)(-1)[-2] = i^!(M).
	\] 
	Thus we obtain
	\begin{align*}
		\pervCoh{0}{i^*(M)} = \pervCoh{2}{i^!(M)} = 0
	\end{align*}
	where the last vanishing follows from the perversity of $M$ and Artin's vanishing theorem. Thus we see that the above morphism is surjective.
	
	Suppose $m > 1$. Let $j\colon \BA^n - H_1 \rightarrow \BA^n$ be the open immersion associated to the complement of a generic affine hyperplane. Let $j_2\colon \BA^n - H_2 \rightarrow \BA^n$ be an the complement of a generic affine  $(m - 1)$-hyperplane. The complement of the union $H_1\cup H_2$ is given by
	\[
	j = j_1j_2 = j_2j_1.
	\]
	The above morphism can be given as a composite
	\[
	j_{1!}(j_{2!}(M)) = j_{2!}(j_{1!}(M)) \rightarrow j_{1!}(M) \rightarrow M.
	\]	
	We can show the first morphism is surjective by applying the hyperplane case and the second morphism by the induction hypothesis.
\end{proof}
\begin{lemma}\label{hypplanecommute_LEM}
	Let $k$ be an algebraically closed field and $n \geq 0$. Let $K \in \derCat{c}{\BA^n}{}$ and $j_1\colon U_1 \rightarrow \BA^n$ an open immersion. Let $i_2\colon H_2 \rightarrow \BA^n$ be the inclusion of a generic affine hyperplane. We put $U_2 := \BA^n - H_2$. 
	Then we have 
	\[
	j_{1!}j_{2*}(K) = j_{2*}j_{1!}(K)
	\]
	via the natural base change map.
\end{lemma}
\begin{proof}
	Note that we have
	\begin{align*}
		i^!(K) &= i^*(K)(-1)[-2] \\
		i^!j_!(K) & =i^*j_!(K)(-1)[-2].
	\end{align*}
	because the hyperplane $H$ is generic. We have the exact triangle
	\[
	i^!(K) \rightarrow K \rightarrow k_{*}(K) \xrightarrow{+}.
	\]
	We apply $j_!$ to obtain the triangle
	\[
	j_!i^!(K) \rightarrow j_!(K) \rightarrow j_!k_*(K) \xrightarrow{+}.
	\]
	Now we have
	\[
	j_!i^!(K) = j_!i^*(K)(-1)[-2] = i^*j_!(K)(-1)[-2] = i^!j_!(K).
	\]
	We can plug this into the triangle above and obtain
	\[
	i^!j_!(K) \rightarrow j_!(K) \rightarrow j_!k_*(K)\xrightarrow{+}
	\]
	Thus we obtain an isomorphism
	\[
	j_!k_*(K) = k_*j_!(K) 
	\]	
	from the five lemma. 
%
%
	
\end{proof}
\begin{lemma}\label{hypplanenonvanish_LEM}
	Let $k$ be an algebraically closed field and $n \geq 0$. Let $M \in \perv{\BA^n}$ be non-zero. For a generic $m$-hyperplane $H \subset \BA^n$, where $m \geq n$, we have 
	\[
	H^0(\BA^n - H, M) \neq 0.
	\]
\end{lemma}
\begin{proof}
	Note that once we have established the case $m = n$, the other cases follow trivially by considering the pushforward to the complement of a $n$-hyperplane.
	
	We argue by induction on $n$. If $n = 0$, the statement is trivial. Let $i_1\colon H_1 \rightarrow \BA^n$ be a generic hyperplane and $i_2\colon H_2 \rightarrow \BA^n$ a generic $n$ hyperplane. We denote by $j_1$ and $j_2$ the inclusions of the complement into $\BA^n$. We denote by $j$ and $i$ the open immersion/closed immersion associated to the union of $H_1\cup H_2$. We distinguish two cases.
	
	Suppose $i_1^*(M) = 0$. Then $M$ is punctual. In this case, the statement is true whenever the $n$-hyperplane avoids one point in the support of $M$. This is true for all hyperplanes outside a proper closed $X_m$, since not containing a point is a condition which gives a dense open subset of $X_m$.
	
	Thus we may assume $i_1^*(M) \neq 0$. We have
	\[
	i_1^*(M)(-1)[-2] = i_1^!(M).
	\]
	Artin's vanishing theorem implies $i_1^*(M)[-1]$ is perverse. We have an exact sequence of perverse sheaves
	\[
	0 \rightarrow M \rightarrow j_{1*}(M) \rightarrow \pervCoh{-1}{i_1^*(M)} \rightarrow 0.
	\]
	The functor $j_{2*}$ is $t$-exact by Artin's vanishing theorem, so we obtain a short exact sequence of perverse sheaves
	\[
	0 \rightarrow j_{2*}(M) \rightarrow j_*(M) \rightarrow {j_2}_*\big(\pervCoh{-1}{i_1^*(M)}\big) \rightarrow 0.
	\]
	Artin's vanishing theorem implies that we obtain a surjection
	\[
	H^0(U, M) \rightarrow H^0\big(H_1, j_{2*}\big(\pervCoh{-1}{i_1^*(M)}\big)\big)
	\]
	We know the term on the right is non-zero for a generic hyperplane $H_2$ by induction on dimension, hence the claim. 
\end{proof}
\begin{lemma}\label{hypplanecollect_LEM}
	Let $k$ be a field, $n \geq 0$, and consider $M \in \perv{\BA^n}$. Let $m \geq 1$. Let $H_1$ be a generic affine hyperplane and $H_2$ a generic affine $(m - 1)$-hyperplane. 
	\begin{enumerate}
		\item 
		We have
		\[
		j_{1*}j_{2!}(M) = j_{2!}j_{1*}(M) .
		\]
		\item If $m \geq n$ and $M$ is non-zero, then we have
		\[
		H^0(U_2, M) \neq 0.
		\]
		\item The morphism
		\[
		{j_{1!}}(M) \rightarrow M
		\]
		is a surjection of perverse sheaves. 
	\end{enumerate}
\end{lemma}
\begin{proof}
	The first point follows from Lemma \ref{hypplanecommute_LEM}, the second point follows from Lemma \ref{hypplanenonvanish_LEM}, and the third point is given by Lemma \ref{hypplanesurj_LEM}.
\end{proof}
\begin{theorem}\label{hypplanefunctor_THM}
	Let $k$ be a field and $n \geq 1$. For each $m > n$, there exists an exact, $\Qbarl$-linear, and faithful functor
	\[
	\alpha_m\colon \perv{\BA^n_k} \rightarrow \Qbarl-\text{fdVec}.
	\]
	The functor is given by
	\[
	\alpha_m(M) := H^0(U_2, {j_1}_!(M_{k_{m}})).
	\]
\end{theorem}
\begin{proof}
	Suppose $M \in \perv{\BA^n}$ is non-zero. The morphism
	\[
	{j_1}_!(M_{k_{m}}) \rightarrow M_{k_m}
	\]
	is surjective by point (3) in Lemma \ref{hypplanecollect_LEM}. Therefore, Artin vanishing implies that we have a surjection
	\[
	H^0(U_2, {j_1}_!(M_{k_m})) \rightarrow H^0(U_2, M_{k_m}).
	\]
	The term on the right does not vanish by point (2) of Lemma \ref{hypplanecollect_LEM}, so we obtain $\alpha_m(M) \neq 0$. Therefore, the above functor is faithful. Note that we have
	\[
%
	H^*(U_2, j_{1!}(M))= H^*(\BA^n, j_{2*}j_{1!}(M)) = H^*(\BA^n, j_{1!}j_{2*}(M)) = H^*_c(U_1, j_{2*}(M))
	\]
	by point (1) in Lemma \ref{hypplanecollect_LEM}. Artin's vanishing theorem implies that both sides have to be concentrated in degree zero. Therefore, the functor $\alpha_m$ is exact.
\end{proof}

\begin{proof}[Proof. (Thm. \ref{fibfuncperv_THM})]
	We can find a cover $X = U_1\cup \ldots \cup U_n$ by affine opens with inclusions $j_i\colon U_i\rightarrow X$. Given functors $\alpha_i\colon \perv{U_i} \rightarrow \Qbarl-\text{fdVec}$ for $1 \leq i \leq n$ as in the theorem, we can form the direct sum
	\[
	\alpha(M) := \bigoplus_{i = 1}^n \alpha_i(j^*_i(M)).
	\]
	Since a perverse sheaf $M$ is zero if and only if $j_i^*(M)$ is zero for all $1\leq i\leq n$, this functor is faithful. It is exact because the functor $j^*_i$ is exact for $1\leq i \leq n$. 
	
	Thus we may assume $X$ is affine. We can find a closed immersion $i\colon X \rightarrow \BA^n$ for some $n \geq 0$ because $X$ is of finite type over $k$. We form the composite with the functor for $\BA^n$ constructed in Theorem \ref{hypplanefunctor_THM}
	\[
	\perv{X} \xrightarrow{i_*} \perv{\BA^n} \xrightarrow{\alpha_{n + 1}} \Qbarl-\text{fdVec}.
	\]
	Since $i_*$ and $\alpha_{n + 1}$ are exact, faithful, and $\Qbarl$-linear, the composite satisfies these conditions as well. 
\end{proof}

	\subsection{Monodromy filtration}
	We often make use of the following lemma.
	\begin{lemma}
		Let $\CA$ be an abelian category and $N\colon M \rightarrow M$ a nilpotent endomorphism of an object $M \in \CA$. Let $F\colon \CA\rightarrow \CB$ be an exact functor of abelian categories. The formation of the monodromy filtration commutes with $F$, i.e. we have
		\[
		F(\text{Fil}^M_i(M)) = \text{Fil}^M_i(F(M))
		\]
		and
		\[
		F(\text{Gr}^M_i(M)) = \text{Gr}^M_i(F(M)).
		\]
	\end{lemma} 
	\begin{proof}
		Since the functor $F$ is exact, the formation of the graded commutes with $F$. Let $G^i(F(M))$ be the filtration on $F(M)$ induces by the monodromy filtration of $M$. The nilpotent operator $F(N)$ induces isomorphism
		\[
		F(N)^i\colon \text{Gr}^G_{i}(F(M)) \rightarrow \text{Gr}^G_{-i}(F(M)).
		\]
		for all $i \geq 0$. By uniqueness of the monodromy filtration (see \cite[Prop.~1.6.1]{DeligneWeilII}), we see that $G^i(F(M))$ is the monodromy filtration of $F(M)$.
	\end{proof}
	For example, it allows us to prove the following proposition:
	\begin{proposition}\label{tensorformula_PROP} Let $K$ be a field of characteristic zero. Let $\CA$ be an abelian category enriched over $K$-vector spaces. We assume $\CA$ admits an exact and faithful $K$-linear functor
	\[
		\alpha\colon \CA \rightarrow K-\text{fdVec}.
	\]
	Let $M \in \CA$ be an object with a nilpotent operator $N\colon M(1) \rightarrow M$. Let $V$ be a finite dimensional $K$-vector space with a nilpotent operator $N'\colon V(1) \rightarrow V$. We equip $M\otimes V$ with the monodromy operator $N\otimes 1 + 1\otimes N'$. The monodromy filtration is given by the tensor product formula
		\[
		\text{Fil}^M_k(M\otimes V) = \sum_{i + j = k} \text{Fil}^M_i(M)\otimes \text{Fil}^M_j(V). 
		\]
		The graded pieces of the monodromy filtration satisfy 
		\[
		\text{Gr}^M_k(M\otimes V) = \sum_{i + j = k} \text{Gr}^M_i(M)\otimes \text{Gr}^m_j(V). 
		\]
		Let $P_i, P'_j, P''_k$ be the primitive parts of the monodromy filtrations on $M$, $V$, and $M\otimes V$ respectively. Then we have
		\[
		P''_{k} = \bigoplus_{(i, j) \in I_k} P_{i}\otimes P'_{j} ({k - i - j}/{2}).
		\]
		where
		\[
		I_j := \big\{(j', j'') ~\big|~ {|j' - j''| \leq j \leq |j' + j''|,~ j \equiv j' + j''\text{ mod }2}\big\}
		\]
	\end{proposition}
	\begin{proof}
		The formation of the monodromy filtration commutes with exact functors. Note that we have
		\[
		\alpha(M\otimes V) = \alpha(M)\otimes V
		\]
		because the functor is $\Qbarl$-linear. Thus we may apply the tensor product formula of \cite[Prop.~1.6.9]{DeligneWeilII} to $\alpha(M\otimes V)$. One can now deduce the three claims because the functor is faithful and exact.
		
		Let us indicate how to prove that the monodromy filtration and the tensor product filtration agree. We are given two filtrations on $M\otimes V$ such that the induced filtrations on $\alpha(M\otimes V)$ agree. This implies that the filtrations of $M\otimes V$ are equal because the functor is faithful and exact. 
		
		To determine the graded pieces and the primitive parts, one can construct morphisms of objects in $\CA$ and verify that they are isomorphisms by applying $\alpha$. 
	\end{proof}
\subsection{Generic base change theorems} The goal of this section is to prove generic base change, generic universal local acyclicity, and generic existence of an internal mapping object for constructible proetale complexes on qcqs schemes. We prove this theorem by reducing to the case of discrete coefficients, where these theorems are known.
\subsubsection{Generic base-change}

The goal of the following remark is to clarify the definitions used in the subsequent theorems.  
\begin{remark}\label{genbcprep_REM}
	Let $X$ be a qcqs scheme. In the following theorems, we consider constructible complexes on $X$ over one of two base rings $R$. In the first case, $R$ is a Noetherian ring with an ideal $I \subset R$. We assume $R$ is complete with respect to the ideal $I$ and we equip $R$ with the $I$-adic topology. We also assume the kernel of the canonical morphism $\BZ \rightarrow R$ is invertible on $X$.
	
	In the second case, we consider the ring $R = \Qbarl$. In this case, we can write
	\[
	\Qbarl = \varinjlim K
	\]  
	as a filtered colimit over finite extensions $K/\BQ_\ell$. The fields $K$ are locally compact Hausdorff spaces and we equip $\Qbarl$ with the colimit topology. In both cases, a category of constructible complexes is constructed in \cite{hemo2023constructiblesheavesschemes}.
	
	We remark that  \cite[Prop.~5.2]{hemo2023constructiblesheavesschemes} and \cite[Cor.~2.4]{hansen2023relativeperversity} show that for each $K \in \text{D}_c^b(X, \Qbarl)$, there is a finite extension $K/\BQ_\ell$ with ring of integers $\CO$, and a complex $K_0 \in \text{D}_c^b(X, \CO)$ with $K_0\otimes \Qbarl = K$. This will allow us to reduce statements for the coefficient ring $R = \Qbarl$ to the Noetherian rings considered above. 
	
	In the first case, we remark that \cite[Prop.~7.6]{hemo2023constructiblesheavesschemes} and \cite[Prop.~6.1]{hemo2023constructiblesheavesschemes} show that the category of constructible complexes is equivalent to the category of constructible complexes constructed in \cite[Def.~6.5.1]{BhattSPE}. We will work primarily with the latter definition.

	For the pushforward, we consider the pushforward in the proetale topology. For the pullback, we take the completed pullback defined as the $I$-derived completion of the pullback. For the ring $R = \Qbarl$, we can utilize the equivalences \cite[Prop.~5.2]{hemo2023constructiblesheavesschemes} and \cite[Cor.~2.4]{hansen2023relativeperversity} to construct a completed pullback functor by localization.
\end{remark}
We can now state the first theorem of this section: 
\begin{theorem}\label{genbc_THM}
	Let $f\colon X \rightarrow Y$ be a finitely presented morphism of qcqs schemes. We assume one of the following two conditions on the coefficient ring $R$:
	\begin{enumerate}
		\item The ring $R$ is a Noetherian ring complete with respect to an ideal $I \subset R$. We assume the characteristic of $R/I$ is invertible on $Y$.
		\item The ring $R$ is $R = \Qbarl$ for a prime $\ell$ invertible on $Y$.
	\end{enumerate}
	Let $K \in \text{D}^b_c(X, R)$. There exists a finite constructible partition
	\[
	Y = \bigsqcup_{i = 1}^n Y_i
	\]
	such that all $1 \leq i \leq n$ satisfy: 
	\begin{enumerate}
		\item The pushforward $f_*(K|_{Y_i})$ is constructible over $Y_i$.
		\item The formation of the pushforward $f_*(K|_{Y_i})$ commutes with any base-change $Y' \rightarrow Y_i$.
	\end{enumerate}
\end{theorem}
Before we give the proof, we prove a lemma and recall the definition of the Koszul complex.
\begin{definition}
	Let $\mathfrak{X}$ be a topos and $R \in \mathfrak{X}$ a ring object. For each $x \in \mathfrak{X}(R)$, we define the \textit{Koszul complex} $R(x)$ to be the cone of the morphism $$R \xrightarrow{x} R$$ of complexes concentrated in degree zero. For elements $x_1, \ldots,x_n \in \mathfrak{X}$, we define the Koszul complex to be the totalized tensor product
	\[
	R(x_1, \ldots, x_n) := R(x_1)\otimes\cdots\otimes R(x_n).
	\]
\end{definition}
\begin{remark}\label{koszulcom_REM}
	Let $R$ be a commutative ring and $D$ a triangulated category enriched over $R$-modules. Let $x \in R$. For each $K \in D$, we can define $K\otimes R(x)$ to be the cone of the morphism
	\[
	K \xrightarrow{x} K \rightarrow K\otimes R(x) \xrightarrow{+}
	\]
	
	Let $F\colon D \rightarrow D'$ be an exact $R$-linear functor of triangulated categories enriched over $R$-modules. Note that we can find an isomorphism
	\[
	F(K\otimes R(x)) \cong F(K)\otimes R(x)
	\]
	for all $K \in D$ because the functor is exact and $R$-linear. In particular, if $x_1, \ldots, x_n \in R$ are elements, we can find an isomorphism
	\[
	F(K\otimes R(x_1, \ldots,x_n)) \cong F(K)\otimes R(x_1, \ldots, x_n)
	\]
	by induction on $n$. We see that tensoring by a Koszul complex commutes with any exact $R$-linear functor. 
\end{remark}
\begin{remark}\label{genbcqcqs_REM}
	During the proofs of the following theorems, we need the generic base-change theorem, the generic universal local aciclicity theorem, and the generic constructibility theorem for the internal mapping object to hold for constructible etale sheaves over qcqs bases. We were only able to find the theorems over Noetherian schemes in the literature. One can reduce the case of qcqs bases to Noetherian bases by \cite[Cor.~2.7.4]{SGA4}. We give a detailed example of such a reduction during the proof of the generic universal local acyclicity theorem. We also give the reduction during the proof of the generic constructibility of the pushforward.
\end{remark}
\begin{lemma}\label{constlemma_LEM}
	Let $X$ be a qcqs scheme and $R$ an Noetherian ring, complete with respect to an ideal $I \subset R$. Let $K \in \text{D}^b_c(X, R)$ be a bounded constructible complex of $R$-modules. Let $L \in \text{D}^b_{\text{coh}}(R)$ be a bounded complex with finite cohomology modules. We assume that the cohomology modules of $L$ are $I$-power torsion. The complex $K\otimes L$ is classical. Moreover, all cohomology sheaves of $K\otimes L$ are constructible.
\end{lemma}
\begin{proof} 
	The cohomology sheaves of the bounded complex $L$ are finitely generated $I$-power torsion modules. Since $L$ is bounded, the complex $L$ admits a finite filtration by complexes of the form $M'[e]$, where $M'$ is a $R/I$-module and $e \in \BZ$. We have
	\[
	K\otimes_R M' = (K\otimes_R R/I)\otimes_{R/I} M'.
	\]
	Thus we see that $K\otimes_R M'$ is classical and the cohomology sheaves of $K\otimes_R M'$ are constructible. It follows from \cite[Cor.~5.1.9]{BhattSPE} and \cite[Prop.~5.2.9]{BhattSPE} that classical complexes are closed under cones for bounded below complexes. Since $K\otimes R/I$ is tor-finite, we obtain that all the terms $K\otimes_R M'$ are bounded. Thus $K\otimes_R L$ is classical.  Moreover, the cohomology sheaves of $K\otimes L$ are constructible because the category of constructible sheaves is closed under exact sequences.
\end{proof}
\begin{proof}[Proof.~(Thm.~\ref{genbc_THM})]
	By Remark \ref{genbcprep_REM}, we may assume $R$ to be a Noetherian ring complete with respect to an ideal $I \subset R$. 
	
	The statement (2) requires a set of morphisms between pushforwards to be isomorphisms. The formation of these morphisms commutes with tensoring with the Koszul complex $R(x_1, \ldots, x_n)$ by Remark \ref{koszulcom_REM}. Moreover, the pushforward of an $I$-derived complete complex is always $I$-derived complete by \cite[Lem.~6.5.9]{BhattSPE}. Therefore, we may apply Nakayama's Lemma (see \cite[\href{https://stacks.math.columbia.edu/tag/0H83}{Tag 0H83}]{stacks-project}) to show the morphisms are isomorphism. Thus it is sufficient to verify the isomorphism after taking the derived tensor product with the Koszul complex. Thus we may pass to $K\otimes R(x_1, \ldots, x_n)$. We remark that the cohomology sheaves of $K \otimes R(x_1, \ldots,x_n)$ are classical and constructible sheaves of $R/I$-modules by Lemma \ref{constlemma_LEM}. 
	
	
	For statement (2), the comparison isomorphism in \cite[Lem.~5.4.3]{BhattSPE} shows that the pushforward of the cohomology sheaves of $K\otimes R(x_1, \ldots,x_n)$ can be computed in the etale topology. Note that the completed pullback tensored with a Koszul complex reduces to an ordinary pullback tensored with a Koszul complex by definition of the completion functor. Therefore, the comparison isomorphism in \cite[Lem.~5.4.1]{BhattSPE} shows that the pullback can be computed in the etale topology. To prove that the pushforward commutes with base-change, it is therefore sufficient to prove that the pushforward of the cohomology sheaves of $K\otimes R(x_1, \ldots,x_n)$ commutes with base-change in the etale-topology.
	
	Thus we may assume $K$ is a constructible sheaf on the etale site of $X$. In this case, Noetherian approximation (see \cite[Cor.~2.7.4]{SGA4}) allows us to assume the sheaf descends to a base scheme $Y_0$ of finite type over $\BZ[1/\ell]$. The property of satisfying base change is preserved by base change, so we may replace $Y$ by $Y_0$ and assume $Y$ is of finite type over $\BZ[1/\ell]$. In this case, the existence of the stratification can be deduced from Noetherian induction and \cite[Ch.~7,~Cor.~2.9]{DeligneSGA45}.
		
	We turn to proving the statement (1). Our first step is to show that we have an isomorphism
	\[
	f_*(K)\otimes R/I \xrightarrow{\sim} f_*(K\otimes R/I).
	\] 
	Note that both sides of this morphism are $I$-derived complete complexes because they are annihilated by $I$. Thus it is sufficient if we show that
	\[
	(f_*(K)\otimes R/I) \otimes R(x_1, \ldots, x_n) \rightarrow f_*(K\otimes R/I)\otimes R(x_1, \ldots, x_n)
	\]
	is an isomorphism by Nakayama's Lemma. This morphism identifies with
	\[
	f_*(K\otimes R(x_1, \ldots, x_n))\otimes R/I \rightarrow f_*((K\otimes R(x_1, \ldots, x_n))\otimes R/I)
	\]
	by Remark \ref{koszulcom_REM}. Note that $K\otimes R(x_1, \ldots, x_n)$ is classical by Lemma \ref{constlemma_LEM}. Since we can compute the pushforward and the tensor product with $R/I$ in the etale topology, it is sufficient to show the isomorphism in the etale topology. In particular, it is sufficient if we show it over each geometric fiber of $X$ over $Y$ because the collection of all geometric points on $X$ is faithful in the etale topology. By (2), we may assume the formation of $f_*(K\otimes R(x_1, \ldots, x_n))$ and $f_*((K\otimes R(x_1, \ldots, x_n))\otimes R/I)$ commutes with base-change. In particular, we know the above morphism is an isomorphism over each geometric fiber of $X$ over $Y$ by \cite[Lem.~6.5.11~(3)]{BhattSPE}. Thus we obtain the desired isomorphism. 
	
	The isomorphism implies $f_*(K)\otimes R/I$ is classical. Therefore, it is sufficient if we show that the pushforward $f_*(K\otimes R/I)$, formed in the etale topology, is constructible. Point (2) and Noetherian approximation allow us to assume $Y$ is of finite type over $\BZ[1/\ell]$. In this case, the reference \cite[Ch.~7,~Cor.~2.9]{DeligneSGA45} shows that the cohomology sheaves of $f_*(K\otimes R/I)$ are constructible. To prove that the complex $f_*(K\otimes R/I)$ is constructible, all that is left to show is that the stalks are perfect complexes. Since we may assume the formation of $f_*(K\otimes R/I)$ commutes with base-change, we may reduce to $Y$ a field. In this case, this follows from the discrete coefficient case over fields treated in \cite[Thm.~9.5.2]{LeiFu}. 
\end{proof}
\subsubsection{Generic universal local acyclicity}
We establish the generic universal local acyclicity theorem. We make a remark on the definition of universal local acyclicity: 
\begin{remark}
	Universal local acylicity is defined in \cite{hansen2023relativeperversity}. Note that loc. cit. does not define universal local acyclicity for the coefficient rings we consider below. The argument below proves universal local acylicity with respect to the definition \cite[Thm.~1.6,~(ii)]{hansen2023relativeperversity}. This definition agrees with the classical definition when the coefficient ring is discrete (see \cite[p.~378,~Ch.~7.6]{LeiFu}).
\end{remark}
\begin{theorem}\label{genula_THM}
	Let $f\colon X \rightarrow Y$ be a finitely presented morphism of qcqs schemes. We assume one of the following two conditions on the coefficient ring $R$:
	\begin{enumerate}
		\item The ring $R$ is Noetherian and complete with respect to an ideal $I \subset R$. We assume the characteristic of $R/I$ is invertible on $Y$.
		\item The ring $R$ is $R = \Qbarl$ for a prime $\ell$ invertible on $Y$.
	\end{enumerate}
	Let $K \in \text{D}^b_c(X, R)$. There exists a finite constructible partition
	\[
	Y = \bigsqcup_{i = 1}^n Y_i
	\]
	such that $K|_{Y_i}$ is universally locally acyclic over $Y_i$.
\end{theorem}
\begin{proof}
	We can assume $R$ is a Noetherian ring complete with respect to an ideal $I \subset R$ by Remark \ref{genbcprep_REM}. We denote by $x_1, \ldots, x_n \in I$ a set of generators of $I$. Note that the condition of being ULA involves a set of morphisms between cohomology groups of $I$-derived complete complexes on qcqs schemes being isomorphisms. Cohomology groups of an $I$-derived complete complex are $I$-derived complete by \cite[Lem.~6.5.9~(ii)]{BhattSPE}. By Nakayama's lemma (see \cite[\href{https://stacks.math.columbia.edu/tag/0H83}{Tag 0H83}]{stacks-project}) and Remark \ref{genbcqcqs_REM}, it is sufficient if we prove that such a partition exists for $K\otimes R(x_1, \ldots, x_n)$. Note that the cohomology sheaves of $K\otimes R(x_1, \ldots,x_n)$ are constructible and classical by Lemma \ref{constlemma_LEM}. Thus we may assume $K$ is a classical constructible sheaf over a discrete coefficient ring with characteristic invertible on $Y$. The comparison isomorphism in \cite[Lem.~5.4.3]{BhattSPE} shows it is sufficient to show that the cohomology sheaves of $K\otimes R(x_1, \ldots, x_n)$ are ULA over a constructible partition of $Y$ in the etale topology.
	
	In this case, Noetherian approximation (see \cite[Cor.~2.7.4]{SGA4}) allows us to assume the sheaf descends to a base scheme $Y_0$ of finite type over $\BZ[1/\ell]$. Since the property of being ULA is, by definition, preserved by base-change, we may assume $Y$ is of finite type over $\BZ[1/\ell]$. 
	
	In this case, Noetherian induction on $Y$ implies it is sufficient to prove that there exists a non-empty open subscheme of $Y$ over which $K$ is ULA. This statement is proven in \cite[Ch.~7,~Thm.~2.13]{DeligneSGA45}.
\end{proof}
\subsubsection{Constructible internal mapping object} We retain the definitions of Remark \ref{genbcprep_REM}. 
The goal of this section is to prove the following theorem:
\begin{theorem}\label{genbchom_THM}
	Let $f\colon X \rightarrow Y$ be a finitely presented morphism of qcqs schemes. We assume one of the following two conditions on the coefficient ring $R$:
	\begin{enumerate}
		\item The ring $R$ is a complete Noetherian local ring. We assume the characteristic of the residue field of $R$ is invertible on $Y$.
		\item The ring $R$ is $\Qbarl$ for a prime $\ell$ invertible on $Y$.
	\end{enumerate}
	Let $K, L \in \text{D}^b_c(X, R)$ be two constructible complexes. There exists a constructible finite partition
	\[
	Y = \bigsqcup_{i = 1}^n Y_i
	\]
	such that all $1 \leq i \leq n$ satisfy: 
	\begin{enumerate}
		\item The internal mapping object $\text{Hom}(K|_{Y_i}, L|_{Y_i}) \in \text{D}^b_c(X_{Y_i}, R)$ exists. 
		\item The formation of the internal mapping object $\text{Hom}(K|_{Y_i}, L|_{Y_i})$ commutes with base-change $Y' \rightarrow Y_i$. 
	\end{enumerate}
\end{theorem}
\begin{remark}
	Let $X$ be a qcqs scheme and $\Lambda$ a discrete coefficient ring. We consider classical complexes $K, L \in \text{D}(X, \Lambda)$ in the proetale topology on $X$. During the proof of the above theorem, we will utilize the classical internal mapping object. This is the internal mapping object of $K$ and $L$ formed in the etale topology and then sheafified in the proetale topology. We write $\text{Hom}_{\Lambda, \text{cl}}(K, L)$ for this object.
\end{remark}
%
The classical internal mapping object behaves like an internal mapping object: 
\begin{lemma}\label{genbcinternalhom_LEM}
	Let $X$ be a qcqs scheme and $\Lambda$ a discrete Noetherian ring. Let $K, L, L' \in \text{D}^b(X, \Lambda)$ be bounded classical complexes of $\Lambda$-modules on the proetale site of $X$. We assume $K$ is tor-finite. We have a natural isomorphism
	\[
	\text{Hom}_{\Lambda}(K, \text{Hom}_{\Lambda, \text{cl}}(L, L')) = \text{Hom}_{\Lambda}(K\otimes_{\Lambda} L, L')
	\]
	of objects in $\text{D}(\Lambda)$.
\end{lemma}
\begin{proof}
	The comparison functor is fully faithful for bounded below complexes by \cite[Prop.~5.2.6]{BhattSPE}. Thus we obtain
	\begin{align*}
		\text{Hom}_{\Lambda}(K, \text{Hom}_{\Lambda, \text{cl}}(L, L')) &= \text{Hom}_{\Lambda, \text{cl}}(K, \text{Hom}_{\Lambda, \text{cl}}(L, L')) \\
		&= \text{Hom}_{\Lambda, \text{cl}}(K\otimes_{\Lambda, \text{cl}} L, L') \\
		&= \text{Hom}_{\Lambda}(K\otimes_{\Lambda, \text{cl}} L, L').
	\end{align*}
	We have
	\begin{align*}
		K\otimes_{\Lambda, \text{cl}} L = K\otimes_{\Lambda} L,
	\end{align*}
	hence the claim. 
\end{proof}
We also require the following lemma, which is the reason we restrict to local rings: 
\begin{lemma}\label{genbchomcommutes_LEM}
	Let $R$ be a complete Noetherian local ring. Let $f\colon X \rightarrow Y$ be a finitely presented morphism of qcqs schemes. Let $K, L \in \text{D}^b_c(X, R)$ be constructible complexes on $X$. Suppose the formation of $\text{Hom}_{R/\idm, \text{cl}}(K\otimes R/\idm, L\otimes R/\idm)$ commutes with base change over $Y$. Let $n \geq m\geq 1$. We have
		\[
		\text{Hom}_{R/\idm^n, \text{cl}}(K\otimes R/\idm^n, L\otimes R/\idm^n)\otimes R/\idm^m = \text{Hom}_{R/\idm^m, \text{cl}}(K\otimes R/\idm^m, L\otimes R/\idm^m).
		\]
	Moreover, the formation of 
	\[
	\text{Hom}_{R/\idm^n, \text{cl}}(K\otimes R/\idm^n, L\otimes R/\idm^n)
	\]
	commutes with base-change.
\end{lemma}
\begin{proof}
	Let $n \geq 1$. We can filter $R/\idm^n$ by $R$-submodules isomorphic to $R/\idm$ because $R/\idm$ is a field. With this filtration, proving that the formation of the internal mapping object $\text{Hom}_{R/\idm^n, \text{cl}}(K\otimes R/\idm^n, L\otimes R/\idm^n)$ commutes with base-change can be reduced to proving that the formation of $\text{Hom}_{R/\idm^n, \text{cl}}(K\otimes R/\idm^n, L\otimes R/\idm)$ commutes with base-change. We have an isomorphism
	\[
	\text{Hom}_{R/\idm^n, \text{cl}}(K\otimes R/\idm^n, L\otimes R/\idm) = \text{Hom}_{R/\idm, \text{cl}}(K\otimes R/\idm, L\otimes R/\idm).
	\]
	
	To prove the coefficient formula, we note that the base-change implies we may assume $Y$ is a field. In this case, the statement is proven in the proof of \cite[Prop.~10.1.17]{LeiFu}
\end{proof}
\begin{proof}[Proof.~({Thm.~\ref{genbchom_THM}})] By Remark \ref{genbcprep_REM}, we may assume $R$ is a complete Noetherian local ring.  We define 
	\[
	\text{Hom}(K, L) := \varprojlim \text{Hom}_{R/\idm^n, \text{cl}}(K\otimes R/\idm^n, L\otimes R/\idm^n).
	\]
This object is well-defined by Lemma \ref{constlemma_LEM}. Moreover, it is $\idm$-derived complete because it is a limit of $\idm$-derived complete complexes. Let $K' \in \text{D}^b_c(X, R)$ be a constructible complex. If $\text{Hom}(K, L)$ is constructible, then Lemma \ref{genbcinternalhom_LEM} implies that we have isomorphisms 
\begin{align*}
	\text{Hom}_R(K', \text{Hom}(K, L)) &= \varprojlim \text{Hom}_R(K', \text{Hom}_{R/I^n, \text{cl}}(K\otimes R/I^n, L\otimes R/I^n)) \\ 
	&= \varprojlim \text{Hom}_{R/I^n}(K'\otimes R/I^n, \text{Hom}_{R/I^n, \text{cl}}(K\otimes R/I^n, L\otimes R/I^n)) \\
	&= \varprojlim \text{Hom}_{R/I^n}((K'\otimes_R R/I^n)\otimes_{R/I^n}(K \otimes_R R/I^n), L\otimes_R R/I^n)\\
	&= \varprojlim \text{Hom}_{R/I^n}(K'\otimes_R K \otimes_R R/I^n, L\otimes R/I^n)\\
	&= \text{Hom}_R(K'\otimes_R K, L) \\
\end{align*}
Therefore, if $\text{Hom}(K, L)$ is constructible, then it is the internal mapping object of $K$ and $L$ in the category $\text{D}^b_c(X, R)$. Our next step is to find a finite stratification of $Y$ over which the complexes are constructible. 

By \cite[Ch.~7,~Cor.~2.9]{DeligneSGA45} and Remark \ref{genbcqcqs_REM}, we may pass to a stratum in a stratification of $Y$ and assume that the formation of $\text{Hom}_{R/\idm, \text{cl}}(K\otimes R/\idm, L\otimes R/\idm)$ commutes with base change in the etale topology. Since the pullback in the etale and the proetale topology are isomorphic by \cite[Lem.~5.4.1]{BhattSPE}, we see that the formation of the classical internal mapping object commutes with base-change in the proetale topology. Lemma \ref{genbchomcommutes_LEM} implies $$\text{Hom}_{R/\idm^n, \text{cl}}(K\otimes R/\idm^n, L\otimes R/\idm^n)\otimes R/\idm = \text{Hom}_{R/\idm, \text{cl}}(K\otimes R/\idm, L\otimes R/\idm)$$ and that the formation of 
\[
\text{Hom}_{R/\idm^n, \text{cl}}(K\otimes R/\idm^n, L\otimes R/\idm^n)
\]
commutes with base change for all $n \geq 1$.

By \cite[Ch.~7,~Cor.~2.9]{DeligneSGA45} and Remark \ref{genbcqcqs_REM}, we may pass to a further stratification to assume the complex $\text{Hom}_{R/\idm, \text{cl}}(K\otimes R/\idm, L\otimes R/\idm)$ is constructible. This implies $\text{Hom}_{R/\idm^n, \text{cl}}(K\otimes R/\idm^n, L\otimes R/\idm^n)$ is constructible for all $n \geq 1$ by \cite[Lem.~6.3.12]{BhattSPE}. In particular, these complexes are bounded for all $n \geq 1$. 

Because these complexes are bounded and satisfy the above compatibility, \cite[Lem.~3.5.5]{BhattSPE} implies the isomorphism
\[
\text{Hom}(K, L)\otimes R/\idm = \text{Hom}_{R/\idm}(K\otimes R/\idm, L\otimes R/\idm). 
\]
This implies that the complex on the left is classical and constructible. Since $\text{Hom}(K, L)$ is $\idm$-derived complete, it is a constructible complex. In particular, we obtain (1).

We now verify that the formation of the above complex commutes with base-change from the assumptions we imposed during the proof of (1). Let $g\colon Y' \rightarrow Y$ be a morphism of qcqs schemes. The reference \cite[Lem.~3.5.5]{BhattSPE} provides the isomorphism
\begin{align*}
	g^*_{\text{}}(\text{Hom}(K, L)) = \varprojlim g^*(\text{Hom}_{R/\idm^n, \text{cl}}(K\otimes R/\idm^n, L\otimes R/\idm^n)).
\end{align*}
The formation of the object on the right commutes with base-change, so we obtain
\[
g^*_{\text{}}(\text{Hom}(K, L))  = \varprojlim \text{Hom}_{R/\idm^n, \text{cl}}(g^*(K)\otimes R/\idm^n, g^*(L)\otimes R/\idm^n).
\]
This is the construction above applied to $g^*(K)$ and $g^*(L)$. Note that the assumptions imposed during the proof of (1) are invariant under base change. Therefore, we may apply the same argument as above to show that this is the internal mapping object of $g^*(K)$ and $g^*(L)$ in the category of constructible complexes.
\end{proof}
	\printbibliography		

@Book{KatzESDE,
  author    = {Katz, Nicholas M.},
  publisher = {Princeton University Press, Princeton, NJ},
  title     = {Exponential sums and differential equations},
  year      = {1990},
  series    = {Annals of Mathematics Studies},
  volume    = {124},
}

@Article{DeligneWeilII,
  author   = {Deligne, Pierre},
  journal  = {Inst. Hautes \'{E}tudes Sci. Publ. Math.},
  title    = {La conjecture de {W}eil. {II}},
  year     = {1980},
  pages    = {137--252},
  fjournal = {Institut des Hautes \'{E}tudes Scientifiques. Publications Math\'{e}matiques},
}

@Book{DeligneSGA45,
  author    = {Deligne, P.},
  publisher = {Springer-Verlag, Berlin},
  title     = {Cohomologie \'{e}tale},
  year      = {1977},
  note      = {S\'{e}minaire de g\'{e}om\'{e}trie alg\'{e}brique du Bois-Marie SGA $4\frac{1}{2}$},
  series    = {Lecture Notes in Mathematics},
  volume    = {569},
}

@InCollection{BBD,
  author    = {Beilinson, Alexander and Bernstein, Joseph and Deligne, Pierre and Gabber, Ofer},
  booktitle = {Ast\'erisque},
  title     = {Faisceaux pervers},
  year      = {2018},
  number    = {100},
}

@Book{LeiFu,
  author    = {Fu, Lei},
  publisher = {World Scientific Publishing Co. Pte. Ltd., Hackensack, NJ},
  title     = {Etale cohomology theory},
  year      = {2015},
  series    = {Nankai Tracts in Mathematics},
  volume    = {14},
}

@Article{hansen2023relativeperversity,
  author   = {Hansen, David and Scholze, Peter},
  journal  = {Commun. Am. Math. Soc.},
  title    = {Relative perversity},
  year     = {2023},
  pages    = {631--668},
  volume   = {3},
  fjournal = {Communications of the American Mathematical Society},
}

@Article{LafforguePurity,
  author   = {Lafforgue, Laurent},
  journal  = {Invent. Math.},
  title    = {Chtoucas de {D}rinfeld et correspondance de {L}anglands},
  year     = {2002},
  number   = {1},
  pages    = {1--241},
  volume   = {147},
  fjournal = {Inventiones Mathematicae},
}

@Misc{stacks-project,
  author       = {The {Stacks Project Authors}},
  howpublished = {\url{https://stacks.math.columbia.edu}},
  title        = {\textit{Stacks Project}},
  year         = {2018},
  shorthand    = {Stacks},
}

@Article{Lu2022,
  author  = {Lu, Qing and Zheng, Weizhe},
  journal = {Forum of Mathematics, Sigma},
  title   = {Categorical traces and a relative Lefschetz–Verdier formula},
  year    = {2022},
  volume  = {10},
}

@Article{DrinfeldKashiwara,
  author  = {Drinfeld, Vladimir},
  journal = {Mathematical Research Letters},
  title   = {On a conjecture of {K}ashiwara},
  year    = {2001},
  number  = {5-6},
  pages   = {713--728},
  volume  = {8},
}

@Article{BhattSPE,
  author   = {Bhatt, Bhargav and Scholze, Peter},
  journal  = {Ast\'erisque},
  title    = {The pro-\'etale topology for schemes},
  year     = {2015},
  number   = {369},
  pages    = {99--201},
  fjournal = {Ast\'erisque},
}

@Article{orgogozoFundamentalGroup,
  author   = {Orgogozo, Fabrice},
  journal  = {Bull. Soc. Math. France},
  title    = {Alt\'erations et groupe fondamental premier \`a{} {$p$}},
  year     = {2003},
  number   = {1},
  pages    = {123--147},
  volume   = {131},
  fjournal = {Bulletin de la Soci\'et\'e{} Math\'ematique de France},
}

@Article{conradgrpcohomology,
  author  = {Brian Conrad},
  journal = {Preprint available at \url{https://math.stanford.edu/~conrad/papers/locchar.pdf}},
  title   = {Lifting global representations with local properties},
  year    = {2011},
}

@Book{SGA1,
  author    = {Grothendieck, Alexander},
  publisher = {Institut des Hautes \'Etudes Scientifiques, Paris},
  title     = {Rev\^etements \'etales et groupe fondamental},
  year      = {1963},
  note      = {3\`eme \'edition, S\'eminaire de G\'eom\'etrie Alg\'ebrique},
}

@Article{esnault2025semistablelefschetzpencils,
  author  = {Esnault, H\'el\`ene and Kerz, Moritz},
  journal = {Journal f\"ur die Reine und Angewandte Mathematik. [Crelle's Journal]},
  title   = {Semistable {L}efschetz pencils},
  year    = {2026},
  pages   = {233--277},
  volume  = {835},
}

@InCollection{beilinsonbernsteinjantzen,
  author    = {Beilinson, A. and Bernstein, J.},
  booktitle = {I. M. Gelfand Seminar},
  title     = {A proof of {Jantzen} conjectures},
  year      = {1993},
  pages     = {1-50},
  volume    = {16, Part 1},
}

@Misc{milne2024descentalgebraicschemes,
  author        = {James S Milne},
  title         = {Descent for Algebraic Schemes},
  year          = {2024},
  archiveprefix = {arXiv},
  eprint        = {2406.05550},
}

@Book{SGA4,
  author    = {Grothendieck, A and Artin, M and Verdier, J. L. and Deligne, P. and Saint-Donat, B.},
  publisher = {Springer-Verlag, Berlin},
  title     = {Theorie des topos et cohomologie etale des schemas. {T}ome 1: {T}heorie des topos},
  year      = {1972},
  series    = {Lecture Notes in Mathematics},
  volume    = {269},
}

@Article{zockmasterthesis,
  author  = {Zock, H.},
  journal = {Master thesis, Universität Regensburg},
  title   = {A local version of Kashiwara’s conjecture},
  year    = {2024},
}

@Book{serreabelianladic,
  author    = {Serre, Jean-Pierre},
  publisher = {W. A. Benjamin, Inc., New York-Amsterdam},
  title     = {Abelian {$l$}-adic representations and elliptic curves},
  year      = {1968},
}

@Article{petrov23,
  author  = {Petrov, Alexander},
  journal = {Duke Mathematical Journal},
  title   = {Geometrically irreducible p-adic local systems are de Rham up to a twist},
  year    = {2023},
  number  = {5},
  volume  = {172},
}

@Misc{esnault2026nonabelianversiondelignesfixed,
  author        = {Hélène Esnault and Moritz Kerz},
  title         = {A non-abelian version of Deligne's Fixed Part Theorem},
  year          = {2026},
  archiveprefix = {arXiv},
  eprint        = {2408.13910},
}

@Article{EGA4,
  author  = {Grothendieck, A.},
  journal = {Inst. Hautes \'Etudes Sci. Publ. Math.},
  title   = {\'El\'ements de g\'eom\'etrie alg\'ebrique. {IV}. \'Etude locale des sch\'emas et des morphismes de sch\'emas {IV}},
  year    = {1967},
}

@Article{hemo2023constructiblesheavesschemes,
  author  = {Hemo, Tamir and Richarz, Timo and Scholbach, Jakob},
  journal = {Adv. Math.},
  title   = {Constructible sheaves on schemes},
  year    = {2023},
  pages   = {Paper No. 109179, pp. 46},
  volume  = {429},
}

@Article{zhensun16paritysymmetry,
  author  = {Shenghao Sun and Weizhe Zheng},
  journal = {Algebra {\&} Number Theory},
  title   = {{Parity and symmetry in intersection and ordinary cohomology}},
  year    = {2016},
  number  = {2},
  pages   = {235 -- 307},
  volume  = {10},
}

@InBook{kashiwara98conjectures,
  author    = {Kashiwara, Masaki},
  pages     = {267--271},
  publisher = {Birkh{\"a}user Boston},
  title     = {Semisimple Holonomic D-Modules},
  year      = {1998},
  booktitle = {Topological Field Theory, Primitive Forms and Related Topics},
}

@Misc{hansen2025arithmeticpropertieselladicetale,
  author        = {David Hansen and Bogdan Zavyalov},
  title         = {Arithmetic Properties Of $\ell$-adic \'Etale Cohomology and Nearby Cycles of Rigid-Analytic Spaces},
  year          = {2025},
  archiveprefix = {arXiv},
  eprint        = {2301.01800},
}
\end{document}